\documentclass[a4paper,11pt]{article}
\usepackage{graphicx}

\usepackage{amsmath}
\allowdisplaybreaks

\usepackage{color}
\usepackage{bm}
\usepackage{amsfonts,amssymb}
\usepackage[mathscr]{eucal}

\usepackage{amsthm}	
\usepackage{cite}
\usepackage{hyperref}
\usepackage{cleveref}

\newtheorem{theorem}{Theorem}[section]
\newtheorem{lemma}[theorem]{Lemma}

\newtheorem{corollary}[theorem]{Corollary}

\theoremstyle{definition}
\newtheorem{definition}[theorem]{Definition}
\newtheorem{example}[theorem]{Example}

\theoremstyle{remark}
\newtheorem{remark}[theorem]{Remark}

\def\mt{\mathcal{T}}

\title{On the superconvergence of a hydridizable discontinuous Galerkin method for the Cahn-Hilliard equation}

\title{On the superconvergence of a hydridizable discontinuous Galerkin method for the Cahn-Hilliard equation}

\author{
	{\sc Gang Chen}\footnote{School of Mathematics Sciences, University of Electronic Science and Technology of China, Chengdu, China.
		Email: \emph{cglwdm@uestc.edu.cn}} ,\ \
	{\sc Daozhi Han}\footnote{ Department of Mathematics
		and Statistics, Missouri University of Science and Technology. Email:
		\emph{handaoz@mst.edu}},\ \
	\ {\sc John Singler}\footnote{Department of Mathematics
		and Statistics, Missouri University of Science and Technology. Email: \emph{singlerj@mst.edu}} \\
	and \ {\sc Yangwen Zhang}\footnote{Department of Mathematics Science, University of Delaware
		Newark, DE. Email: \emph{ywzhangf@udel.edu}}
}

\date{}

\begin{document}
\maketitle

\begin{abstract}
	We propose a hydridizable discontinuous Galerkin (HDG) method for solving the Cahn-Hilliard equation. The temporal discretization can be based on either the backward Euler method or the convex-splitting method. We show that the fully discrete scheme admits a unique solution, and we  establish  optimal convergence rates for all variables in the $L^2$ norm for arbitrary polynomial orders. In terms of the globally coupled degrees of freedom,   the scalar variables are superconvergent. Another theoretical contribution of this work is a novel HDG  Sobolev inequality that is useful for HDG error analysis of nonlinear problems. Numerical results are reported to confirm the theoretical convergence rates.

\end{abstract}


\section{Introduction}

Let $\Omega\subset \mathbb{R}^d$ ($d=2,3$) be a polygonal domain with Lipshitz boundary $\partial\Omega$ and $T$ be a positive constant.
We consider the following Cahn-Hilliard equation:
\begin{subequations}\label{ori}
	\begin{align}
	u_t-\Delta\phi&= 0  &\text{in} \; \Omega\times(0,T],\label{o1}\\
	-\epsilon\Delta u+\epsilon^{-1}f(u)&= \phi  &\text{in} \; \Omega\times(0,T],\label{o2}\\
	\nabla u\cdot\bm n=\nabla\phi\cdot\bm n&= 0 &\text{on} \; \partial\Omega\times(0,T],\label{o3}\\
	u(\cdot,0)&= u^0(\cdot)  &\text{in}\; \Omega,\label{o4}
	\end{align}
\end{subequations}
where $f(u)=u^3-u$. The Cahn-Hilliard equation is a fourth order, nonlinear parabolic equation which was originally proposed by Cahn and Hilliard \cite{CaHi1958, Cahn1959, CaHi1959} as a phenomenological model for phase separation and coarsening in a binary alloy. Since then Cahn-Hilliard-type equations have found applications in a variety of fields, including multiphase flow \cite{LoTr1998, AMW1998}, two-phase flow in porous media \cite{GaBr2000}, tumor growth \cite{WLC2011}, pattern formation \cite{ZPN2005}, thin films \cite{BePu1996} and many others. Owing to its importance,  many works have been devoted to the design and analysis of numerical schemes for solving the Cahn-Hilliard equation; see, e.g., finite difference methods \cite{Furihata2001}, mixed and nonconforming finite element methods \cite{ElFr1989, DuNi1991, ElLa1992, BBG1999, FePr2004, DWW2016} and Fourier-spectral methods \cite{ShYa2009, LiQi2017, SXY2018}.

In recent years, the discontinuous Galerkin (DG) method has become popular for solving the Cahn-Hilliard equation, owing to its flexibility in handling higher order derivatives, high-order accuracy, the property of local conservation which is crucial for applications in porous {{red}medium} flow and transport phenomenon, high parallelizability and ease of achieving $hp$-adaptivity. Applications of DG methods to fourth order elliptic problems have been considered by Babu\v{s}ka and Zl\'{a}mal in \cite{BaZl1973}, by Baker in \cite{Baker1977}, and more recently by Mozolevski et al. in a series of works \cite{MoSu2003, MoBo2007, MoSuBo2007a, MoSuBo2007b, SuMo2007}. In \cite{FeKa2007}, Feng and Karakashian design and analyze a DG method of interior penalty type based on the fourth order formulation of the Cahn-Hilliard equation. Optimal error estimates in various energy norms are established: see also \cite{FLX2016}. Kay et al. propose and analyze a different DG method \cite{KSS2009} that treats the Cahn-Hilliard equation as a system of second order equations allowing a relatively larger penalty term. A fully adaptive version of the interior penalty DG method was recently constructed in \cite{AKW2015} for the Cahn-Hilliard equation with a source and optimal $L^2$ error bound were derived; see also \cite{FLAR2018} for solving the advective Cahn-Hilliard equation. The local discontinuous Galerkin (LDG) method has also been proposed for the discretization of the Cahn-Hilliard equation by writing it as a system of four first-order equations. Dong and Shu in \cite{DoSh2009} analyzed an LDG scheme for general elliptic equations including the linearized Cahn-Hilliard equation and obtained optimal error estimate in $L^2$. Recently, an LDG method has been employed for solving a number of Cahn-Hilliard fluid models, cf. \cite{GXX2014, GuXu2014, SoSh2017}.

The DG method is however often criticized for the larger amount of degrees of freedom compared to the continuous Galerkin (CG) method . In the seminal work \cite{Cockburn_Gopalakrishnan_Lazarov_Unify_SINUM_2009} Cockburn et al.  propose a hybridizable discontinuous Galerkin (HDG) method for second order elliptic problems. In a nutshell, the HDG method maps the flux and solution into the numerical trace of the solution via a local solver, which are in turn connected by the continuity of fluxes across inter-element boundaries (a transmission condition). Hence the globally coupled degrees of freedom are those numerical traces, resulting in a significant reduction of the number of unknowns in traditional  DG methods. Moreover, the HDG methods possess the same favorable properties as classical mixed methods. In particular,
HDG methods provide optimal convergence rates for both the gradient and the primal
variables of the mixed formulation. This property enables the construction of superconvergent solutions, contrary to other
DG methods. These advantages of the HDG methods have made HDG an attractive alternative for solving problems governed by PDEs and PDE control problems, see \cite{Cockburn_Shi_Stokes_MathComp_2013,Cockburn_Gopalakrishnan_Nguyen_Peraire_Sayas_Stokes_MathComp_2011,Cockburn_Shi_Stokes_CF_2014,Cockburn_Says_Divergence_Free_MathComp_2014,Cesmelioglu_Cockburn_Nguyen_Peraire_Oseen_JSC_2013,Rhebergen_Cockburn_NS_JCP_2013,Rhebergen_Cockburn_Deforming_JCP_2012,Cockburn_Nguyen_Peraire_Hyperbolic_Book_2016,Sanchez_Ciuca_Nguyen_Peraire_Cockburn_Hamiltonian_JCP_2017,Stanglmeier_Nguyen_Peraire_Wave_CMAME_2016,ChenHuShenSinglerZhangZheng_HDG_Convection_Dirtributed_Control_JCAM_2018,HuShenSinglerZhangZheng_HDG_Dirichlet_control3,HuMateosSinglerZhangZhang2}.

Most  study currently focuses on establishing 
optimal and superconvergent rates of HDG methods for  second order problems, such as elliptic PDEs \cite{Cockburn_Gopalakrishnan_Sayas_Porjection_MathComp_2010}, convection diffusion equations \cite{Chen_Cockburn_Convection_Diffusion_IMAJNA_2012,Chen_Cockburn_Convection_Diffusion_MathComp_2014,Qiu_Shi_Convection_Diffusion_JSC_2016}, Stokes equations \cite{Cockburn_Shi_Stokes_CF_2014,Cockburn_Gopalakrishnan_Nguyen_Peraire_Sayas_Stokes_MathComp_2011}, Oseen equations \cite{Cesmelioglu_Cockburn_Nguyen_Peraire_Oseen_JSC_2013} and Navier-Stokes equations \cite{Qiu_Shi_NS_IMAJNA_2016,Cesmelioglu_NS_MathComp_2017}.  However, in \cite{Cockburn_Dong_Guzman}, the authors utilized an HDG method with polynomial degree $k$ for all variables to investigate the biharmonic equations and obtained an optimal convergence rate for the solution and suboptimal convergence rates for the other variables. To the best of our knowledge, there does not exist an HDG work that achieves optimal convergence rates for all variables  for a  fourth order problem.

In this work, we propose a HDG method for the Cahn-Hilliard equation with  Lehrenfeld-Sch\"oberl stabilization function, polynomials of degree $k+1$ for the scalar unknown, and polynomials of degree $k$ for the other unknowns. The HDG framework with reduced stabilization and polynomials of mixed orders was first introduced by Lehrenfeld in \cite{Lehrenfeld_PhD_thesis_2010} where it was alluded that the scheme could be a superconvergent method, i.e.,  $O(h^{k+2})$ error estimates is expected  for the solution variables even though polynomials of  degree $k$ are used for the globally coupled unknowns (numerical traces of the solution). Optimal convergence and hence superconvergence was then rigorously established for convection diffusion problems \cite{Qiu_Shi_Convection_Diffusion_JSC_2016}, Navier-Stokes equations \cite{Qiu_Shi_NS_IMAJNA_2016}, and more recently for linear elasticity problems \cite{Qiu_Shen_Shi_Elasticity_MathComp_2018}. We provide the HDG formulation for the Cahn-Hilliard equation in \Cref{FDCH} and prove the existence, uniqueness and stability of the HDG method in \Cref{E_U_S}. In \Cref{error_analysis}, we perform a rigorous error analysis for the HDG method and obtain the following a priori error bounds for the solution $\phi$, $u$ and their fluxes $\bm p = -\nabla\phi$ and $\bm q = -\nabla u$:
\begin{align*}
\Delta t \sum_{n=1}^N \|\bm p^n -\bm p_h^n\|_{L^2(\Omega)}^2 +\Delta t \sum_{n=1}^N \|\bm q^n -\bm q_h^n\|_{L^2(\Omega)}^2\le C (h^{k+1}+\Delta t)^2, \\
\Delta t \sum_{n=1}^N \|\phi^n -\phi_h^n\|_{L^2(\Omega)}^2 +\max_{1\le n\le N} \|u^n-u_h^n\|_{L^2(\Omega)}^2 \le C (h^{k+2}+\Delta t)^2.
\end{align*}
These convergence rates are further validated by numerical experiments  in \Cref{numerics}. A particular theoretical contribution of this article is the establishment of  a {\em novel} HDG Sobolev inequality (cf. \Cref{discrete-soblev}) which is a useful tool in the numerical analysis of nonlinear problems.

\section{The HDG formulation}\label{FDCH}

To introduce the fully discrete HDG formulation for the Cahn-Hilliard equation, we first fix some notation.
Let $\mathcal{T}_h$ be a shape-regular, quasi-uniform triangulation of $\Omega$. Let $\mathcal{E}_h$ denote the set of all faces $E$ of all simplexes $K$ of the triangulation $\mathcal{T}_h$.  Also let $\mathcal{E}_h^o$ and $\mathcal{E}_h^{\partial}$ denote the set of interior faces and boundary faces, respectively.
 Furthermore, we  introduce the discrete inner products
\begin{align*}
(w,v)_{\mathcal{T}_h} &:= \sum_{K\in\mathcal{T}_h} (w,v)_K  = \sum_{K\in\mathcal{T}_h}\int_K w v ,   \\
 \left\langle \zeta,\rho\right\rangle_{\partial\mathcal{T}_h} &:= \sum_{K\in\mathcal{T}_h} \left\langle \zeta,\rho\right\rangle_{\partial K} = \sum_{K\in\mathcal{T}_h}\int_{\partial K} \zeta\rho.
\end{align*}

For any integer $k\ge 0$, let $\mathcal{P}^k(K)$ denote the set of polynomials of degree at most $k$ on the element $K$.  We introduce the following  discontinuous finite element spaces:
\begin{align*}
\bm V_h&:=\{\bm v_h\in [L^2(\Omega)]^d:\bm v_h|_K\in [\mathcal{P}^{k}(K)]^{d},\forall K\in\mathcal{T}_h\},\\
W_h&:=\{w_h\in L^2(\Omega):w_h|_K\in \mathcal{P}^{k+1}(K),\forall K\in\mathcal{T}_h\},\\
\mathring{W}_h&:=\{w_h\in L^2_0(\Omega):w_h|_K\in \mathcal{P}^{k+1}(K),\forall K\in\mathcal{T}_h\},\\
M_h&:=\{\mu_h\in L^2(\mathcal{E}_h):\mu_h|_E\in \mathcal{P}^{k}(E),\forall E\in\mathcal{E}_h\},
\end{align*}
where $L^2_0(\Omega)$ is the subspace of $L^2(\Omega)$ of mean zero functions.

Since the HDG method is based on a mixed formulation,  we rewrite the Cahn-Hilliard equation as a first order system by setting $\bm p+\nabla \phi=0$ and $\bm q+\nabla u=0$ in \eqref{ori}. The mixed formulation of \eqref{ori} then reads
\begin{subequations}\label{mixed}
	\begin{align}
	\bm p+\nabla \phi&=0  &\text{in} \; \Omega\times(0,T],\\
	u_t+\nabla\cdot\bm p&= 0  &\text{in} \; \Omega\times(0,T],\\
	\bm q+\nabla u&=0  &\text{in} \; \Omega\times(0,T],\\
	\epsilon\nabla\cdot\bm q+\epsilon^{-1}f(u)&= \phi  &\text{in} \; \Omega\times(0,T],\\
	\bm p\cdot\bm n=\bm q\cdot\bm n&= 0 &\text{on} \; \partial\Omega\times(0,T],\\
	u(\cdot,0)&= u^0(\cdot) &\text{in}\; \Omega.
	\end{align}
\end{subequations}

Now we introduce the fully discrete HDG formulation of the Cahn-Hilliard equation  based on backward Euler method and convex-splitting approach.
\begin{subequations}\label{HDG_no_compact}
For a fixed integer $N$, let $0=t_0<t_1<\cdots<t_N=T$ be a uniform partition of $[0, T]$ with $ \Delta t = T/N$. Based on the mixed form \eqref{mixed}, the HDG method seeks $(\bm p_h^n, \phi_h^n, \widehat{\phi}_h^n)\in \bm V_h\times W_h\times M_h$ satisfying
\begin{align}
(\bm p_h^n, \bm r_1)_{\mathcal T_h} - (\phi_h^n, \nabla\cdot \bm r_1)_{\mathcal T_h} + \langle \widehat \phi_h^n, \bm r_1\cdot\bm n\rangle_{\partial \mathcal T_h}=0,\\
(\partial_{t}^+u_h^n, w_1)_{\mathcal{T}_h}-(\bm p_h^n, \nabla w_1)_{\mathcal T_h} + \langle \widehat {\bm p}_h^n\cdot \bm n, w_1\rangle_{\partial \mathcal T_h}=0,\\
\langle \widehat{\bm p}_h^n\cdot \bm n, \mu_1\rangle_{\partial \mathcal T_h}=0,
\end{align}
for all $(\bm r_1, w_1,\mu_1)\in \bm V_h\times W_h\times M_h$; and $(\bm q_h^n, u_h^n, \widehat u_h^n)\in \bm V_h\times W_h\times M_h$ such that
\begin{align}
(\bm q_h^n, \bm r_2)_{\mathcal T_h} - (u_h^n, \nabla\cdot \bm r_2)_{\mathcal T_h} + \langle \widehat u_h^n, \bm r_2\cdot\bm n\rangle_{\partial \mathcal T_h}=0,\\
-(\epsilon\bm q_h^n, \nabla w_2)_{\mathcal T_h} + \langle \epsilon \widehat {\bm q}_h^n\cdot \bm n, w_2\rangle_{\partial \mathcal T_h}+ (\epsilon^{-1}f^n(u_h^n), w_2)_{\mathcal T_h}=0,\\
\langle \epsilon\widehat{\bm q}_h^n\cdot \bm n, \mu_2\rangle_{\partial \mathcal T_h}=0,
\end{align}
for all $(\bm r_2, w_2,\mu_2)\in \bm V_h\times W_h\times M_h$.
Here, $\partial_{t}^+u_h^n=(u_h^n-u_h^{n-1})/\Delta t$, $f^n(u_h^n)=(u_h^n)^3-u_h^n$ for the fully implicit scheme and 
$f^n(u_h^n)=(u_h^n)^3-u_h^{n-1}$ for the energy-splitting scheme, and the numerical fluxes  on $\partial \mathcal T_h$ are defined as
\begin{align}
\widehat{\bm p}_h^n\cdot \bm n = \bm p_h^n \cdot \bm n + h_K^{-1} (\Pi_k^\partial \phi_h^n - \widehat \phi_h^n), \\
\widehat{\bm q}_h^n\cdot \bm n = \bm q_h^n \cdot \bm n + h_K^{-1} (\Pi_k^\partial u_h^n - \widehat u_h^n),
\end{align}
where $\Pi_k^\partial $ is the element-wise $L^2$ projection onto $\mathcal{P}^k (E)$ such that 
\begin{align*}
\langle\Pi_k^\partial u_h, \mu_h\rangle_{E}=\langle u_h, \mu_h\rangle_{E}, \quad \forall \mu_h\in \mathcal{P}^k (E) \ \textup{and} \ E\in \partial K.
\end{align*}
\end{subequations}
We shall also make use of the standard $L^2$ projection, denoted by $P_M$,  onto $M_h$. Note that $\Pi_k^\partial$ coincides with $P_M$ on the space $H^1(\Omega)$.

To make the expressions concise, we introduce the operator  $\mathcal A: [\bm V_h\times W_h\times M_h]^2 \to \mathbb R$ by
\begin{align}\label{def_A}
\begin{split}
\hspace{1em}&\hspace{-1em}\mathcal A(\bm q_h,u_h,\widehat u_h;\bm r_h,w_h,\mu_h)\\
&=(\bm q_h,\bm r_h)_{\mathcal{T}_h} - (u_h,\nabla\cdot\bm r_h)_{\mathcal{T}_h}
+\langle \widehat u_h,\bm r_h\cdot\bm n \rangle_{\partial\mathcal{T}_h}\\
&\quad +(\nabla\cdot\bm q_h, w_h)_{\mathcal{T}_h} -\langle \bm q_h\cdot\bm n, \mu_h \rangle_{\partial\mathcal{T}_h}\\
&\quad +\langle  h_K^{-1}({\Pi}_k^{\partial}u_h-\widehat u_h), {\Pi}_k^{\partial}w_h-\mu_h \rangle_{\partial\mathcal{T}_h},
\end{split}
\end{align}
for all $(\bm q_h,u_h,\widehat u_h), (\bm r_h,w_h,\mu_h)\in \bm V_h\times W_h\times M_h$.

Then the HDG formulation \eqref{HDG_no_compact} can be recasted as: for $n=1,2,\cdots, N$, find $(\bm p_h^n,\phi_h^n,\widehat{\phi}_h^n)$, $(\bm q_h^n, u_h^n,\widehat u_h^n)\in \bm V_h\times W_h\times M_h$ such that 
\begin{subequations}\label{HDG-hill}
		\begin{align}
	    (\partial_{t}^+ u_h^n,w_1)_{\mathcal{T}_h}+\mathcal A(\bm p_h^n,\phi_h^n,\widehat{\phi}_h^n;\bm r_1,w_1,\mu_1)&=0,\label{hdg01}\\
	    (\epsilon^{-1}f^n(u^n_h),w_2)_{\mathcal{T}_h}+\epsilon\mathcal{A}(\bm q_h^n,u_h^n,\widehat{u}_h^n;\bm r_2, w_2,\mu_2)-(\phi_h^n,w_2)_{\mathcal{T}_h}&=0\label{hdg02},\\
	    (u_h^0,w_3)_{\mathcal{T}_h}-(u^0,w_3)_{\mathcal{T}_h}&=0, \label{hdg03}
	\end{align}
	 for all $(\bm r_1,w_1,\mu_1), (\bm r_2,w_2,\mu_2) \in \bm V_h\times W_h\times M_h$ and $w_3\in W_h$.
\end{subequations}

\section{Preliminaries}
Throughout, $C$ shall denote a generic constant independent of the mesh parameters $h, \Delta t$.
We first recall the standard $L^2$ projections $\bm\Pi_{k}^o :  [L^2(\Omega)]^d \to \bm V_h$ and $\Pi_{k+1}^o :  L^2(\Omega) \to W_h$
\begin{equation}\label{L2_projection}
\begin{split}
(\bm\Pi_k^o \bm q,\bm r_h)_{K} &= (\bm q,\bm r_h)_{K} ,\qquad \forall \bm r_h\in [{\mathcal P}_{k}(K)]^d,\\
(\Pi_{k+1}^o u,w_h)_{K}  &= (u,w_h)_{K} ,\qquad \forall w_h\in \mathcal P_{k+1}(K),
\end{split}
\end{equation}
which obey  the following classical error estimates (see for instance \cite[Lemma 3.3]{ChenSinglerZhang1}):
\begin{subequations}\label{classical_ine}
	\begin{align}
	\|{\bm q -\bm\Pi_k^o \bm q}\|_{\mathcal T_h} \le  C h^{k+1} \|{\bm q}\|_{H^{k+1}}, \ \| {u -{\Pi_{k+1}^o u}}\|_{\mathcal T_h} \le  C h^{k+2} \|{u}\|_{H^{k+2}},\label{eq27a}\\
	\| {u - {\Pi_{k+1}^o u}}\|_{\partial\mathcal T_h} \le  C h^{k+\frac 3 2} \|{u}\|_{H^{k+2}},
	 \  \| {w}\|_{\partial \mathcal T_h} \le  C h^{-\frac 12} \| {w}\|_{ \mathcal T_h}, \: \forall w\in V_h,\label{eq27b}\\
	\| {u - {\Pi_{k+1}^o u}}\|_{L^\infty} \le  C h^{k+2-d/2} |{u}|_{H^{k+2}}.\label{est-infty}
	\end{align}
\end{subequations}
The same error bounds hold true for the projections of  $\bm p$ and $\phi$.

We shall also utilize the following version of the piecewise Poincar\'{e}-Friedrichs inequality , cf. \cite{Brenner_Poincare_SINUM_2003}.
\begin{lemma}\label{Poincare0}
	Let $v$ be a piecewise $H^1$ function with respect to the partition $\mathcal{T}_h$. The following Poincar\'{e}-Friedrichs inequality holds
	\begin{align*}
	\|v\|_{\mathcal{T}_h}^2\le 	C\left(\|\nabla v\|^2_{\mathcal{T}_h}+ |( v,1)_{\mathcal{T}_h}|^2
	+\sum_{E\in\mathcal{E}^o_h}|E|^{d/(1-d)}\Big(\int_E [\![v]\!] ds\Big)^2
	\right),
	\end{align*}
	where the generic constant $C$ depends only on the regularity of the partition, and $[\![v]\!]$ denotes the jump of $v$ across a side $E$.
\end{lemma}
The following HDG Poincar\'{e} inequality is then an immediate consequence of Lemma \ref{Poincare0}, the Cauchy-Schwarz inequality and the triangle inequality.
\begin{lemma}[The HDG Poincar\'{e} inequality] If $(v_h, \widehat v_h) \in W_h \times M_h$, we have
	\begin{align}\label{HDG-poincare}
	\|v_h\|^2_{\mathcal{T}_h}\le C
	\left(
	\|\nabla v_h\|^2_{\mathcal{T}_h}+|( v_h,1)_{\mathcal{T}_h}|^2+
	\|h_K^{-1/2}(\Pi_k^{\partial}v_h-\widehat v_h)\|^2_{\partial\mathcal{T}_h}
	\right).
	\end{align}
\end{lemma}

Now, we glean some basic properties of the operator $\mathcal A$. First, the definition of $\mathcal A$ in Eq. \eqref{def_A} immediately implies  lemmas  \ref{energy_of_A} and \ref{summertic_A}.
\begin{lemma}\label{energy_of_A}
	For any $(\bm r_h, w_h, \mu_h)\in \bm V_h\times W_h\times M_h$, we have 
	\begin{align}
	\mathcal A(\bm r_h, w_h, \mu_h,\bm r_h, w_h, \mu_h) = \|\bm r_h\|_{\mathcal T_h}^2 + \|h_{K}^{-1/2}(\Pi_k^{\partial} w_h-\mu_h)\|_{\partial\mathcal  T_h}^2. \label{energy_of_A1}
	\end{align}
\end{lemma}

\begin{lemma}\label{summertic_A}
	For all $(\bm q_h, u_h,\widehat u_h), (\bm p_h, \phi_h,\widehat \phi_h) \in \bm V_h\times W_h\times M_h$, the operator $\mathcal{A}$ has  the following property
	\begin{align}\label{summertic_A1}
	\mathcal{A}(\bm q_h, u_h,\widehat u_h;\bm p_h,-\phi_h,-\widehat \phi_h)=
	\mathcal{A}(\bm p_h,\phi_h,\widehat \phi_h;\bm q_h,-u_h,-\widehat u_h).
	\end{align}
\end{lemma}
Next, we show that the operator $\mathcal{A}$ satisfies the following bound.
\begin{lemma}\label{A_bound}
	For all $(\bm q_h, u_h,\widehat u_h), (\bm p_h, \phi_h,\widehat \phi_h) \in \bm V_h\times W_h\times M_h$, we have
	\begin{align}\label{A_bound1}
	&\left|\mathcal{A}(\bm q_h,u_h,\widehat u_h;\bm p_h,\phi_h,\widehat \phi_h )\right|
	 \le C\left(\|\bm q_h\|_{\mathcal{T}_h}+\|\nabla u_h\|_{\mathcal{T}_h}
	+\|h_K^{-1/2}( \Pi_k^{\partial}u_h-\widehat u_h)\|_{\partial\mathcal{T}_h}\right) \nonumber\\
	& \times \left(\|\bm p_h\|_{\mathcal{T}_h}+\|\nabla \phi_h\|_{\mathcal{T}_h}+\|h_K^{-1/2}( \Pi_k^{\partial}\phi_h-\widehat \phi_h)\|_{\partial\mathcal{T}_h}\right).
	\end{align}
\end{lemma}
\begin{proof} By the definition of $\mathcal A$ in \eqref{def_A} and integration by parts, one gets
	\begin{align*}
	\left|\mathcal{A}(\bm q_h,u_h,\widehat u_h;\bm p_h,\phi_h,\widehat \phi_h)\right|
	& \le \left| (\bm q_h,\bm p_h)_{\mathcal{T}_h} + (\nabla u_h,\bm p_h)_{\mathcal{T}_h}
	-\langle \Pi_k^{\partial}u_h-\widehat u_h, \bm p_h\cdot\bm n\rangle_{\partial\mathcal{T}_h}\right|\\
	& \quad +\left|-(\bm q_h,\nabla \phi_h)+\langle\bm q_h\cdot\bm n,\Pi_k^{\partial}\phi_h -\widehat \phi_h
	\rangle_{\partial\mathcal{T}_h}\right|\\
	&\quad +\left|\langle h_K^{-1}(\Pi_k^{\partial}u_h-\widehat u_h), \Pi_k^{\partial}\phi_h-\widehat \phi_h \rangle_{\partial\mathcal{T}_h}\right|.
	\end{align*}
	Then the bound \eqref{A_bound1} follows from the Cauchy-Schwarz inequality and the  inverse inequality \eqref{eq27b}. This completes the proof.
\end{proof}
Furthermore, we establish a crucial lemma that bounds the gradient of the scalar variable in terms of the flux variable and the reduced stabilization.
\begin{lemma}\label{es_u}
	If $(\bm q_h,u_h,\widehat u_h)\in \bm V_h\times W_h\times M_h$ satisfies
	\begin{align}\label{Ap}
	\mathcal A(\bm q_h,u_h,\widehat u_h; \bm r_h,0,0)=0, \quad \forall \bm r_h\in \bm V_h, 
	\end{align}
	 then the following inequality holds 
	\begin{align}\label{es_u1}
	\|\nabla u_h\|_{\mathcal{T}_h}+\|h_K^{-1/2}(u_h-\widehat u_h)\|_{\partial\mathcal{T}_h}\le C\left(\|\bm q_h\|_{\mathcal{T}_h}+\|h_K^{-1/2}(\Pi_k^{\partial}u_h-\widehat u_h)\|_{\partial\mathcal{T}_h}\right).
	\end{align}
\end{lemma}

\begin{proof} 
	By the definition of $\mathcal A$ in \eqref{def_A}, let $\bm r_h = \nabla u_h$ in \eqref{Ap} and perform integration by parts to get
	\begin{align*}
	(\bm q_h,\nabla u_h)_{\mathcal{T}_h} +(\nabla u_h,\nabla u_h)_{\mathcal{T}_h} +\langle
	\widehat u_h-u_h,\nabla u_h \cdot \bm n
	\rangle_{\partial\mathcal{T}_h}=0.
	\end{align*}
	Note that 
	$\langle\Pi_k^\partial u_h-u_h, \nabla u_h\cdot \bm n\rangle_{\partial \mt_h}=0$.	
	It follows from the element-wise Cauchy-Schwarz inequality and the inverse inequality \eqref{eq27b}  that
	\begin{align*}
	\|\nabla u_h\|_{\mathcal{T}_h} \le C \left(\|\bm q_h\|_{\mathcal{T}_h} +\|h_K^{-1/2}(\Pi_k^{\partial}u_h-\widehat u_h)\|_{\partial\mathcal{T}_h}\right).
	\end{align*}
	The triangle inequality gives 
	\begin{align*}
	\|h_K^{-1/2}(u_h-\widehat u_h)\|_{\partial\mathcal{T}_h}&\le
	\|h_K^{-1/2}(\Pi_k^{\partial}u_h-\widehat u_h)\|_{\partial\mathcal{T}_h}
	+\|h_K^{-1/2}(u_h-\Pi_k^{\partial}u_h)\|_{\partial\mathcal{T}_h}\\
	&\le C \left(\|h_K^{-1/2}(\Pi_k^{\partial}u_h-\widehat u_h)\|_{\partial\mathcal{T}_h}
	+\|\nabla u_h\|_{\mathcal{T}_h}	\right).
	\end{align*}
	The desired inequality \eqref{es_u1} now follows by combining the last two inequalities. This completes the proof. 
\end{proof}
Finally, we show that the operator $\mathcal A$ satisfies a version of the discrete LBB condition.
\begin{lemma}[Discrete LBB Condition of $\mathcal A$]\label{LBB} 
	For all $(\bm q_h,u_h,\widehat u_h)\in \bm V_h\times \mathring{W}_h\times M_h$, we have 
	\begin{subequations}
		\begin{align}
		&\sup_{\bm 0\neq (\bm p_h,\phi_h,\widehat \phi_h)\in \bm V_h\times \mathring{W}_h\times M_h}\frac{\mathcal{A}(\bm q_h,u_h,\widehat u_h;\bm p_h,\phi_h,\widehat \phi_h)}{\|\bm p_h\|_{\mathcal{T}_h}+\|\nabla \phi_h\|_{\mathcal{T}_h}+\|h_K^{-1/2}( \Pi_k^{\partial}\phi_h-\widehat \phi_h)\|_{\partial\mathcal{T}_h}}\nonumber\\
		&\qquad\ge C\left(\|\bm q_h\|_{\mathcal{T}_h}+\|\nabla u_h\|_{\mathcal{T}_h}+
		\|h_K^{-1/2}( \Pi_k^{\partial}u_h-\widehat u_h)\|_{\partial\mathcal{T}_h}\right),\label{LBBa}\\
		&
		\sup_{\bm 0\neq (\bm p_h,\phi_h,\widehat \phi_h)\in \bm V_h\times \mathring{W}_h\times M_h}\frac{\mathcal{A}(\bm p_h,\phi_h,\widehat \phi_h;\bm q_h,u_h,\widehat u_h)}{\|\bm p_h\|_{\mathcal{T}_h}+\|\nabla \phi_h\|_{\mathcal{T}_h}+\|h_K^{-1/2}( \Pi_k^{\partial}\phi_h-\widehat \phi_h)\|_{\partial\mathcal{T}_h}}\nonumber\\
		&\qquad\ge C\left(\|\bm q_h\|_{\mathcal{T}_h}+\|\nabla u_h\|_{\mathcal{T}_h}+
	\|h_K^{-1/2}( \Pi_k^{\partial}u_h-\widehat u_h)\|_{\partial\mathcal{T}_h}\right).\label{LBBb}
		\end{align}
	\end{subequations}
\end{lemma}
\begin{proof} 
	We only give the details of the proof of the inequality \eqref{LBBa}, since the argument for \eqref{LBBb} is similar.  
	
	First, we note that $\|\bm p_h\|_{\mathcal{T}_h}+\|\nabla \phi_h\|_{\mathcal{T}_h}+\|h_K^{-1/2}( \Pi_k^{\partial}\phi_h-\widehat \phi_h)\|_{\partial\mathcal{T}_h}$ defines a norm in the product space $\bm V_h\times \mathring{W}_h\times M_h$, thanks to the HDG Poincar\'{e}  inequality \eqref{HDG-poincare}.  Let $\alpha$ be a positive number to be specified later. For any fixed $(\bm q_h,u_h,\widehat u_h)\in \bm V_h\times \mathring{W}_h\times M_h$, we take $(\bm p_h,\phi_h,\widehat \phi_h)=(\bm q_h+\alpha\nabla u_h,u_h,\widehat u_h)\in \bm V_h\times \mathring{W}_h\times M_h$ to get
	\begin{align*}
	\hspace{1em}&\hspace{-1em}\mathcal A(\bm q_h,u_h,\widehat u_h;\bm p_h, \phi_h, \widehat \phi_h)\\
	&=(\bm q_h,\bm q_h+\alpha\nabla u_h)_{\mathcal{T}_h} - ( u_h,\nabla\cdot(\bm q_h+\alpha\nabla u_h))_{\mathcal{T}_h}\\
	&\quad + \langle \widehat u_h,(\bm q_h+\alpha\nabla u_h)\cdot\bm n
	\rangle_{\partial\mathcal{T}_h} + (\nabla\cdot\bm q_h,u_h)_{\mathcal{T}_h}
	-\langle \bm q_h\cdot\bm n,\widehat u_h \rangle_{\partial\mathcal{T}_h}\\
	&\quad +\langle 
	h_K^{-1}({\Pi}_k^{\partial}u_h-\widehat u_h),{\Pi}_k^{\partial}u_h-\widehat u_h
	\rangle_{\partial\mathcal{T}_h}\\
	&=\|\bm q_h\|^2_{\mathcal{T}_h}+\alpha\|\nabla u_h\|^2_{\mathcal{T}_h}
	+\|h_K^{-1/2}(\Pi_k^{\partial}u_h-\widehat u_h)\|^2_{\partial\mathcal{T}_h}\\
	&\quad +\alpha(\bm q_h,\nabla u_h)_{\mathcal{T}_h}+\alpha\langle\widehat u_h-{\Pi}_k^{\partial}u_h,\bm n\cdot \nabla u_h \rangle_{\partial\mathcal{T}_h}\\
	&\ge  (1-C\alpha)\left(\|\bm q_h\|^2_{\mathcal{T}_h} +\|h_K^{-1/2}(\Pi_k^{\partial}u_h-\widehat u_h)\|^2_{\partial\mathcal{T}_h}
	\right)+\frac{\alpha}{2}\|\nabla u_h\|^2_{\mathcal{T}_h}.
	\end{align*}
	By choosing  $\alpha$ such that $1-C\alpha >0$,  one obtains
	\begin{align*}
	\mathcal A(\bm q_h,u_h,\widehat u_h;\bm p_h, \phi_h, \widehat \phi_h) \ge C_1\left(\|\bm q_h\|^2_{\mathcal{T}_h}
	+\|h_K^{-1/2}(\Pi_k^{\partial}u_h-\widehat u_h)\|^2_{\partial\mathcal{T}_h}+\|\nabla u_h\|^2_{\mathcal{T}_h}\right).
	\end{align*}
	On the other hand, $(\bm p_h,\phi_h,\widehat \phi_h)=(\bm q_h+\alpha\nabla u_h,u_h,\widehat u_h)$ gives
	\begin{align*}
	\hspace{1em}&\hspace{-1em}\|\bm p_h\|^2_{\mathcal{T}_h}+\|h_K^{-1/2}(\Pi_k^{\partial}\phi_h-\widehat \phi_h)\|^2_{\partial\mathcal{T}_h}+\|\nabla \phi_h\|^2_{\mathcal{T}_h}\\
	&\le C_2(\|\bm q_h\|^2_{\mathcal{T}_h}+\|h_K^{-1/2}(\Pi_k^{\partial}u_h-\widehat u_h)\|^2_{\partial\mathcal{T}_h}
	+\|\nabla u_h\|^2_{\mathcal{T}_h}).
	\end{align*}
	Then \eqref{LBBa} follows immediately. This completes the proof.
	
\end{proof}

We now introduce the HDG inversion of the Laplace operator equipped with homogeneous Neumann boundary condition.
\begin{definition}\label{-1h}
For any $u_h\in W_h$,  we define $(\bm{\Pi}_{\bm V}u_h,\Pi_Wu_h,\Pi_{M}u_h)\in \bm V_h\times\mathring{W}_h\times M_h$ such that 
\begin{align}\label{def_-1_h}
\mathcal A(\bm{\Pi}_{\bm V}u_h,\Pi_W u_h,\Pi_{M}u_h;\bm r_h,w_h,\mu_h)=(u_h,w_h)_{\mathcal{T}_h},
\end{align}
for all $(\bm r_h,w_h,\mu_h)\in \bm V_h\times\mathring{W}_h\times M_h$, 
\end{definition}

Thanks to the discrete LBB condition \Cref{LBB}, the inversion \eqref{def_-1_h}  in \Cref{-1h} is well defined.
For all $u_h\in \mathring{W}_h$, we define the semi-norm
\begin{align*}
\|u_h\|^2_{-1,h}:=
\mathcal A(\bm{\Pi}_{\bm V}u_h,\Pi_Wu_h,\Pi_{M}u_h;\bm{\Pi}_{\bm V}u_h,\Pi_Wu_h,\Pi_{M}u_h).
\end{align*}
Then for all $u_h\in \mathring{W}_h$, by \Cref{energy_of_A} and \Cref{-1h}, we have
\begin{align}\label{-1_equal}
\begin{split}
\|u_h\|^2_{-1,h}&=\|\bm{\Pi_V}u_h\|^2_{\mathcal{T}_h}
+\|h_K^{-1/2}( \Pi_k^{\partial}\Pi_Wu_h-\Pi_M u_h)\|^2_{\partial\mathcal{T}_h}\\
&=(u_h,\Pi_Wu_h)_{\mathcal{T}_h}.
\end{split}
\end{align}

Next, we show that $\|\cdot\|_{-1,h}$  is  a norm on the space $\mathring{W}_h$.
\begin{lemma} 
	$\|\cdot\|_{-1,h}$ defines a norm on the space $\mathring{W}_h$.
\end{lemma}
\begin{proof} Thanks to \eqref{-1_equal}, one only needs to show that    $\|u_h\|_{-1,h}=0$ implies $u_h=0$ for $u_h\in \mathring{W}_h$. It follows readily from \eqref{-1_equal} that
	\begin{align*}
	\bm{\Pi_V}u_h=\bm 0,\qquad \Pi_k^{\partial}\Pi_Wu_h-\Pi_M u_h=0.
	\end{align*}
	Then  \Cref{-1h} and \eqref{def_A} give that for all 	$(\bm r_h, w_h) \in \bm{V}_h\times\mathring{W}_h$ 
	\begin{align*}
	(u_h,w_h)_{\mt_h}=(\Pi_W u_h, \nabla \cdot \bm{r}_h)_{\mt_h}-\langle\Pi_M u_h, \bm{r}_h \cdot \bm n\rangle_{\partial \mt_h}.
	\end{align*}
	This is only possible if  $u_h=0$. The proof is complete.
\end{proof}




%
%
For the negative norm $\|\cdot\|_{-1, h}$, the following HDG interpolation inequality holds true.
\begin{lemma} 
	If $u_h\in \mathring{W}_h$ and $(w_h,\mu_h)\in W_h\times M_h$,  one has
\begin{align}
(u_h,w_h)_{\mathcal{T}_h}\le C\|u_h\|_{-1,h}\left(\|\nabla w_h\|_{\mathcal{T}_h}
+\|h_K^{-1/2}(\Pi_k^{\partial}w_h-\mu_h)\|_{\partial\mathcal{T}_h}
\right)\label{Cauthy},
\end{align}
where $h_k$ is the diameter of the element $K$. 
\end{lemma}
\begin{proof}
Let $(w_h, \mu_h)\in {W}_h\times M_h$ and $u_h\in \mathring{W}_h$.  By \Cref{-1h} and \eqref{def_A} we have
	\begin{align*}
	(u_h,w_h)_{\mathcal{T}_h}&=\mathcal A(\bm{\Pi}_{\bm V}u_h,\Pi_W u_h,\Pi_M u_h;\bm 0,w_h,\mu_h)\\
	&=(\nabla\cdot\bm{\Pi_V}u_h,w_h)_{\mathcal{T}_h}-\langle \bm n\cdot\bm{\Pi_V}u_h, \mu_h\rangle_{\partial\mathcal{T}_h}\\
	&\quad +\langle h_K^{-1}(\Pi_k^{\partial}\Pi_W u_h-\Pi_M u_h),\Pi_k^{\partial}w_h-\mu_h\rangle_{\partial\mathcal{T}_h}.
	\end{align*}
	By integration by parts,  the identity \eqref{-1_equal} and the inverse inequality \eqref{eq27b} we have 
	\begin{align*}
	(u_h,w_h)_{\mathcal{T}_h} &\le 
	\|\bm{\Pi_V}u_h\|_{\mathcal{T}_h}\|\nabla w_h\|_{\mathcal{T}_h} + C\|h_K^{-1/2}( \Pi_k^{\partial} w_h-\mu_h)\|_{\partial\mathcal{T}_h}\\
	&\quad \times \left(\|\bm{\Pi_V}u_h\|_{\mathcal{T}_h}+\|h_K^{-1/2}( \Pi_k^{\partial}\Pi_Wu_h-\Pi_M u_h)\|_{\partial\mathcal{T}_h}\right)\\
	&\le C \|u_h\|_{-1,h}\left(\|\nabla w_h\|_{\mathcal{T}_h}+\|h_K^{-1/2}(\Pi_k^{\partial}w_h-\mu_h)\|_{\partial\mathcal{T}_h}\right).
	\end{align*}
	This concludes our proof.
\end{proof}

In addition, by the \Cref{-1h}, the identity \eqref{-1_equal} and \Cref{es_u} one can easily establish the following relation.
\begin{lemma} \label{lem-nega}
	For any $u_h\in \mathring{W}_h$ there holds
\begin{align}
\|\nabla \Pi_W u_h \|_{\mathcal{T}_h}+\|h_K^{-1/2}(\Pi_k^{\partial} \Pi_W u_h-\Pi_M u_h)\|_{\partial\mathcal{T}_h}  \le C\|u_h\|_{-1,h}.
\end{align}

\end{lemma}

For the error analysis  of the nonlinear equation we need  to establish the  discrete HDG Sobolev inequalities for which we will make use of the so-called Oswald interpolation operator \cite{MR2300291}.

\begin{lemma}[\cite{MR2300291}] \label{Oswald_inequality}
	There exists an interpolation operator, called Oswald interpolation, $\mathcal{I}_h^c:{W}_{h}\to{W}_{h}\cap H^1(\Omega)$, such that for any ${w}_{h}\in{W}_{h}$, 
	\begin{align}
	\sum_{K\in\mathcal{T}_h}\|{w}_{h}-\mathcal{I}_h^c{w}_{h}\|^2_{0,K}\le C  \sum_{E\in\mathcal{E}_h^o}h_E\|[\![{w}_{h}]\!]\|^2_{0,E},\label{ICK1}\\
	\sum_{K\in\mathcal{T}_h}|{w}_{h}-\mathcal{I}_h^c{w}_{h}|^2_{1,K}\le C  \sum_{E\in\mathcal{E}_h^o}h_E^{-1}\|[\![{w}_{h}]\!]\|^2_{0,E},\label{ICK2}
	\end{align}
	where $[\![w_h]\!]$ denotes the jump of $w_h$ across a side $E$
\end{lemma}

\begin{remark} From  the proof of \cite[Page 644, Theorem 2.1]{MR2300291}, one can also obtain the following estimate:
	\begin{align}\label{ICK3}
	\sum_{K\in\mathcal{T}_h}h_K^{\varepsilon}\|{v}_{h}-\mathcal{I}_h^c{v}_{h}\|^2_{0,K}\le C  \sum_{E\in\mathcal{E}_h^o}h_E^{1+\varepsilon}\|[\![{v}_{h}]\!]\|^2_{0,E},
	\end{align}
	where  $\varepsilon$ is any fixed  real constant.
\end{remark}

Now we are ready to prove  the HDG Sobolev inequalities.
\begin{theorem}[Discrete Sobolev inequalities] \label{discrete-soblev} 
	Let $\mu$  be the exponents as in the classical $H^1$ Sobolev embedding, i.e., $\mu$ satisfying
	\begin{eqnarray}\label{mu_interval}
	\left\{
	\begin{aligned}
	&1\le \mu < \infty,&\text{ if } \ d=2,\\
	&1\le \mu\le 6 ,&\text{ if } \ d=3.
	\end{aligned}
	\right.
	\end{eqnarray}For $w_h\in W_h$, it holds
	\begin{eqnarray}\label{sobolev-001}
	\|{w}_{h}\|_{0,\mu}\le C 
	\left(
	\|w_h\|_{\mathcal{T}_h}
	+
	\|\nabla w_h\|_{\mathcal{T}_h}
	+\|h_E^{-1/2}[\![ w_h]\!]\|_{\mathcal{E}_h^o}
	\right).
	\end{eqnarray}
	If further $w_h\in \mathring{W}_h$, then
	\begin{eqnarray}\label{sobolev-002}
	\|{w}_{h}\|_{0,\mu}\le C 
	\left(
	\|\nabla w_h\|_{\mathcal{T}_h}
	+\|h_E^{-1/2}[\![ w_h]\!]\|_{\mathcal{E}_h^o}
	\right).
	\end{eqnarray}
	\end{theorem}
\begin{proof}  
	We only give the proof of the inequality  \eqref{sobolev-001}, since \eqref{sobolev-002} is a direct consequence of the inequality \eqref{sobolev-001} and the Poincar\'{e}-Friedrichs inequality in \Cref{Poincare0}. 
	
	First, we note that the case $1\leq \mu<2$ is trivial since $L^2$ is embedded in $L^\mu$ by H\"{o}lder's inequality. We consider the case $\mu\ge 2$.
	
	 By the triangle inequality we have 
	\begin{align*}
	\|w_h\|_{L^\mu} \le \|\mathcal{I}_h^c {w}_{h}\|_{L^\mu} + 	\|\mathcal{I}_h^c {w}_{h} - w_h\|_{L^\mu}. 
	\end{align*}
	 Since $\mathcal{I}_h^c w_{h} \in H^1(\Omega)$, by the classical Sobolev embedding, the triangle inequality and \Cref{Oswald_inequality}, we have 
	\begin{align}
	\|\mathcal{I}_h^c w_{h}\|_{L^\mu}&\le C \left(\|\mathcal{I}_h^c w_{h}\|_{\mathcal T_h}+\|\nabla\mathcal{I}_h^c w_{h}\|_{\mathcal T_h}\right)\nonumber\\
	&\le C\left(\|\mathcal{I}_h^c w_{h} -w_h\|_{\mathcal T_h}+\|w_h\|_{\mathcal T_h}+\|\nabla w_{h}\|_{\mathcal T_h} +\|\nabla(\mathcal{I}_h^c w_{h}-w_h)\|_{\mathcal T_h}\right)\nonumber\\
	&\le C \left(\| w_{h}\|_{\mathcal T_h}+\|\nabla w_{h}\|_{\mathcal{T}_h}
	+\|h_E^{-1/2}[\![w_h]\!]\|_{\mathcal E_h^o}\right).\label{411}
	\end{align}
	For the term $\|\mathcal{I}_h^c w_{h} - w_h\|_{0,\mu}$, we use the element-wise Sobolev embedding and the discrete Minkowski's inequality to get
	\begin{align*}
	\|w_{h}- \mathcal{I}_h^c w_{h}\|_{L^\mu}
	&=\left(\sum_{K\in\mathcal{T}_h}\| w_{h}- \mathcal{I}_h^c w_{h}\|^{\mu}_{L^\mu,K}\right)^{\frac{1}{\mu}}\\
	&\le C \left(\sum_{K\in\mathcal{T}_h}
	\big(\| w_{h}- \mathcal{I}_h^c w_{h}\|_{L^2,K}+\|\nabla (w_{h}- \mathcal{I}_h^c w_{h})\|_{L^2,K}\big)^\mu \right)^{\frac{1}{\mu}}\\
	&\le C \left(\sum_{K\in\mathcal{T}_h}
	\| w_{h}- \mathcal{I}_h^c w_{h}\|_{L^2,K}^\mu \right)^{\frac{1}{\mu}}+ C \left(\sum_{K\in\mathcal{T}_h}
	\|\nabla (w_{h}- \mathcal{I}_h^c w_{h})\|_{L^2,K}^\mu \right)^{\frac{1}{\mu}}\\
	&\le C \left(\sum_{K\in\mathcal{T}_h}
	\| w_{h}- \mathcal{I}_h^c w_{h}\|_{L^2,K}^2 \right)^{\frac{1}{2}}+ \left(\sum_{K\in\mathcal{T}_h}
	\|\nabla (w_{h}- \mathcal{I}_h^c w_{h})\|_{L^2,K}^2 \right)^{\frac{1}{2}},
	\end{align*}
where the last inequality follows from the fact $\mu\ge2$ and the inequality $\sum_{i=1}^n |a_i|^{\mu/2}\le
	(\sum_{i=1}^n |a_i|)^{\mu/2}$. Lemma \ref{Oswald_inequality} then yields
	\begin{align}\label{inter-inequa}
	\| {w}_{h}- \mathcal{I}_h^c w_{h}\|_{L^\mu} \le C\|h_E^{-1/2}[\![w_h]\!]\|_{\mathcal E_h^o}.
	\end{align}
	The desired inequality \eqref{sobolev-001} now follows from the inequalities \eqref{411} and \eqref{inter-inequa}. This completes the proof.

\end{proof}

The combination of the above theorem and the triangle inequality gives the following HDG Sobolev inequality.
\begin{corollary}[HDG Sobolev inequality] \label{discrete-soblev-hdg} 
	For $(w_h,\mu_h)\in W_h\times M_h$, it holds
	\begin{align}\label{sobolev-01}
	\|{w}_{h}\|_{L^\mu}\le C \left( \|w_h\|_{\mathcal{T}_h}+\|\nabla w_h\|_{\mathcal{T}_h}
	+\|h_K^{-1/2}(\Pi_k^{\partial}w_h-\mu_h)\|_{\partial\mathcal{T}_h}\right),
	\end{align}
	if in addition $w_h\in \mathring{W}_h$, then
	\begin{eqnarray}\label{sobolev-02}
	\|w_{h}\|_{L^\mu}\le C \left(\|\nabla w_h\|_{\mathcal{T}_h}+\|h_K^{-1/2}(\Pi_k^{\partial}w_h-\mu_h)\|_{\partial\mathcal{T}_h}\right),
	\end{eqnarray}
	where  $\mu$ satisfying \eqref{mu_interval}.
\end{corollary}
\begin{proof}
For any $\mu_h \in M_h$, since $\mu_h$ is single-valued, by the triangle inequality, we have
\begin{align*}
\|h_E^{-1/2}[\![ w_h]\!]\|_{\mathcal{E}_h^o} &\leq \|h_K^{-1/2}(w_h-\mu_h)\|_{\partial\mathcal{T}_h}\\
& \leq \|h_K^{-1/2}(w_h-\Pi_k^{\partial}w_h)\|_{\partial\mathcal{T}_h}+\|h_K^{-1/2}(\Pi_k^{\partial}w_h-\mu_h)\|_{\partial\mathcal{T}_h} \\
&\leq C \|\nabla w_h \|_{\partial\mathcal{T}_h}+\|h_K^{-1/2}(\Pi_k^{\partial}w_h-\mu_h)\|_{\partial\mathcal{T}_h}.
\end{align*}
The HDG Sobolev inequality now follows from the inequality \eqref{sobolev-001}. 

The second inequality \eqref{sobolev-02} follows from the first inequality \eqref{sobolev-01} and the HDG Poincar\'{e} inequality \eqref{HDG-poincare}.
This completes the proof.
\end{proof}

\section{Well-posedness of the HDG formulation}
\label{E_U_S}

In this section we establish the well-posedness  of the HDG method \eqref{HDG-hill}, that is, existence and uniqueness of solutions as well as the energy stability of the solutions. The results differ slightly between the fully implicit discretization (FI) and the convex-spliting method (CS): the CS time marching enjoys unconditionally unique solvability and stability while there is a time-step constraint in the FI scheme for uniqueness and stability. For convenience, we will focus on the analysis of one method and point out the difference of the other. 


\subsection{Existence and uniqueness}
\begin{theorem}\label{exist}
	The HDG scheme  \eqref{HDG-hill}  admits at least one solution.
\end{theorem}
\begin{proof}
	We take $(\bm r_1, w_1,\mu_1)=(\bm 0,1,1)$ and  
	$(\bm r_2, w_2,\mu_2)=(\bm 0,1,1)$ in \eqref{HDG-hill} to get
	\begin{align}
	(u_h^n,1)_{\mathcal T_h}=(u_h^{n-1},1)_{\mathcal T_h}=\cdots=(u^0,1)_{\mathcal T_h},\quad
	(\phi_h^n,1)_{\mathcal T_h}=(\epsilon^{-1}f^n(u^n_h),1)_{\mathcal T_h}. \label{mass_conservative}
	\end{align} 
	Introduce the space
	\begin{align*}
	X=\left\{(\overline{\bm p}_h,\overline \phi_h,\widehat{\overline \phi}_h, \overline{\bm q}_h, \overline u_h,\widehat{\overline{u}}_h)\in [\bm V_h\times \mathring W_h\times M_h]^2:\right.\\
	\left.\mathcal A(\overline{\bm p}_h, \overline\phi_h,\widehat{\overline\phi}_h;\bm r_1,0,0)=
	\mathcal A(\overline{\bm q}_h,\overline u_h,\widehat{\overline u}_h;\bm r_2,0,0)=0, 
	\forall \bm r_1,\bm r_2 \in \bm V_h\right\},
	\end{align*}
	with the inner product

	\begin{align*}
	&\left( (\overline{\bm p}_h,\overline \phi_h,\widehat{\overline \phi}_h, \overline{\bm q}_h, \overline u_h,\widehat{\overline{u}}_h), (\bm r_1,w_1,\mu_1,\bm r_2,w_2,\mu_2)\right)_X\\
	&	=(\overline{\bm p}_h,  \bm r_1)_{\mathcal T_h}+\big(h_K^{-1/2}(\Pi_k^{\partial} \overline\phi_h-\widehat {\overline \phi}_h), (\Pi_k^{\partial} w_1-\mu_1) \big)_{\partial\mathcal T_h}\\
	&+\epsilon (\overline {\bm q}_h, \bm r_2)_{\mathcal T_h}+\epsilon \big(h_K^{-1/2}(\Pi_k^{\partial} \overline u_h-\widehat{\overline u}_h), (\Pi_k^{\partial} w_2-\mu_2) \big)_{\partial\mathcal T_h}.
	\end{align*}
	Thanks to \Cref{es_u} and the HDG Sobolev inequality \eqref{discrete-soblev-hdg}, the inner product is well-defined on space $X$ with an induced norm  $\|\cdot\|_X^2=(\cdot,\cdot)_{X}$.   Now we introduce a nonlinear operator $G: X \to X$ as follows
	\begin{align*}
	\hspace{-1em}&\hspace{-1em}\left(G(\overline{\bm p}_h^n,\overline\phi_h^n,\widehat{\overline\phi}_h^n, \overline {\bm q}_h^n,\overline u_h^n,\widehat {\overline u}_h^n),(\bm r_1,w_1,\mu_1,\bm r_2,w_2,\mu_2)\right)_X\\
	&:= \Delta t \mathcal A(\overline{\bm p}_h^n, \overline\phi_h^n,\widehat{\overline\phi}_h^n;\bm r_1,w_1,\mu_1)+(\overline u_h^n- u_h^{n-1},w_1)_{\mathcal T_h}\\
	&\quad +\epsilon \mathcal A(\overline{\bm q}_h^n,\overline u_h^n,\widehat{\overline u}_h^n;\bm r_2,w_2,\mu_2)+(\epsilon^{-1}f^n(\overline u_h^n+\alpha),w_2)_{\mathcal T_h}-(\overline \phi_h^n,w_2)_{\mathcal T_h},
	\end{align*}
	where $\alpha=\frac{1}{|\Omega|}(u^0,1)_{\mathcal T_h}$. By  \Cref{A_bound,es_u} and the HDG Sobolev inequality \eqref{discrete-soblev-hdg}, the operator $G$ is a continuous operator. Therefore, 
	\begin{align*}
	\hspace{-1em}&\hspace{-1em}\left(G(\overline{\bm p}_h^n,\overline\phi_h^n,\widehat{\overline \phi}_h^n, \overline {\bm q}_h^n, \overline u_h^n,\widehat{\overline u}_h^n), 
	(\overline{\bm p}_h^n,\overline\phi_h^n,\widehat{\overline \phi}_h^n, \overline {\bm q}_h^n, \overline u_h^n,\widehat{\overline u}_h^n)\right)_X\\
	& = \Delta t \left(\|\overline{\bm p}_h^n\|^2_{\mathcal T_h}+\|h_K^{-1/2}(\Pi_k^{\partial}\overline \phi_h^n-\widehat{\overline \phi}_h^n)\|_{\partial\mathcal T_h}^2\right)-( u_h^{n-1},\overline \phi_h^n)_{\mathcal T_h}\\
	&\quad + \epsilon \left( \|\overline{\bm q}_h^n\|^2_{\mathcal T_h}+\|h_K^{-1/2}(\Pi_k^{\partial}\overline u_h^n-\widehat{\overline u}_h^n)\|_{\partial\mathcal T_h}^2\right)+(\epsilon^{-1}f^n(\overline u_h^n+\alpha),\overline u_h^n)_{\mathcal T_h}.
	\end{align*}
	By \Cref{A_bound}, the Cauchy-Schwarz inequality and noting that
	\begin{align*}
	\big(f^n(\overline u_h^n+\alpha),  \overline u_h^n \big)_{\mathcal T_h} &\geq  \|\overline u_h^n\|_{L^4(\Omega)}^4 -\|\overline u_h^n\|_{L^2(\Omega)}^2 -\| u_h^{n-1}\|_{L^2(\Omega)}^2\\
	&\geq \frac{1}{2}  \|\overline u_h^n\|_{L^4(\Omega)}^4-\| u_h^{n-1}\|_{L^2(\Omega)}^2-C.
	\end{align*}
	Hence, we have 
	\begin{align*}
	\hspace{-1em}&\hspace{-1em} \left(G(\overline{\bm p}_h^n,\overline\phi_h^n,\widehat{\overline \phi}_h^n, \overline {\bm q}_h^n, \overline u_h^n,\widehat{\overline u}_h^n), 
	(\overline{\bm p}_h^n,\overline\phi_h^n,\widehat{\overline \phi}_h^n, \overline {\bm q}_h^n, \overline u_h^n,\widehat{\overline u}_h^n)\right)_X\\
	&\ge \frac{\epsilon\Delta t}{2}\|(\overline{\bm p}_h^n,\overline\phi_h^n,\widehat{\overline \phi}_h^n, \overline {\bm q}_h^n, \overline u_h^n,\widehat{\overline u}_h^n)\|_X^2+ \frac 1 2 \|\overline u_h^n\|_{L^4(\Omega)}^4 -\| u_h^{n-1}\|_{L^2(\Omega)}^2-C.
	\end{align*}
	Hence for the fixed $u_h^{n-1}$,  if
	$\|(\overline{\bm p}_h^n,\overline\phi_h^n,\widehat{\overline \phi}_h^n, \overline {\bm q}_h^n, \overline u_h^n,\widehat{\overline u}_h^n)\|_X$ is large enough, we have
	\begin{align*}
	\left(G(\overline{\bm p}_h^n,\overline\phi_h^n,\widehat{\overline \phi}_h^n, \overline {\bm q}_h^n, \overline u_h^n,\widehat{\overline u}_h^n),(\overline{\bm p}_h,\overline\phi_h^n,\widehat{\overline \phi}_h^n, \overline {\bm q}_h^n, \overline u_h^n,\widehat{\overline u}_h^n)\right)_X>0.
	\end{align*}
	Then \cite[II, Lemma 1.4]{Roger_Book_1977} implies the 
	existence of $(\overline{\bm p}_h^n,\overline\phi_h^n,\widehat{\overline \phi}_h^n, \overline {\bm q}_h^n, \overline u_h^n,\widehat{\overline u}_h^n)$ such that $G(\overline{\bm p}_h^n,\overline\phi_h^n,\widehat{\overline \phi}_h^n, \overline {\bm q}_h^n, \overline u_h^n,\widehat{\overline u}_h^n)=\bm 0$.
	
	Now we define
	\begin{align}\label{def-solution}
	(\bm p^n_h,\phi^n_h,\widehat{\phi}^n_h, \bm q_h^n, u_h^n,\widehat u_h^n)
	=(\overline{\bm p}_h^n,\overline\phi_h^n+\beta,\widehat{\overline \phi}_h^n+\beta, \overline {\bm q}_h^n, \overline u_h^n+\alpha,\widehat{\overline u}_h^n+\alpha),
	\end{align}
	with
	\begin{align*}
	\beta=\frac{1}{|\Omega|}(\epsilon^{-1}f^n(\overline u_h^{n}+\alpha),1)_{\mathcal T_h}.
	\end{align*}
	We proceed to show that $(\bm p^n_h,\phi^n_h,\widehat{\phi}^n_h, \bm q_h^n, u_h^n,\widehat u_h^n)$ is the solution to the HDG scheme  \eqref{HDG-hill}. 
	
	Since
	\begin{align*}
	&-(\beta,\nabla\cdot\bm r_1)_{\mathcal T_h}+\langle\beta,\bm r_1\cdot\bm n \rangle_{\partial\mathcal T_h}=0, \\
	&-(\alpha,\nabla\cdot\bm r_2)_{\mathcal T_h}+\langle\alpha,\bm r_2\cdot\bm n \rangle_{\partial\mathcal T_h}=0,
	\end{align*}	
	and $(1,w_1)_{\mathcal T_h}=(1,w_2)_{\mathcal T_h}=0$, one gets that
	\begin{align} \label{proof1}
	&\Delta t \mathcal A({\bm p}_h^n, \phi_h^n,\widehat{\phi}_h^n;\bm r_1,w_1,\mu_1)+( u_h^n- u_h^{n-1},w_1)_{\mathcal T_h}\nonumber\\
	&\quad +\epsilon \mathcal A({\bm q}_h^n, u_h^n,\widehat{ u}_h^n;\bm r_2,w_2,\mu_2)-( \phi_h^n,w_2)_{\mathcal T_h}\nonumber\\
	&\quad +(\epsilon^{-1}f^n( u_h^n),w_2)_{\mathcal T_h}=0
	\end{align}
	for all $(\bm r_h,w_1,\mu_1,\bm r_2,w_2,\mu_2)\in [\bm V_h\times \mathring W_h\times M_h]^2$.
	
	Noting that  Eqs. \eqref{proof1} only hold for  $w_1, w_2 \in \mathring{W_h}$. Next, we  prove that they  are true for the  $w_1, w_2 \in W_h$.
	A  direct calculation gives
	\begin{align} \label{proof2}
	&\Delta t \mathcal A({\bm p}_h^n, \phi_h^n,\widehat{\phi}_h^n;\bm 0,1,1)+( u_h^n- u_h^{n-1},1)_{\mathcal T_h}
	=( u_h^n- u_h^{n-1},1)_{\mathcal T_h} =0.
	\end{align} 
	Likewise, 
	\begin{align} \label{proof3}
	&\epsilon \mathcal A({\bm q}_h^n, u_h^n,\widehat{ u}_h^n;\bm 0,1,1)+(\epsilon^{-1}f^n( u_h^n),1)_{\mathcal T_h}-( \phi_h^n,1)_{\mathcal T_h}\nonumber\\
	&\quad=(\epsilon^{-1}f^n( u_h^n),1)_{\mathcal T_h}-( \phi_h^n,1)_{\mathcal T_h}\nonumber\\
	&\quad=(\epsilon^{-1}f^n( \overline u_h^n+\alpha),1)_{\mathcal T_h}-( \overline{\phi}_h^n+\beta,1)_{\mathcal T_h}\nonumber\\
	&\quad=(\epsilon^{-1}f^n( \overline u_h^n+\alpha),1)_{\mathcal T_h}-( \beta,1)_{\mathcal T_h}		\nonumber\\
	&\quad=0.
	\end{align}
	Collectively, \eqref{proof1}, \eqref{proof2}, \eqref{proof3} implies
	\begin{align*} 
	&\Delta t \mathcal A({\bm p}_h^n, \phi_h^n,\widehat{\phi}_h^n;\bm r_1,w_1,\mu_1)+( u_h^n- u_h^{n-1},w_1)_{\mathcal T_h}\nonumber\\
	&\quad +\epsilon \mathcal A({\bm q}_h^n, u_h^n,\widehat{ u}_h^n;\bm r_2,w_2,\mu_2)+(\epsilon^{-1}f^n( u_h^n),w_2)_{\mathcal T_h}-( \phi_h^n,w_2)_{\mathcal T_h}=0,
	\end{align*}
	holds for all $(\bm r_h,w_1,\mu_1,\bm r_2,w_2,\mu_2)\in [\bm V_h\times  W_h\times M_h]^2$. 
	This completes the proof.
\end{proof}

Next we show that the solution to the fully discrete scheme with the convex-splitting $f^n(u_h^n) = (u_h^n)^3-u_h^{n-1}$ is unique, while the solution is only conditionally unique for the case of the Backward Euler temporal discretization.
\begin{theorem}\label{unique}
	The solution to the HDG scheme of \eqref{HDG-hill} with the splitting  $f^n(u_h^n) = (u_h^n)^3-u_h^{n-1}$ is unique. On the other hand, the solution corresponding to the Backward Euler discretization is unique provided that $\Delta t\leq C \epsilon^3$.
\end{theorem}
\begin{proof} 
	Given $(\bm p_{h}^{n-1},\phi_{h}^{n-1},\widehat{\phi}_{h}^{n-1},\bm q_{h}^{n-1},u_{h}^{n-1},\widehat u_{h}^{n-1})$,
	we let $(\bm p_{h,1}^n,\phi_{h,1}^n,\widehat{\phi}_{h,1}^n,\bm q_{h,1}^n,\\u_{h,1}^n,\widehat u_{h,1}^n)$
	and $(\bm p_{h,2}^n,\phi_{hthe rest of ,2}^n,\widehat{\phi}_{h,2}^n, \bm q_{h,2}^n,u_{h,2}^n,\widehat u_{h,2}^n)$
	be two solutions of \eqref{HDG-hill}. Let
	\begin{align*}
	\bm P_h^n:=\bm p_{h,1}^n-\bm p_{h,2}^n,\qquad
	\Phi_h^n:=\phi_{h,1}^n-\phi_{h,2}^n,\qquad
	\widehat{\Phi}_h^n:=\widehat{\phi}_{h,1}^n-\widehat{\phi}_{h,2}^n,\\	
	\bm Q_h^n:=\bm q_{h,1}^n-\bm q_{h,2}^n,\qquad
	U_h^n:=u_{h,1}^n-u_{h,2}^n,\qquad
	\widehat U_h^n:=\widehat u_{h,1}^n-\widehat u_{h,2}^n.
	\end{align*}
	After inserting the two solutions into \eqref{HDG-hill}, subtracting the two  equations gives
	\begin{subequations}
		\begin{align}
		\frac{1}{\Delta t}( U_h^n,w_1)_{\mathcal{T}_h}+\mathcal A(\bm P_h^n,\Phi_h^n,\widehat{\Phi}_h^n;\bm r_1, w_1, \mu_1)&=0,\label{u1}\\
		\epsilon\mathcal{A}(\bm Q_h^n,U^n_h,\widehat{U}_h^n;\bm r_2,w_2,\mu_2)-(\Phi_h^n,w_2)_{\mathcal{T}_h}&\nonumber\\
		+\epsilon^{-1}\Big(f^n(u_{h,1}^n)-f^n(u_{h,2}^n),w_2\Big)_{\mathcal{T}_h}&=0.\label{u2}
		\end{align}
	\end{subequations}
	Take $(\bm r_1, w_1, \mu_1, \bm r_2, w_2, \mu_2)=(\Delta t\bm  P_h^n, \Delta t \Phi_h^n, \Delta t \widehat{\Phi}_h^n, \bm Q_h^n,U_h^n,\widehat U_h^n)$ to get
	\begin{align*}
	(U_h^n,\Phi_h^n)_{\mathcal T_h}+\Delta t\left(\|\bm P_h^n\|_{\mathcal T_h}^2+\|h_K^{-1/2}(\Pi_k^{\partial}\Phi_h^n-\widehat{\Phi}_h^n)\|_{\partial\mathcal T_h}^2\right)&=0, \nonumber \\
	\epsilon\left( \|\bm Q_h^n\|_{\mathcal T_h}^2+\|h_K^{-1/2}(\Pi_k^{\partial}U_h^n-\widehat U_h^n)\|_{\partial\mathcal T_h}^2\right)& \nonumber \\
	+\epsilon^{-1}\Big(f^n(u_{h,1}^n)-f^n(u_{h,2}^n),U_h^n\Big)_{\mathcal T_h}-(\Phi_h^n,U_h^n)_{\mathcal T_h}&=0.
	\end{align*}
	We add the above two equations together  to get
	\begin{align}\label{Final}
	& \Delta t\left(\|\bm P_h^n\|_{\mathcal T_h}^2+\|h_K^{-1/2}(\Pi_k^{\partial}\Phi_h^n-\widehat{\Phi}_h^n)\|_{\mathcal T_h}^2\right) \nonumber \\
	&\quad +\epsilon\left( \|\bm Q_h^n\|_{\mathcal T_h}^2+\|h_K^{-1/2}(\Pi_k^{\partial}U_h^n-\widehat U_h^n)\|_{\mathcal T_h}^2\right) \nonumber \\
	&\quad +\epsilon^{-1} \left(f^n(u_{h,1}^n)-f^n(u_{h,2}^n),U_h^n\right)_{\mathcal T_h}
	=0.
	\end{align}
	
	In the case of the convex-splitting $f^n(u_h^n) = (u_h^n)^3-u_h^{n-1}$, one has 
	\begin{align*}
	\epsilon^{-1}\left(f^n(u_{h,1}^n)-f^n(u_{h,2}^n),U_h^n\right)_{\mathcal T_h}&=\epsilon^{-1}\left((u_{h,1}^n)^3-(u_{h,2}^n)^3,U_h^n\right)_{\mathcal T_h} \\
	&=\epsilon^{-1}\left((u_{h,1}^n)^2+u_{h,1}^n u_{h,2}^n+(u_{h,2}^n)^2,(U_h^n)^2\right)_{\mathcal T_h} \\
	&\geq 0.
	\end{align*}
	Hence all  terms on the left hand side of \eqref{Final} are nonnegative. It follows that
	\begin{align}\label{uniquness_proof_2}
	\begin{split}
	\|\bm P_h^n\|_{\mathcal T_h}^2+\|h_K^{-1/2}(\Pi_k^{\partial}\Phi_h^n-\widehat{\Phi}_h^n)\|_{\partial \mathcal T_h}^2=0,\\
	\|\bm Q_h^n\|_{\mathcal T_h}^2+\|h_K^{-1/2}(\Pi_k^{\partial}U_h^n-\widehat U_h^n)\|_{\partial \mathcal T_h}^2=0.
	\end{split}
	\end{align}
	Now, we take $(\bm r_1, w_1, \mu_1)=(\bm 0,1,1)$ in \eqref{u1} and  $(\bm r_2, w_2, \mu_2)=(\bm 0,1,1)$ in \eqref{u2} to get $(U_h^n,1)_{\mathcal T_h}=0$, $	(\Phi_h^n,1)_{\mathcal T_h}=0$.
	By \eqref{sobolev-02},  \Cref{es_u} and \eqref{uniquness_proof_2}, we have $(\bm P_{h}^{n},\Phi_{h}^{n},\widehat{\Phi}_{h}^{n}, \bm Q_{h}^{n},U_{h}^{n},\widehat U_{h}^{n})=0$.

	For the case of the Backward Euler method, i.e., 	In the case of the convex-splitting $f^n(u_h^n) = (u_h^n)^3-u_h^{n}$, one has 
	\begin{align*}
	\epsilon^{-1}\left(f^n(u_{h,1}^n)-f^n(u_{h,2}^n),U_h^n\right)_{\mathcal T_h}&=\epsilon^{-1}\left((u_{h,1}^n)^3-(u_{h,2}^n)^3,U_h^n\right)_{\mathcal T_h}-\epsilon^{-1}\|U_h^n\|_{\mathcal{T}_h}^2. 
	\end{align*}
	In light of  Eq. \eqref{u1} we estimate $\epsilon^{-1}\|U_h^n\|_{\mathcal{T}_h}^2$  using the continuity of $\mathcal{A}$ and \Cref{es_u} as follows
	\begin{align*}
	-\epsilon^{-1} \|U_h^n\|_{\mathcal{T}_h}^2 &\geq -\frac{\Delta t}{2}\left(\|\bm P_h^n\|_{\mathcal T_h}^2+\|h_K^{-1/2}(\Pi_k^{\partial}\Phi_h^n-\widehat{\Phi}_h^n)\|_{\mathcal T_h}^2\right) \\
	&\quad -C \Delta t \epsilon^{-2} \left(\|\bm Q_h^n\|_{\mathcal T_h}^2+\|h_K^{-1/2}(\Pi_k^{\partial}U_h^n-\widehat U_h^n)\|_{\partial \mathcal T_h}^2\right)
	\end{align*}
	Hence provided that $\Delta t\leq \frac{\epsilon^3}{2C}$, Eq. \eqref{Final} yields
	\begin{align*}
	& \Delta t\left(\|\bm P_h^n\|_{\mathcal T_h}^2+\|h_K^{-1/2}(\Pi_k^{\partial}\Phi_h^n-\widehat{\Phi}_h^n)\|_{\mathcal T_h}^2\right) 
	+\epsilon\left( \|\bm Q_h^n\|_{\mathcal T_h}^2+\|h_K^{-1/2}(\Pi_k^{\partial}U_h^n-\widehat U_h^n)\|_{\mathcal T_h}^2\right) \leq 0.
	\end{align*}
	One then obtains uniqueness of the solution for the Backward Euler scheme. This completes the proof.
\end{proof}

\subsection{Energy stability}
In this subsection, we analyze the stability of the HDG formulation \eqref{HDG-hill}, focusing on  the fully implicit scheme, i.e., $f^n(u_h^n)=(u_h^n)^3-u_h^n$.  We first recall some useful identities.

%
	%
	%
	%
	%
	%
	%
	%
	%
	%



\begin{lemma}\label{identity} 
	Let a,b be two real numbers. Let  $\{a_n\}_{n=0}^m$ and $\{b_n\}_{n=0}^m$ be two sequences such that $b_0=0$. Then the following  identities hold
	\begin{subequations}
		\begin{align}
		(a-b)a&=\frac{1}{2}[a^2-b^2+(a-b)^2],\label{identity1} \\
		(a^3-a)(a-b)\color{black}&=\frac{1}{4}\left[(a^2-1)^2-(b^2-1)^2+(a^2-b^2)^2 \right.\nonumber \\
		&\left.\quad +2(a(a-b))^2-2(a-b)^2\right],\label{identity2}\\
		(a^3-b)(a-b)\color{black}&=\frac{1}{4}\left[(a^2-1)^2-(b^2-1)^2+(a^2-b^2)^2\right.\nonumber\\
		&\left.\quad +2(a(a-b))^2+2(a-b)^2\right],\label{identity3} \\
		\sum_{n=1}^m a_n b_n (b_n-b_{n-1})&=-\frac{1}{2}\sum_{n=1}^m(a_n-a_{n-1})b_{n-1}^2+\frac{1}{2}\sum_{n=1}^m a_n(b_n-b_{n-1})^2 \nonumber \\
		& \quad+\frac{1}{2}a_m{b_m}^2. \label{identity4}
		\end{align}
	\end{subequations}
\end{lemma}

The first energy identity makes use of the negative norm and takes the following form.

\begin{lemma}[Discrete energy identy I] \label{DEL}
	Let $(\bm p_h^n,\phi_h^n,\widehat \phi_h^n)$, $(\bm q_h^n,u_h^n,\widehat u_h^n)$ be the solution of \eqref{HDG-hill}. 
 The following energy identity holds for $m=1, \ldots, N$
 \begin{equation}\label{EnergyI1}
	\begin{aligned}
	& \quad \frac{1}{4\epsilon} \|( u_h^m)^2-1\|^2_{\mathcal{T}_h}+\frac{\epsilon}{2}\|\bm q_h^m\|^2_{\mathcal{T}_h} \\
	&\qquad+\frac{\epsilon}{2}\|h_K^{-1/2}(\Pi_k^{\partial}u_h^m-\widehat u_h^m)\|^2_{\partial\mathcal{T}_h}+\Delta t \sum_{n=1}^m\|\partial_{t}^+u_h^n\|^2_{-1,h}\\
	& \qquad+(\Delta t)^2\sum_{n=1}^m\left(\frac{\epsilon}{2}\|\partial_{t}^+ \bm q_h^n\|^2_{\mathcal{T}_h}+\frac{\epsilon}{2}\|h_K^{-1/2}\partial_{t}^+(\Pi_k^{\partial}u_h^n-\widehat u_h^n)\|^2_{\partial\mathcal{T}_h} \right)\\
	&\qquad +(\Delta t)^2 \sum_{n=1}^m\left(\frac{1}{4\epsilon}\|\partial_{t}^+(u_h^n)^2\|^2_{\mathcal{T}_h}+\frac{1}{2\epsilon}\|u_h^n\partial_{t}^+ u_h^n\|^2_{\mathcal{T}_h}-\frac{1}{2\epsilon}\|\partial_{t}^+u_h^n\|^2_{\mathcal{T}_h}\right)\\
	&=	\frac{1}{4\epsilon}\|( u_h^{0})^2-1\|^2_{\mathcal{T}_h}+\frac{\epsilon}{2}\|\bm q_h^{0}\|^2_{\mathcal{T}_h}+\frac{\epsilon}{2}\|h_K^{-1/2}(\Pi_k^{\partial}u_h^{0}-\widehat u_h^{0})\|^2_{\partial\mathcal{T}_h}.
	\end{aligned}
\end{equation}
\end{lemma}

\begin{proof}
Due to the mass conservation property of the scheme, it holds $\partial_{t}^+ u_h^n \in \mathring{W}_h$. We take $(\bm r_1, w_1,\mu_1)=-(\bm{\Pi}_{\bm V}\partial_{t}^+ u_h^n,-\Pi_W\partial_{t}^+ u_h^n,-\Pi_M \partial_{t}^+ u_h^n)$ in \eqref{hdg01} to get
	\begin{align*}
	(\partial_{t}^+ u_h^n,\Pi_W \partial_t^+ u_h^n)_{\mathcal{T}_h}&=\mathcal A(\bm p_h^n,\phi_h^n,\widehat{\phi}_h^n;\bm{\Pi}_{\bm V}\partial_{t}^+ u_h^n,-\Pi_W\partial_{t}^+ u_h^n,-\Pi_M \partial_{t}^+ u_h^n)\\
	& = \mathcal A(\bm{\Pi}_{\bm V}\partial_{t}^+ u_h^n, \Pi_W\partial_{t}^+ u_h^n,\Pi_M \partial_{t}^+ u_h^n;\bm p_h^n,-\phi_h^n,-\widehat{\phi}_h^n) &\text{by }\eqref{summertic_A1}\\
	& = -(\partial_{t}^+ u_h^n,\phi_h^n)_{\mathcal{T}_h}. &\text{by }\eqref{def_-1_h}\color{blue}
	\end{align*}
	In view of \eqref{-1_equal}, one obtains that
	\begin{align}\label{proof:s1}
	\|\partial_{t}^+ u_h^n\|^2_{-1,h}+(\partial_{t}^+u_h^n,\phi_h^n)_{\mathcal{T}_h}=0.
	\end{align}

	Next, we take $(\bm r_2, w_2,\mu_2)=(\bm 0,\partial_{t}^+u_h^n,\partial_{t}^+ \widehat u_h^n)$ in \eqref{hdg02} to get
	\begin{equation}\label{proof:s2}
	\begin{aligned}
	&(\epsilon^{-1}f^n(u_h^n),\partial_t^+ u_h^n)_{\mathcal{T}_h}+\epsilon(\nabla\cdot\bm q_h^n,\partial_{t}^+ u_h^n)_{\mathcal{T}_h}-\epsilon\langle\bm q_h^n\cdot\bm n,\partial_{t}^+\widehat u_h^n \rangle_{\partial\mathcal{T}_h}\\
	&\qquad+\epsilon\langle h_K^{-1}(\Pi_k^{\partial}u_h^n-\widehat u_h^n),\partial_{t}^+ (\Pi_k^{\partial}u_h^n-\widehat u_h^n) \rangle_{\partial\mathcal T_h}-(\phi_h^n,\partial_{t}^+ u_h^n)_{\mathcal{T}_h}=0\color{blue}.\color{black}
	\end{aligned}
	\end{equation}
	Then we apply $\partial_{t}^+$ to \eqref{hdg02} and take $(\bm r_2, w_2,\mu_2)=(\bm q_h^n,0,0)$ so that
	\begin{align}\label{proof:s3}
	\epsilon(\partial_{t}^+ \bm q_h^n,\bm q_h^n)_{\mathcal{T}_h}-\epsilon(\partial_{t}^+ u_h^n,\nabla\cdot\bm q_h^n)_{\mathcal{T}_h}+\epsilon\langle\partial_{t}^+ \widehat u_h^n,\bm q_h^n\cdot\bm n\rangle_{\partial\mathcal{T}_h}=0.
	\end{align}
	Taking summation of \eqref{proof:s1}, \eqref{proof:s2} and \eqref{proof:s3} gives
	\begin{align*}
	&\|\partial_{t}^+ u_h^n\|^2_{-1,h}+\epsilon^{-1}(f^n(u_h^n),\partial_{t}^+ u_h^n)_{\mathcal{T}_h}+\epsilon(\partial_{t}^+\bm q_h^n,\bm q_h^n)_{\mathcal{T}_h}\\
	&\qquad+\epsilon\langle h_K^{-1}(\Pi_k^{\partial}u_h^n-\widehat u_h^n),\partial_{t}^+(\Pi_k^{\partial}u_h^n-\widehat u_h^n) \rangle_{\partial\mathcal T_h }=0.
	\end{align*}
	Then multiply this equation by $\Delta t$  and use the identities in  \Cref{identity} to get
	\begin{align*}
	&\quad \frac{1}{4\epsilon}\|( u_h^n)^2-1\|^2_{\mathcal{T}_h}+\frac{\epsilon}{2}\|\bm q_h^n\|^2_{\mathcal{T}_h}+\frac{\epsilon}{2}\|h_K^{-1/2}(\Pi_k^{\partial}u_h^n-\widehat u_h^n)\|^2_{\partial\mathcal{T}_h}\\
	& \qquad +\Delta t \|\partial_{t}^+ u_h^n\|^2_{-1,h}+(\Delta t)^2\left(\frac{\epsilon}{2}\|\partial_{t}^+ \bm q_h^n\|^2_{\mathcal{T}_h}+\frac{\epsilon}{2}\|h_K^{-1/2}\partial_{t}^+ (\Pi_k^{\partial}u_h^n-\widehat u_h^n)\|^2_{\partial\mathcal{T}_h}\right)\\
	&\qquad +(\Delta t)^2 \left(\frac{1}{4\epsilon}\|\partial_{t}^+ (u_h^n)^2 \|^2_{\mathcal{T}_h}+\frac{1}{2\epsilon}\|u_h^n\partial_{t}^+ u_h^n\|^2_{\mathcal{T}_h}-\frac{1}{2\epsilon}\|\partial_{t}^+ u_h^n\|^2_{\mathcal{T}_h}\right)\\
	&=	\frac{1}{4\epsilon}\|( u_h^{n-1})^2-1\|^2_{\mathcal{T}_h}+\frac{\epsilon}{2}\|\bm q_h^{n-1}\|^2_{\mathcal{T}_h}+\frac{\epsilon}{2}\|h_K^{-1/2}(\Pi_k^{\partial}u_h^{n-1}-\widehat u_h^{n-1})\|^2_{\partial\mathcal{T}_h}.
	\end{align*}
	The identity \eqref{EnergyI1} follows from summing the above equation from $n=1$ to $n=m$. This completes the proof.
	\end{proof}

Next, we give the second energy identity which involves the $L^2$ norm of $\bm p_h^n$.
\begin{lemma}[Discrete energy identity II] \label{DEL-2}
	Let $(\bm p_h^n,\phi_h^n,\widehat \phi_h^n)$, $(\bm q_h^n,u_h^n,\widehat u_h^n)$ be the solution of \eqref{HDG-hill}. Then
	for $m= 1, \ldots N$, we have the following energy identity
	\begin{equation}\label{EnergyI2}
	\begin{aligned}
	&\quad \frac{1}{4\epsilon} \|( u_h^m)^2-1\|^2_{\mathcal{T}_h}+\frac{\epsilon}{2}\|\bm q_h^m\|^2_{\mathcal{T}_h}+\frac{\epsilon}{2}\|h_K^{-1/2}(\Pi_k^{\partial}u_h^m-\widehat u_h^m)\|^2_{\partial\mathcal{T}_h}\\
	&\qquad +\Delta t\sum_{n=1}^m\left(\|\bm p_h^n\|_{\mathcal{T}_h}^2+\|h_K^{-1/2}(\Pi_k^{\partial}\phi_h^n-\widehat{\phi}_h^n)\|^2_{\partial\mathcal{T}_h}\right)\\
	&\qquad +(\Delta t)^2 \sum_{n=1}^m\left(\frac{1}{2\epsilon}\|\partial_{t}^+\bm q_h^n\|^2_{\mathcal{T}_h}+\frac{1}{2}\|h_K^{-1/2}\partial_{t}^+(\Pi_k^{\partial}u_h^n-\widehat u_h^n)\|^2_{\partial\mathcal{T}_h}\right)\\
	&\qquad +(\Delta t)^2\sum_{n=1}^m\left(\frac{1}{4\epsilon}\|\partial_{t}^+[(u_h^n)^2]\|^2_{\mathcal{T}_h}+\frac{1}{2\epsilon}\|u_h^n\partial_{t}^+ u_h^n\|^2_{\mathcal{T}_h}-\frac{1}{2\epsilon}\|\partial_{t}^+u_h^n\|^2_{\mathcal{T}_h}\right)\\
	&=	\frac{1}{4\epsilon}\|( u_h^{0})^2-1\|^2_{\mathcal{T}_h}+\frac{\epsilon}{2}\|\bm q_h^{0}\|^2_{\mathcal{T}_h}+\frac{\epsilon}{2}\|h_K^{-1/2}(\Pi_k^{\partial}u_h^{0}-\widehat u_h^{0})\|^2_{\partial\mathcal{T}_h}.
	\end{aligned}
	\end{equation}	
\end{lemma}
\begin{proof}
	We take $(\bm r_1, w_1, \mu_1)=(\bm p_h^n,\phi_h^n,\widehat\phi_h^n)$ in \eqref{hdg01} to get
	\begin{align}\label{pp03}
	(\partial_{t}^+ u_h^n,\phi_h^n)_{\mathcal{T}_h}+\|\bm p_h^n\|_{\mathcal{T}_h}^2+\|h_K^{-1/2}(\Pi_k^{\partial}\phi_h^n-\widehat{\phi}_h^n)\|^2_{\partial\mathcal{T}_h}=0.
	\end{align}
	Taking the sum of \eqref{proof:s2}, \eqref{proof:s3} and \eqref{pp03} gives
	\begin{align*}
	&\quad \epsilon^{-1}(f^n(u_h^n),\partial_{t}^+u_h^n)_{\mathcal{T}_h}+\epsilon(\partial_{t}^+ \bm q_h^n,\bm q_h^n)_{\mathcal{T}_h}\\
	&\qquad +\epsilon\langle h_K^{-1}(\Pi_k^{\partial}u_h^n-\widehat u_h^n),\partial_{t}^+ (\Pi_k^{\partial}u_h^n-\widehat u_h^n) \rangle_{\partial\mathcal T_h }\\
	&\qquad+\|\bm p_h^n\|_{\mathcal{T}_h}^2+\|h_K^{-1/2}(\Pi_k^{\partial}\phi_h^n-\widehat{\phi}_h^n)\|^2_{\partial\mathcal{T}_h}\\
	&=0.
	\end{align*}
	The energy identity now follows from applications of the identities in  \Cref{identity} and addition of the  resulting equation from $n=1$ to $n=m$.
\end{proof}

There is a negative term, i.e., $-\frac{(\Delta t)^2}{2\epsilon}\|\partial_{t}^+u_h^n\|^2_{\mathcal{T}_h}$ in the energy identities I and II in the foregoing \Cref{DEL,DEL-2}. In  \Cref{corEnergy}, we  bound this term and derive the stability result.
\begin{theorem}\label{corEnergy}
	Let $C$ be the product of the constants in \eqref{Cauthy} and \Cref{es_u}. 
	If the time-step satisfies the constraint $\Delta t\leq \frac{2 \epsilon^3}{C}$, then for  $m=1, 2, \ldots N$,  the following energy bounds hold
\begin{align}
	\hspace{1em}&\hspace{-1em} \|u_h^m\|_{0,4}^2+\|\bm q_h^m\|_{\mathcal{T}_h}^2 + \|\nabla u_h^m\|_{\mathcal{T}_h}^2+  \|h_K^{-1/2}(\Pi_k^{\partial}u_h^m-\widehat u_h^m)\|_{\partial\mathcal{T}_h}^2 +\Delta t \sum_{n=1}^m \|\nabla \phi_h^n\|_{\mathcal{T}_h}^2 \nonumber \\ 
	& +\Delta t \sum_{n=1}^m\|\partial_{t}^+ u_h^n\|^2_{-1,h} + \Delta t \sum_{n=1}^m
	\left(\|\bm p_h^n\|_{\mathcal{T}_h}^2 +\|h_K^{-1/2}(\Pi_k^{\partial}\phi_h^n-\widehat{\phi}_h^n)\|^2_{\partial\mathcal{T}_h}\right) \leq C.	\label{corEnergy-1}
\end{align}

\end{theorem}

\begin{proof}
	Since $\partial_t^+ u_h^n \in \mathring{W}_h$, by the interpolation inequality \eqref{Cauthy} and  \Cref{es_u} we have 
	\begin{align*}
	\hspace{1em}&\hspace{-1em}\frac{(\Delta t)^2}{2\epsilon}\|\partial_t^+ u_h^n\|_{\mathcal{T}_h}^2 \\
	&\leq \frac{(\Delta t)^2 C}{2\epsilon} \|\partial_t^+ u_h^n\|_{-1, h} \left(\|\partial_{t}^+ \nabla u_h^n\|_{\mathcal{T}_h}+\|h_K^{-1/2}\partial_{t}^+(\Pi_k^{\partial}u_h^n-\widehat u_h^n)\|^2_{\partial\mathcal{T}_h}\right) \\
	&\leq \frac{(\Delta t)^2 C}{2\epsilon} \|\partial_t^+ u_h^n\|_{-1, h} \left(\|\partial_{t}^+ \bm q_h^n\|_{\mathcal{T}_h}+\|h_K^{-1/2}\partial_{t}^+(\Pi_k^{\partial}u_h^n-\widehat u_h^n)\|^2_{\partial\mathcal{T}_h}\right)\\
	&\leq \frac{\Delta t}{2}\|\partial_t^+ u_h^n\|_{-1, h}^2+ \frac{(\Delta t)^3 C}{4\epsilon^3}\left(\frac{\epsilon}{2}\|\partial_{t}^+ \bm q_h^n\|^2_{\mathcal{T}_h}
	+\frac{\epsilon}{2}\|h_K^{-1/2}\partial_{t}^+(\Pi_k^{\partial}u_h^n-\widehat u_h^n)\|^2_{\partial\mathcal{T}_h}\right)\color{blue}.
	\end{align*}
	Now, by the energy identity I \eqref{EnergyI1} and the assumption $\Delta t \leq \frac{2 \epsilon^3}{C}$, we have 
	\begin{align*}
	&\quad \frac{1}{4\epsilon} \|( u_h^m)^2-1\|^2_{\mathcal{T}_h}+\frac{\epsilon}{2}\|\bm q_h^m\|^2_{\mathcal{T}_h}+\frac{\epsilon}{2}\|h_K^{-1/2}(\Pi_k^{\partial}u_h^m-\widehat u_h^m)\|^2_{\partial\mathcal{T}_h}\\
	& \qquad +\frac{\Delta t}{2}\sum_{n=1}^m\|\partial_{t}^+ u_h^n\|^2_{-1,h}+ (\Delta t)^2 \sum_{n=1}^m\left(\frac{1}{4\epsilon}\|\partial_t ^+(u_h^n)^2\|^2_{\mathcal{T}_h}+\frac{1}{2\epsilon}\|u_h^n\partial_{t}^+ u_h^n\|^2_{\mathcal{T}_h}\right)\\
	&\qquad +\frac{(\Delta t)^2}{2}\sum_{n=1}^m\left(\frac{\epsilon}{2}\|\partial_{t}^+ \bm q_h^n\|^2_{\mathcal{T}_h}+\frac{\epsilon}{2}\|h_K^{-1/2}\partial_{t}^+ (\Pi_k^{\partial}u_h^n-\widehat u_h^n)\|^2_{\partial\mathcal{T}_h}\right)\\
	&\leq\frac{1}{4\epsilon}\|( u_h^{0})^2-1\|^2_{\mathcal{T}_h}+\frac{\epsilon}{2}\|\bm q_h^{0}\|^2_{\mathcal{T}_h}+\frac{\epsilon}{2}\|h_K^{-1/2}(\Pi_k^{\partial}u_h^{0}-\widehat u_h^{0})\|^2_{\partial\mathcal{T}_h}.
	\end{align*}
The energy bounds \eqref{corEnergy-1} now follow from the energy law above, the energy identity II \eqref{EnergyI2} and \Cref{es_u}.
\end{proof}

\begin{remark}
In the case of the energy-splitting scheme $f^n(u_h^n)=(u_h^n)^3-u_h^{n-1}$, by the identity \eqref{identity3}, an energy law holds in the spirit of  \Cref{DEL} where all terms are associated with the positive sign. 
\end{remark}

We note that in \Cref{corEnergy}, the energy term $\|\phi_h^m\|_{\mathcal T_h}^2$ is not contained. Moreover, the HDG Poincar\'{e} inequality \eqref{sobolev-02} does not apply to $\phi_h^n$ since $\phi_h^n \notin \mathring{W}_h$. Hence, we need a refined analysis for this term. 
In the following we derive further a priori bounds for the solution of the fully implicit  HDG scheme \eqref{HDG-hill} with the assumption $\Delta t\leq \frac{2 \epsilon^3}{C}$.

\begin{lemma}\label{phi_infty} 
	The assumption is the same as in Theorem \ref{corEnergy}. There holds
	 \begin{align}
	\hspace{1em}&\hspace{-1em} \|\phi_h^{m}\|_{\mathcal T_h}^2+\sum_{n=1}^{m}\|\phi_h^n-\phi_h^{n-1}\|^2_{\mathcal T_h}+\epsilon \Delta t\sum_{n=1}^{m}\|\partial_{t}^+ u_h^n\|_{\mathcal T_h}^2 \leq C, \label{phi_infty-1}
\end{align}
where $C$ may depend on $\epsilon, T$ and the initial condition.
\end{lemma}
\begin{proof}

Taking $(\bm r_2, w_2, \mu_2)=(\bm 0, 1, 1)$ in Eq. \eqref{hdg02} yields
\begin{align*}
|(\phi_h^n, 1)_{\mathcal{T}_h}| =\frac{1}{\epsilon}|\big(f^n(u_h^n), 1\big)_{\mathcal{T}_h}| 
\leq \frac{1}{\epsilon} (||u_h^n||_{L^3}^3+||u_h^n||_{L^1})
\leq \frac{C}{\epsilon},
\end{align*}
where the last inequality follows from the stability bounds in \eqref{corEnergy-1}. In light of the stability bounds from \eqref{corEnergy-1} and the HDG Poincar\'{e} inequality \eqref{HDG-poincare}, one gets that 
\begin{align} \label{ine-end}
\delta t \sum_{n=1}^m ||\phi_h^n||_{\mathcal{T}_h}^2 \leq \frac{C}{\epsilon}.
\end{align}

	We apply $\partial_{t}^+$ to \eqref{hdg02} and keep \eqref{hdg01} unchanged to get
	\begin{subequations}
		\begin{align}
		(\partial_{t}^+ u_h^n,w_1)_{\mathcal{T}_h}+\mathcal A(\bm p_h^n,\phi_h^n,\widehat{\phi}_h^n;\bm r_1,w_1,\mu_1)&=0,\label{hdg001}\\
		\epsilon\mathcal{A}( \partial_{t}^+ \bm q_h^n,\partial_{t}^+ u_h^n,\partial_{t}^+ \widehat{u}_h^n;\bm r_2,w_2,\mu_2)&\nonumber\\
		+(\epsilon^{-1}\partial_{t}^+f^n(u_h^n),w_2)_{\mathcal{T}_h}-(\partial_{t}^+\phi_h^n,w_2)_{\mathcal{T}_h}&=0.\label{hdg002}
		\end{align}
	\end{subequations}
	Take $(\bm r_1, w_1, \mu_1)=-\epsilon(\partial_{t}^+ \bm q_h^n,-\partial_{t}^+ u_h^n,-\partial_{t}^+ \widehat u_h^n)$ in \eqref{hdg001} and $(\bm r_2,w_2,\mu_2)=(\bm p_h^n,\\ -\phi_h^n,-\widehat{\phi}_h^n)$ in \eqref{hdg002} to get
	\begin{align*}
	\epsilon (\partial_{t}^+ u_h^n,\partial_{t}^+ u_h^n)_{\mathcal{T}_h}-\epsilon\mathcal A(\bm p_h^n,\phi_h^n,\widehat{\phi}_h^n;\partial_{t}^+\bm q_h^n,-\partial_{t}^+u^n_h,-\partial_{t}^+{\widehat u}^n_h)&=0,\\
	\epsilon\mathcal{A}(\partial_{t}^+\bm q_h^n,\partial_{t}^+u_h^n,\partial_{t}^+\widehat{u}_h^n;\bm p_h^n, -\phi_h^n,-\widehat \phi_h^n)&\\
	+(\epsilon^{-1}\partial_{t}^+ f^n(u_h^n),-\phi_h^n)_{\mathcal{T}_h}-(\partial_{t}^+\phi_h^n,-\phi^n_h)_{\mathcal{T}_h}&=0.
	\end{align*}
	
	In light of the symmetry property \eqref{summertic_A1},  adding the above two equations together gives
	\begin{align}\label{add:0}
	(\partial_{t}^+\phi_h^n,\phi_h^n)_{\mathcal{T}_h}+\epsilon\|\partial_{t}^+u_h^n\|^2_{\mathcal{T}_h}=\epsilon^{-1}(\partial_{t}^+f^n(u_h^n),\phi_h^n)_{\mathcal{T}_h}.
	\end{align}
	By the identity \eqref{identity1} we have 
	\begin{align*}
	(\partial_{t}^+ \phi_h^n,\phi_h^n)_{\mathcal{T}_h}=\frac{1}{2\Delta t}\left(\|\phi_h^n\|^2_{\mathcal{T}_h}-\|\phi_h^{n-1}\|^2_{\mathcal{T}_h}+\|\phi_h^n-\phi_h^{n-1}\|^2_{\mathcal{T}_h}\right),
	\end{align*}
	and hence
	\begin{align}\label{proof_111}
	\begin{split}
	\hspace{1em}&\hspace{-1em}\|\phi_h^{m}\|_{\mathcal T_h}^2+\sum_{n=1}^{m}\|\phi_h^n-\phi_h^{n-1}\|^2_{\mathcal T_h}+2\Delta t \epsilon\sum_{n=1}^{m}\|\partial_{t}^+u_h^n\|_{\mathcal T_h}^2\\
	&=2\epsilon^{-1}\Delta t \sum_{n=1}^{m}(\partial_{t}^+ f^n(u_h^n),\phi_h^n)_{\mathcal{T}_h}+\|\phi_h^0\|_{\mathcal T_h}^2.
	\end{split}
	\end{align}

	By  H\"{o}lder's inequality, we have 
	\begin{align*}
	\hspace{1em}&\hspace{-1em}\Delta t \sum_{n=1}^{m} (\partial_{t}^+f^n(u_h^n),\phi_h^n)_{\mathcal{T}_h}\\
	&\le C\Delta t \sum_{n=1}^{m} \|\partial_{t}^+ u_h^n\|_{\mathcal T_h}\|\phi_h^n\|_{L^6(\Omega)}\left(\|(u_h^n)^2\|_{L^3(\Omega)}+\|(u_h^{n-1})^2\|_{L^3(\Omega)}+1\right)\\
	&= C\Delta t \sum_{n=1}^{m}\|\partial_{t}^+u_h^n\|_{\mathcal T_h}\|\phi_h^n\|_{L^6(\Omega)}\left(\|u_h^n\|_{L^6(\Omega)}^2+\|u_h^{n-1}\|_{L^6(\Omega)}^2+1\right).
	\end{align*}
	Next, by the HDG Sobolev inequality in \Cref{discrete-soblev-hdg}, we have 
    \begin{align*}
    \hspace{1em}&\hspace{-1em}\Delta t \sum_{n=1}^{m} (\partial_{t}^+f^n(u_h^n),\phi_h^n)_{\mathcal{T}_h}\\
	&\le  C\Delta t\sum_{n=1}^{m} \|\partial_{t}^+ u_h^n\|_{\mathcal T_h} \|\phi_h^n\|_{L^6(\Omega)} \left(\|u_h^n\|_{\mathcal T_h}^2  + \|u_h^{n-1}\|_{\mathcal T_h}^2 + \|\bm q_h^n\|_{\mathcal T_h}^2+ \|\bm q_h^{n-1}\|_{\mathcal T_h}^2 \right.\\
	&\qquad \left. +\|h_K^{-1/2}(\Pi_k^{\partial}u_h^n-\widehat u_h^n)\|_{\partial\mathcal{T}_h}^2+\|h_K^{-1/2}(\Pi_k^{\partial}u_h^{n-1}-\widehat u_h^{n-1})\|_{\partial\mathcal{T}_h}+1\right).
	\end{align*}
	We then use \Cref{corEnergy} and the Cauchy-Schwarz inequality to get 
	\begin{align*}
	\Delta t \sum_{n=1}^{m} (\partial_{t}^+f^n(u_h^n),\phi_h^n)_{\mathcal{T}_h}
	&\le C  \Delta t\sum_{n=1}^{m}\|\partial_{t}^+u_h^n\|_{\mathcal T_h}\|\phi_h^n\|_{L^6(\Omega)}\\
	&\le  \epsilon \Delta t\sum_{n=1}^{m}\|\partial_{t}^+u_h^n\|_{\mathcal T_h}^2 + C\Delta t \sum_{n=1}^m \|\phi_h^n\|_{L^6(\Omega)}^2.
	\end{align*}
	Together with \eqref{proof_111} and the above inequality, we have 
	\begin{align*}
	\hspace{1em}&\hspace{-1em} \|\phi_h^{m}\|_{\mathcal T_h}^2+\sum_{n=1}^{m}\|\phi_h^n-\phi_h^{n-1}\|^2_{\mathcal T_h}+\Delta t\sum_{n=1}^{m}\|\partial_{t}^+ u_h^n\|_{\mathcal T_h}^2\\%
	&\le C\Delta t \sum_{n=1}^m \|\phi_h^n\|_{L^6(\Omega)}^2 + C \|\phi_h^0\|_{\mathcal T_h}^2.
	\end{align*}
	 From the HDG Sobolev embedding inequality  \Cref{discrete-soblev-hdg} again, one gets

	 \begin{align*}
	 \hspace{1em}&\hspace{-1em} \|\phi_h^{m}\|_{\mathcal T_h}^2+\sum_{n=1}^{m}\|\phi_h^n-\phi_h^{n-1}\|^2_{\mathcal T_h}+\Delta t\sum_{n=1}^{m}\|\partial_{t}^+ u_h^n\|_{\mathcal T_h}^2\\%
	 &\le C\Delta t \sum_{n=1}^m (\|\bm p_h^n\|_{\mathcal T_h}^2 +\|h_K^{-1/2}(\Pi_k^{\partial}\phi_h^n-\widehat \phi_h^n)\|_{\partial\mathcal{T}_h})\\%
	&\quad  + C \|\phi_h^0\|_{\mathcal T_h}^2 + C\Delta t \sum_{n=1}^m  \|\phi_h^{n}\|_{\mathcal T_h}^2 \\
	&\leq C,
	 \end{align*}
where one uses the estimate \eqref{ine-end} and the bounds in \Cref{corEnergy} in the derivation of the last step. This finishes the proof.

\end{proof}

\subsection{The uniform estimate of $u_h^n$ }
We now bound  $u_h^n$ in the $L^\infty$ norm. For this, we introduce the discrete Laplacian operator.
For any $u_h\in  W_h$, we define \\
$(\bm q_h, \Delta_h u_h, \widehat{u}_h )\in  \bm V_h\times W_h \times M_h$ such that  
\begin{align}\label{Delta_h}
(\Delta_h u_h, w_h)_{\mathcal{T}_h}=-\mathcal{A}(\bm q_h,u_h,\widehat u_h;\bm r_h,w_h,\mu_h),
\end{align}
for all $ (\bm r_h, w_h, \mu_h)\in \bm V_h\times W_h \times M_h$.
One can verify that $\Delta_h u_h$ is well-defined by the linearity of the operator $\mathcal{A}$ and the fact that uniqueness implies existence for a linear square system of finite dimension, cf. \eqref{def_A}.

\begin{lemma}\label{lap_u_h1} 
	Let $u_h^n$ be the solution of \eqref{HDG-hill}. For all $n=1,2\ldots, N$,  we have 
	\begin{align}\label{lap_u_h1-1} 
	\|\Delta_h u_h^n\|_{\mathcal{T}_h}\le C,
	\end{align}
	where $C$ depends on $\epsilon, T$ and the initial condition.
\end{lemma}
\begin{proof}
	For $u_h=u_h^n$ in Eq. \eqref{Delta_h}, by the uniqueness of the solution, in light of Eq.  \eqref{hdg02}, one identifies  that $\bm q_h=\bm q_h^n$ and $\widehat{u}_h=\widehat{u}_h^n$. Now taking $w_2=\Delta_h u_h^n$ in \eqref{hdg02},  one obtains
	\begin{align*}
	(\epsilon^{-1}f^n(u^n_h),\Delta_h u^n_h)_{\mathcal{T}_h}-\epsilon\|\Delta_h u_h^n\|^2_{\mathcal{T}_h}-(\phi_h^n,\Delta_h u^n_h)_{\mathcal{T}_h}=0.
	\end{align*}
	It follows from the Cauchy-Schwarz inequality
	\begin{align*}
	\epsilon\|\Delta_h u_h^n\|^2_{\mathcal{T}_h}&=	(\epsilon^{-1}f^n(u^n_h),\Delta_h u^n_h)_{\mathcal{T}_h}-(\phi_h^n,\Delta_h u^n_h)_{\mathcal{T}_h}\\
	&\le \epsilon^{-1}(\|u_h\|_{0,6}^3+\|u_h^n\|_{\mathcal T_h})\|\Delta_h u_h^n\|_{\mathcal{T}_h}
	+\|\phi_h^n\|_{\mathcal T_h}\|\Delta_h u_h^n\|_{\mathcal{T}_h}.
	\end{align*}
	Now, one estimates the term $\|u_h\|_{0,6}$ by the HDG Sobolev inequality \eqref{sobolev-01} and the stability bounds in \Cref{corEnergy} as follows
	\begin{align*}
	\|{u}_{h}^n\|_{0,6}&\le C \left(\|u_h^n\|_{\mathcal{T}_h}+\|\nabla u_h^n\|_{\mathcal{T}_h} +\|h_K^{-1/2}(\Pi_k^{\partial}u_h^n-\widehat u_h^n)\|_{\partial\mathcal{T}_h}\right)\\
	&\le C \left(\|u_h^n\|_{0,4}+\|\nabla u_h^n\|_{\mathcal{T}_h}+\|h_K^{-1/2}(\Pi_k^{\partial}u_h^n-\widehat u_h^n)\|_{\partial\mathcal{T}_h}\right)\\
	&\le C.
	\end{align*}
	One readily obtains the estimate \eqref{lap_u_h1-1} in light of the stability bound\eqref{phi_infty-1}. This completes the proof.
\end{proof}

We now estimate the $L^\infty$ norm of functions in $W_h$.
\begin{lemma}\label{lap_u_h2} 
	For all $w_h\in  W_h$, we have the inequality
	\begin{align*}
	&\|w_h\|_{\infty}\le C\left(h^{2-d/2}\|\Delta _h w_h\|_{\mathcal T_h}+\|\Delta_hw_h\|_{\mathcal T_h}+\|w_h\|_{\mathcal{T}_h}\right).
	\end{align*}
\end{lemma}
\begin{proof} 
	Consider the following continuous problem: find $w\in  H^1(\Omega)$ such that
	\begin{align}\label{def-w}
	-\Delta w=-\Delta_h w_h   \text{ in }\Omega, \quad \nabla w \cdot \bm n=0   \text{ on }\partial \Omega,   \quad 	(w,1)_{\mathcal{T}_h}=(w_h,1)_{\mathcal{T}_h}.
	\end{align}
	Since $\Omega$ is convex, we have the regularity  estimate
	\begin{align}\label{es-w}
	|w|_{H^2(\Omega)}=|w-\overline w|_{H^2(\Omega)}\le \|w-\overline w\|_{H^2(\Omega)}\le C_{\text{reg}}\|\Delta_h w_h\|_{\mathcal{T}_h},
	\end{align}
	where $\overline w=\frac{1}{|\Omega|}(w_h,1)_{\mathcal{T}_h}$.
	The definitions \eqref{Delta_h} and \eqref{def-w} imply
	\begin{align*}
	\mathcal A(\bm q_h^w,w_h,\widehat{u}_h^w;\bm s_h,v_h,\mu_h)=-(\Delta w,v_h)_{\mathcal{T}_h},\qquad
	(w_h-w,1)_{\mathcal{T}_h}=0,
	\end{align*}
	for all $(\bm s_h,v_h,\mu_h)\in \bm V_h\times W_h\times M_h$.
	By the uniqueness of solutions to the elliptic projection \eqref{projection}, in view of Remark \ref{rem-gen}, one uses the HDG elliptic projection result   \eqref{err-scal-1} to get
	\begin{align}
	& \|w-w_h\|_{\mathcal{T}_h}\le Ch^{2}|w|_{H^2(\Omega)}.   \label{es-fem}
	\end{align}

	By the triangle inequality, we have 
	\begin{align*}
	\|w_h\|_{0,\infty} &\le \|w_h-\Pi_{k+1}^ow\|_{L^\infty}+ \|\Pi_{k+1}^o w-w\|_{L^\infty}
	+ \|w\|_{L^\infty}\\
	&:=R_1+R_2+R_3\color{blue}. \color{black}
	\end{align*}
	The $\{R_i\}_{i=1}^3$ are estimated as follows
	\begin{align*}
	R_1&\le Ch^{-d/2}\|w_h-\Pi_{k+1}^o w\|_{\mathcal{T}_h}\le Ch^{-d/2} \left( \|w_h-w\|_{\mathcal{T}_h}+\|w-\Pi_{k+1}^o w\|_{\mathcal{T}_h}\right)\\
	&\le Ch^{2-d/2}|w|_{H^2(\Omega)}\le Ch^{2-d/2}\|\Delta_h w_h\|_{\mathcal T_h},\\
	R_2&\le Ch^{2-d/2}|w|_{H^2(\Omega)}\le  Ch^{2-d/2}\|\Delta_h w_h\|_{\mathcal T_h}, \\
	R_3&\le C ||w||_{H^2(\Omega)} \le C (||\Delta_h w_h||_{\mathcal{T}_h}+||w_h||_{\mathcal{T}_h}),
	\end{align*}
	where the last inequality follows from the elliptic regularity result.
	Collecting the above estimates, one concludes the  proof.
	
\end{proof}

Using \Cref{lap_u_h2} and \Cref{lap_u_h1} immediately gives the following result.
\begin{lemma}\label{uh_infty_bounded}
	Let $u_h^n$ be the solution of \eqref{HDG-hill}. For all $n=1,2\ldots, N$,  we have 
	\begin{align*}
	\|u_h^n\|_{L^\infty}\le C,
	\end{align*}
	where $C$ depends on $\epsilon, T$ and the initial condition.
\end{lemma}

\section{Error analysis}
\label{error_analysis}
 In this section, we provide a convergence analysis of the fully implicit HDG method for the Cahn-Hilliard equation. The convex-splitting scheme can be similarly treated. First, we give our main results. Then, we define an HDG elliptic projection as in \cite{Chen_Monk_Peter1}, which is a crucial step to prove the main result. In the end, we provide  rigorous error estimation for our fully implicit HDG method.
 
 Throughout, we assume the data and the solution of \eqref{ori} are smooth enough. As in \Cref{E_U_S}, we do not track the dependence on $\epsilon$ and treat as if $\epsilon = O(1)$. The generic constant $C$ may depend on the data of the problem but is independent of $h$ and may change from line to line.
 
 Given $\Theta\in L^2_0(\Omega)$, let $(\bm\Psi,\Phi)$ be the solution of the following system 
	\begin{align}\label{auxi-ellip}
	\bm{\Psi}+\nabla\Phi=0, \ \ 
	\nabla\cdot\bm{\Psi}=\Theta \ \text{in }\Omega, \ \ 
	\bm{\Psi}\cdot\bm n =0 \ \text{on } \partial \Omega, \ \ 
	\int_{\Omega}\Phi=0.
	\end{align}
 If $\Omega$ is convex, then we have the following regularity result:
		\begin{align}
		\|\bm\Psi\|_{H^1(\Omega)}+\|\Phi\|_{H^2(\Omega)}&\le C_{\text{reg}}\|\Theta\|_{L^2(\Omega)}.\label{reg1}
		\end{align}
\subsection{The main result}
We can now state our main result for the HDG method.
\begin{theorem}\label{main_res}
Let $(\bm p, \phi, \bm q, u)$  and $(\bm p_h^n,\phi_h^n, \bm q_h^n, u_h^n)$ be the solutions of \eqref{mixed} and \eqref{HDG-hill}, respectively. Assume the solution $(\bm p, \phi, \bm q, u)$ attains the maximum regularity for the best approximation  results in \eqref{classical_ine}. If $\Delta t \leq C \epsilon^3$ for the BE scheme and is arbitrary for the CS scheme, one has the following optimal error estimates
\begin{align}\label{main_res_estimates}
\begin{split}
\max_{1\le n\le N} \|u^n-u_h^n\|_{L^2(\Omega)}^2+\Delta t \sum_{n=1}^N \|\phi^n -\phi_h^n\|_{L^2(\Omega)}^2  \le C (h^{k+2}+\Delta t)^2, \\
\max_{1\le n\le N} \|\bm q^n -\bm q_h^n\|_{L^2(\Omega)}^2+\Delta t \sum_{n=1}^N \|\bm p^n -\bm p_h^n\|_{L^2(\Omega)}^2 \le C (h^{k+1}+\Delta t)^2.
\end{split} 
\end{align}
Furthermore, if the polynomial order $k\geq 1$, one also has the optimal error estimate in the negative norm
\begin{align}
\max_{1\le n\le N} \|u^n-u_h^n\|_{(H^1)\prime}^2 \leq C (h^{k+3}+\Delta t)^2.
\end{align}
\end{theorem}

\begin{remark}
To the best of our knowledge, \cite{Cockburn_Dong_Guzman} is the only work for  fourth order problems using an HDG method with polynomial degree $k$ for all variables.  They obtained an optimal convergence rate for the solution and suboptimal convergence rates for the other variables. In contrast, the HDG method proposed in this work  deals with a nonlinear fourth order problem and achieves  optimal convergence rates for all variables. Moreover, from the view point of degrees of freedom, we obtain the superconvergent  rate for the solution.	
\end{remark}

\subsection{The HDG elliptic projection}
For all $t\in [0,T]$, we define the HDG elliptic projection: find $(\bm p_{Ih},\phi_{Ih},\widehat{\phi}_{Ih}), \\ (\bm q_{Ih}, u_{Ih},\widehat u_{Ih})\in \bm V_h\times W_h\times M_h$ such that
\begin{subequations}\label{projection}
	\begin{align}	
	\mathcal A(\bm p_{Ih},\phi_{Ih},\widehat \phi_{Ih};\bm r_1, w_1,\mu_1)&=-(\Delta \phi,w_1)_{\mathcal{T}_h}
	&\text{and\;\;}(\phi_{I_h}-\phi,1)_{\mathcal{T}_h}=0,\label{projection01}\\
	\mathcal{A}(\bm q_{Ih},u_{Ih},\widehat u_{Ih};\bm r_2, w_2, \mu_2)&=-(\Delta u,w_2)_{\mathcal{T}_h}
	&\text{and\;\;}(u_{I_h}-u,1)_{\mathcal{T}_h}=0 \label{projection02}
	\end{align}
\end{subequations}
 for all $(\bm r_1,w_1,\mu_1), (\bm r_2, w_2,\mu_2) \in \bm V_h\times \mathring{W}_h\times M_h$. These projections are well-defined in the sense that there exist unique $(\bm p_{Ih},\phi_{Ih},\widehat{\phi}_{Ih}),  (\bm q_{Ih}, u_{Ih},\widehat u_{Ih})\in \bm V_h\times W_h\times M_h$ such that Eqs. \eqref{projection} hold.

We have the following approximation property for the HDG elliptic projection \eqref{projection}.

\begin{theorem}\label{projection_approximation} 
	Let $(\bm p, \phi, \bm q, u)$  and $(\bm p_{Ih},\phi_{Ih}, \bm q_{Ih}, u_{Ih})$ be the solution of \eqref{mixed} and \eqref{projection}, respectively. For all integer $s \in [0, k], k\geq 0 $ we have
	\begin{subequations}
	\begin{align}
	\|u-u_{Ih}\|_{\mathcal{T}_h}&\le Ch^{s+2}|u|_{H^{s+2}},  \label{projection_approximation_1}\\
	\|\bm q-\bm q_{Ih}\|_{\mathcal{T}_h}+\|h_K^{-1/2}(\Pi_k^{\partial}u_{Ih}-\widehat u_{Ih})\|_{\partial\mathcal{T}_h}&\le Ch^{s+1}|u|_{H^{s+2}},   \label{projection_approximation_2}\\
	\|\partial_tu-\partial_tu_{Ih}\|_{\mathcal{T}_h}&\le Ch^{s+2}|\partial_tu|_{H^{s+2}},  \label{projection_approximation_3}\\
	\|\phi-\phi_{Ih}\|_{\mathcal{T}_h}&\le Ch^{s+2}|\phi|_{H^{s+2}},   \label{projection_approximation_4}\\
	\|\bm p-\bm p_{Ih}\|_{\mathcal{T}_h}+\|h_K^{-1/2}(\Pi_k^{\partial}\phi_{Ih}-\widehat \phi_{Ih})\|_{\partial\mathcal{T}_h}&\le Ch^{s+1}|\phi|_{H^{s+2}}. \label{projection_approximation_5}
	\end{align}
	\end{subequations}
\end{theorem}

\begin{remark}\label{rem-gen}
In the proof of Theorem \ref{projection_approximation} below, we only make use of the regularity  of $(\bm p, \phi, \bm q, u)$ and the fact that $\bm p=-\nabla \phi, \quad \bm q=-\nabla u$. 
Hence the approximation properties in Theorem \ref{projection_approximation} are valid for any regular functions with $\bm p=-\nabla \phi, \quad \bm q=-\nabla u$.
\end{remark}

We only give a proof of \eqref{projection_approximation_1} and \eqref{projection_approximation_2}, and we split the proof into three steps. To simplify the notation, we define
\begin{align}\label{error-func}
\varepsilon_h^{\bm q}:=\bm{\Pi}_k^o\bm q-\bm q_{Ih},\ \ \
\varepsilon_h^{u}:=\Pi_{k+1}^ou-u_{Ih},\ \ \
\varepsilon_h^{\widehat u}:=\Pi_k^{\partial}u-\widehat u_{Ih}.
\end{align}	
Note that $\varepsilon_h^{\widehat u} \in M_h$ since $\Pi_k^{\partial}u=P_M u$ for $u \in H^1(\Omega)$, where $P_M$ is the $L^2$ orthogonal projection onto $M_h$. The following error estimate of $\Pi_k^{\partial}$ ($P_M$) is classical
\begin{align}\label{err-pro-M}
\|\Phi-\Pi_k^{\partial}\Phi \|_{\partial \mathcal{T}_h} \leq C h^{s+1/2} |\Phi|_{H^{s+1}}, \quad s \in [0, k].
\end{align}
	
\subsubsection{Step 1: The error equation}
\begin{lemma}\label{projection_error} 
	For all $(\bm r_2, w_2,\mu_2)\in \bm V_h\times \mathring W_h\times M_h$, we have 
	\begin{align}
	&\mathcal{A}(\bm{\Pi}_k^o\bm q, \Pi_{k+1}^ou,\Pi_k^{\partial} u;\bm r_2, w_2, \mu_2)=-(\Delta u, w_2)_{\mathcal{T}_h} \nonumber \\
	&+\langle\bm q\cdot\bm n - \bm{\Pi}_k^o\bm q\cdot\bm n,\mu_2- w_2\rangle_{\partial\mathcal{T}_h} +\langle h_K^{-1}(\Pi_{k+1}^ou-u),\Pi_k^{\partial}w_2-\mu_2\rangle_{\partial\mathcal{T}_h}. \label{projection_error-1} 
	\end{align}
	
\end{lemma}
\begin{proof} By the definition of $\mathcal{A}$ in \eqref{def_A}, we have 
	\begin{align*}
	\hspace{1em}&\hspace{-1em}\mathcal{A}(\bm{\Pi}_k^o\bm q, \Pi_{k+1}^ou,\Pi_k^{\partial} u;\bm r_2, w_2, \mu_2)\\
	&=(\bm{\Pi}_k^o\bm q,\bm r_2)_{\mathcal{T}_h}-(\Pi_{k+1}^ou,\nabla\cdot\bm r_2)_{\mathcal{T}_h}
	+\langle\Pi_k^{\partial}u,\bm r_2\cdot\bm n\rangle_{\partial\mathcal{T}_h}\\
	&\quad+(\nabla\cdot\bm{\Pi}_k^o\bm q,w_2)_{\mathcal{T}_h}-\langle\bm{\Pi}_k^o\bm q\cdot\bm n, \mu_2\rangle_{\partial\mathcal{T}_h}\\
	&\quad +\langle h_K^{-1}(\Pi_k^{\partial}\Pi_{k+1}^ou-\Pi_k^{\partial}u),\Pi_k^{\partial}w_2-\mu_2
	\rangle_{\partial\mathcal{T}_h}\\
	&=(\bm q+\nabla u,\bm r_2)_{\mathcal{T}_h} - (\bm q,\nabla w_2)_{\mathcal{T}_h}
	-\langle \bm{\Pi}_k^o\bm q\cdot\bm n,\mu_2-w_2\rangle_{\partial\mathcal{T}_h}\\
	&\quad+\langle h_K^{-1}(\Pi_{k+1}^ou-u),\Pi_k^{\partial}w_2-\mu_2\rangle_{\partial\mathcal{T}_h},
	\end{align*}
	where we used the orthogonality of $\bm{\Pi}_k^o$, $\Pi_{k+1}^o$, $\Pi_k^{\partial}$ in the last equality. Since $\bm q=-\nabla u$ and  $\langle \bm q\cdot\bm n,\mu_2\rangle_{\partial\mathcal{T}_h}=0$,
	one gets
	\begin{align*}
&	\mathcal{A}(\bm{\Pi}_k^o\bm q, \Pi_{k+1}^ou,\Pi_k^{\partial} u;\bm r_2, w_2, \mu_2)
	=-(\Delta u, w_2)_{\mathcal{T}_h}\\
&	+\langle\bm q\cdot\bm n - \bm{\Pi}_k^o\bm q\cdot\bm n,\mu_2-w_2\rangle_{\partial\mathcal{T}_h}
	+\langle h_K^{-1}(\Pi_{k+1}^ou-u),\Pi_k^{\partial}w_2-\mu_2\rangle_{\partial\mathcal{T}_h}
	\end{align*}
This completes the proof.	 
\end{proof}

Subtracting Eq. \eqref{projection_error-1} and Eq. \eqref{projection02}  gives the error equation.
\begin{lemma} 
	For all $(\bm r_2, w_2,\mu_2)\in \bm V_h\times \mathring W_h\times M_h$, we have 
	\begin{align}\label{app:error}
	\begin{split}
	\mathcal{A}(\varepsilon_h^{\bm q}, \varepsilon_h^{u},\varepsilon_h^{\widehat u};\bm r_2, w_2,\mu_2) &= \langle \bm q\cdot\bm n-\bm{\Pi}_k^o\bm q\cdot\bm n,\mu_2- w_2\rangle_{\partial\mathcal{T}_h}\\
	&\quad +\langle
	h_K^{-1}(\Pi_{k+1}^ou-u),\Pi_k^{\partial}w_2-\mu_2\rangle_{\partial\mathcal{T}_h}.
	\end{split}
	\end{align}
\end{lemma}

\subsubsection{Step 2: An energy argument}

\begin{lemma}\label{Energy_Argum}
	Let $(\bm q, u)$ and $(\bm q_{Ih}, u_{Ih})$ be the solution of \eqref{ori} and \eqref{projection02}, respectively.  The following error estimate holds for $s\in[0, k]$ and $k\geq 0$.
	\begin{align}
&	\|\varepsilon_h^{\bm q}\|_{\mathcal{T}_h}+\|h_K^{-1/2}(\Pi_k^\partial \varepsilon_h^{u}-\varepsilon_h^{\widehat u}   )\|_{\partial\mathcal{T}_h} \leq C h^{s+1}|u|_{H^{s+2}}, \label{Energy_Argum-1pre} \\
	&\|h_K^{-1/2}( \varepsilon_h^{u}-\varepsilon_h^{\widehat u}   )\|_{\partial\mathcal{T}_h} \leq C h^{s+1}|u|_{H^{s+2}}. \label{Energy_Argum-inp1} 
	\end{align}
	In particular, one has
	\begin{align}
	&\|\bm q-\bm q_{Ih}\|_{\mathcal{T}_h}+\|h_K^{-1/2}(\Pi_k^{\partial}u_{Ih}-\widehat u_{Ih} )\|_{\partial\mathcal{T}_h}\le Ch^{s+1} |u|_{H^{s+2}} \label{Energy_Argum-1}.
	\end{align}
\end{lemma}
\begin{proof}

First, the error equation \eqref{app:error} implies that 
\begin{align*}
\mathcal{A}(\varepsilon_h^{\bm q}, \varepsilon_h^{u},\varepsilon_h^{\widehat u};\bm r_2, 0,0)=0, \quad \forall \bm r_h \in \bm V_h.
\end{align*}
Hence Lemma \ref{es_u} gives 
\begin{align}\label{ineq-inter1}
||h_K^{-1/2}(\varepsilon_h^{u}-\varepsilon_h^{\widehat u})||_{\partial \mathcal{T}_h} \leq C(||\varepsilon_h^{\bm q}||_{\mathcal{T}_h} + ||h_K^{-1/2}(\Pi_k^{\partial}\varepsilon_h^{u}-\varepsilon_h^{\widehat u})||_{\partial \mathcal{T}_h}).
\end{align}
Noting that $\varepsilon_h^{u} \in \mathring{W}_h$ by the definitions of $\Pi_{k+1}^o$ and $u_{Ih}$, we now take $(\bm r_2, w_2, \mu_2) =(\varepsilon_h^{\bm q}, \varepsilon_h^{u},\varepsilon_h^{\widehat u})
	$ in \eqref{app:error}. Then by the Cauchy-Schwartz inequality, the triangle inequality, the inequality \eqref{ineq-inter1} and inequalities in \eqref{classical_ine}, one  get
	\begin{align*}
	&\|\varepsilon_h^{\bm q}\|^2_{\mathcal{T}_h}+\|h_K^{-1/2}(\Pi_k^{\partial}\varepsilon_h^{u}-\varepsilon_h^{\widehat u}   )\|^2_{\partial\mathcal{T}_h}=\langle \bm q\cdot\bm n - \bm{\Pi}_k^o\bm q\cdot\bm n,\varepsilon_h^{\widehat u}-\Pi_k^\partial \varepsilon_h^u
	\rangle_{\partial\mathcal{T}_h}\\
	&\langle \bm q\cdot\bm n - \bm{\Pi}_k^o\bm q\cdot\bm n, \Pi_k^\partial \varepsilon_h^u-\varepsilon_h^u \rangle_{\partial\mathcal{T}_h}+\langle h_K^{-1}(\Pi_{k+1}^ou-u),\Pi_k^{\partial}\varepsilon_h^u-\varepsilon_h^{\widehat u}
	\rangle_{\partial\mathcal{T}_h}\\
	&\le Ch^{s+1}(|u|_{H^{s+2}}+|\bm q|_{H^{s+1}}) \left( \|\varepsilon_h^{\bm q}\|^2_{\mathcal{T}_h}+\|h_K^{-1/2}(\Pi_k^{\partial}\varepsilon_h^{u}-\varepsilon_h^{\widehat u}   )\|^2_{\partial\mathcal{T}_h}\right)^{1/2}.
	\end{align*}
The error estimate \eqref{Energy_Argum-1pre}	 readily follows. Then the estimate \eqref{Energy_Argum-inp1} is a consequence of the inequality \eqref{ineq-inter1}.
	
Now in light of the definitions of the error functions in \eqref{error-func}, one obtains  the desired error estimate \eqref{Energy_Argum-1} by the triangle inequality, the $L^2$ stability of the projection $\Pi_k^\partial$,  the inequalities in \eqref{classical_ine} and the fact that $\bm q=-\nabla u$. This completes the proof.
\end{proof}

\subsubsection{Step 3: The  error estimate of the scalar variable by the duality argument}
Similar to \Cref{projection_error} we have the following result.

\begin{lemma}
	Let $\Theta$ be in $\mathring{W}_h$, and let $(\bm{\Psi}, \phi)$ be the solution to the system \eqref{auxi-ellip}. Then for all $(\bm r_2, w_2, \mu_2)\in \bm V_h\times \mathring{W}_h\times M_h$, we have the equation
	\begin{align}\label{app:pi-dual}
	\begin{split}
&	\mathcal{A}(\bm{\Pi}_k^o\bm{\Psi},\Pi_{k+1}^o\Phi,\Pi_k^{\partial}\Phi;\bm r_2, w_2, \mu_2)
	=(\Theta,w_2)_{\mathcal T_h}  \\
&	+\langle \bm \Psi\cdot\bm n - \bm{\Pi}_k^o\bm \Psi\cdot\bm n,\mu_2- w_2 \rangle_{\partial\mathcal{T}_h}
 + \langle	h_K^{-1}(\Pi_{k+1}^o\Phi-\Phi),\Pi_k^{\partial}w_2-\mu_2
	\rangle_{\partial\mathcal{T}_h}.
	\end{split}
	\end{align}
\end{lemma}

\begin{lemma} Let $u$ and $u_{Ih}$ be the solutions of \eqref{ori} and \eqref{projection02}, respectively.  Then  for $s\in[0, k]$, $k\geq 0$ we have the error estimates
	\begin{align}
&	\|u-u_{Ih}\|_{\mathcal{T}_h}\le Ch^{s+2}|u|_{H^{s+2}} \label{err-scal-1}. 
	\end{align}
\end{lemma}
\begin{proof} 
	We take $(\bm r_2, w_2,\mu_2)=(\varepsilon_h^{\bm q},-\varepsilon_h^u,-\varepsilon_h^{\widehat u})$ and $\Theta=-\varepsilon_h^u$ in \eqref{app:pi-dual} to get
	\begin{align*}
	\|\varepsilon_h^u\|^2_{\mathcal{T}_h}
	&=	\mathcal{A}(\bm{\Pi}_k^o\bm{\Psi},\Pi_{k+1}^o\Phi,\Pi_k^{\partial}\Phi;\varepsilon_h^{\bm q},-\varepsilon_h^u,-\varepsilon_h^{\widehat u} )
	 -\langle\bm \Psi\cdot\bm n- \bm{\Pi}_k^o\bm \Psi\cdot\bm n, \varepsilon_h^{\widehat u}-\varepsilon_h^u
	\rangle_{\partial\mathcal{T}_h}\\
	&\quad 
	 - \langle h_K^{-1} (\Pi_{k+1}^o\Phi-\Phi),\Pi_k^{\partial}\varepsilon_h^u-\varepsilon_h^{\widehat u}\rangle_{\partial\mathcal{T}_h}.
	\end{align*}
	By  \Cref{summertic_A},  the error equation \eqref{app:error},  we have 
	\begin{align*}
	&\|\varepsilon_h^u\|^2_{\mathcal{T}_h} \\
	&=	\mathcal{A}(\varepsilon_h^{\bm q},\varepsilon_h^u,\varepsilon_h^{\widehat u};\bm{\Pi}_k^o\bm{\Psi},-\Pi_{k+1}^o\Phi,-\Pi_k^{\partial}\Phi; )
	-\langle\bm \Psi\cdot\bm n- \bm{\Pi}_k^o\bm \Psi\cdot\bm n, \varepsilon_h^{\widehat u}-\varepsilon_h^u
	\rangle_{\partial\mathcal{T}_h}\\
	&\quad
	 - \langle h_K^{-1} (\Pi_{k+1}^o\Phi-\Phi),\Pi_k^{\partial}\varepsilon_h^u-\varepsilon_h^{\widehat u}\rangle_{\partial\mathcal{T}_h} \\
	&=	-\langle \bm{\Pi}_k^o\bm q\cdot\bm n-\bm q\cdot\bm n,\Pi_k^{\partial}\Phi-\Pi_{k+1}^{o}\Phi)
	\rangle_{\partial\mathcal{T}_h}- \langle h_K^{-1}(\Pi_{k+1}^ou-u),\Pi_k^{\partial}\Pi_{k+1}^{o}\Phi -\Pi_k^{\partial}\Phi
	\rangle_{\partial\mathcal{T}_h}\\	
	&\quad 
	 -\langle \bm \Psi\cdot\bm n-\bm{\Pi}_k^o\bm \Psi\cdot\bm n,\varepsilon_h^{\widehat u}- \varepsilon_h^u
	\rangle_{\partial\mathcal{T}_h}
- \langle h_K^{-1}(\Pi_{k+1}^o\Phi-\Phi),\Pi_k^{\partial}\varepsilon_h^u-\varepsilon_h^{\widehat u}
	\rangle_{\partial\mathcal{T}_h}.
	\end{align*}

Since $\Pi_k^\partial \Phi$ is single valued, one has $\langle\bm q\cdot\bm n,\Pi_k^{\partial}\Phi\rangle_{\partial\mathcal{T}_h}=\langle\bm q\cdot\bm n,\Phi\rangle_{\partial\mathcal{T}_h}=0$. Hence
\begin{align*}
\langle \bm{\Pi}_k^o\bm q\cdot\bm n-\bm q\cdot\bm n,\Pi_k^{\partial}\Phi-\Pi_{k+1}^{o}\Phi)
	\rangle_{\partial\mathcal{T}_h}=\langle \bm{\Pi}_k^o\bm q\cdot\bm n-\bm q\cdot\bm n,\Phi-\Pi_{k+1}^{o}\Phi)
	\rangle_{\partial\mathcal{T}_h}
\end{align*}
It follows that
\begin{align*}
|\langle \bm{\Pi}_k^o\bm q\cdot\bm n-\bm q\cdot\bm n,\Phi-\Pi_{k+1}^{o}\Phi)
	\rangle_{\partial\mathcal{T}_h}| &\leq  \|\bm{\Pi}_k^o\bm q\cdot\bm n-\bm q\cdot\bm n\|_{\partial \mathcal{T}_h} \|\Phi-\Pi_{k+1}^{o}\Phi \|_{\partial \mathcal{T}_h} \\
	&\leq C h^{-1/2}  \|\bm{\Pi}_k^o\bm q-\bm q\|_{ \mathcal{T}_h} h^{3/2} |\Phi|_2 \\
	&\leq Ch \|\bm{\Pi}_k^o\bm q-\bm q\|_{ \mathcal{T}_h} \| \varepsilon_h^u\|_{ \mathcal{T}_h},
\end{align*}
where the regularity result \eqref{reg1} with $\Theta=-\varepsilon_h^u$ is used 	in the derivation of the last inequality. 
The rest of the terms can be dealt with similarly by using  the $L^2$ stability of the projection $\Pi_k^\partial$, and the inequalities in \eqref{classical_ine}.

The desired error estimate \eqref{err-scal-1} now follows from the error estimates \eqref{Energy_Argum-1pre} and \eqref{Energy_Argum-inp1}, and the triangle inequality. This completes the proof.

\end{proof}

\subsection{Error estimate in the negative norm}

To establish the approximation properties of the elliptic projection in the negative norm, we  introduce a Scott-Zhang type (cf. \cite{ScZh1990}) interpolation operator $\mathcal I_h^{\color{black}k+2+d}$ in \Cref{app}. For all $(u_h,\widehat u_h)\in L^2(\mathcal T_h)\times L^2(\partial\mathcal T_h)$,  $\mathcal I_h^{\color{black}k+2+d}(u_h,\widehat u_h)|_K\in \mathcal {P}^{\color{black}k+2+d}(K)$ and satisfies
\begin{subequations}
	\begin{align}
	(\mathcal{I}_h^{\color{black}k+2+d}(u_h,\widehat u_h),w_h)_{K}&= (u_h,w_h)_{K} &\text{for all }w_h\in \mathcal{P}^{k+1}(K),\label{pro_K}\\
	\color{black}\langle\mathcal{I}_h^{\color{black}k+2+d}(u_h,\widehat u_h),\mu_h \rangle_{F}&=\langle 
	\widehat u_h,\mu_h\rangle_F&\text{for all }\mu_h\in \mathcal{P}^{k+2}(F), F \color{black}\subset\color{black} \partial K.\label{pro_F}
	\end{align}
\end{subequations}

\begin{theorem} Let $(\phi, u)$ and $(\phi_{Ih},u_{Ih})$ be the solution of \eqref{ori} and \eqref{projection01}-\eqref{projection02}, respectively. Then \color{black} if $k\ge 1$, \color{black} we have the following error estimates
	\begin{subequations}\label{nega_error}
		\begin{align}
		\|\Pi_{k+1}^ou-u_{Ih}\|_{-1,h}\le Ch^{k+3}|u|_{H^{k+2}},\label{nega_error_1}\\
		\|\Pi_{k+1}^o\phi-\phi_{Ih}\|_{-1,h}\le Ch^{k+3}|\phi|_{H^{k+2}}.\label{nega_error_2}
		\end{align}
	\end{subequations}
\end{theorem}

\begin{proof}
	We only give a proof for \eqref{nega_error_1}, the proof for \eqref{nega_error_2} is similar.
	
	Let $\xi_h^u = \Pi_{k+1}^ou-u_{Ih}$, by \Cref{-1h} and \eqref{pro_K} one gets
	\begin{align}\label{proof:-1}
	\|\xi_h^u\|^2_{-1,h}=(\Pi_W\xi_h^u,\xi_h^u)_{\mathcal{T}_h}=(\mathcal{I}_h^{\color{black}k+2+d}(\Pi_W\xi_h^u,\Pi_M \xi_h^u),\xi_h^u)_{\mathcal{T}_h}.
	\end{align}
	We take $(\bm r_2,w_2,\mu_2)=(\xi_h^{\bm q},-\xi_h^u,-\xi_h^{\widehat u})$ and $\Theta=-\mathcal{I}_h^{\color{black}k+2+d}(\Pi_W\xi_h^u,\Pi_M \xi_h^u)$ in \eqref{app:pi-dual},
	and use \eqref{proof:-1} 
	to get
	\begin{align*}
&	\|\xi_h^u\|^2_{-1,h}=\mathcal{A}(\bm{\Pi}_k^o\bm{\Psi},\Pi_{k+1}^o\Phi,\Pi_k^{\partial}\Phi;\xi_h^{\bm q},-\xi_h^u,-\xi_h^{\widehat u} ) -\langle\bm{\Pi}_k^o\bm \Psi\cdot\bm n-\bm \Psi\cdot\bm n,\xi_h^{\widehat u}-\xi_h^u
	\rangle_{\partial\mathcal{T}_h}\\
	&\quad - \langle h_K^{-1}(\Pi_{k+1}^o\Phi-\Phi),\Pi_k^{\partial}\xi_h^u-\xi_h^{\widehat u}
	\rangle_{\partial\mathcal{T}_h}\\
	&=-	\langle \bm{\Pi}_k^o\bm q\cdot\bm n-\bm q\cdot\bm n,\Pi_k^{\partial}\Phi-\Pi_{k+1}^{o}\Phi \rangle_{\partial\mathcal{T}_h} - \langle h_K^{-1}(\Pi_{k+1}^ou-u),\Pi_k^{\partial}\Pi_{k+1}^{o}\Phi-\Pi_k^{\partial}\Phi
	\rangle_{\partial\mathcal{T}_h} \\
	 	&\quad-\langle \bm{\Pi}_k^o\bm \Psi\cdot\bm n-\bm \Psi\cdot\bm n,\xi_h^{\widehat u}-\xi_h^u
	\rangle_{\partial\mathcal{T}_h}
 - \langle h_K^{-1}(\Pi_{k+1}^o\Phi-\Phi),\Pi_k^{\partial}\xi_h^u-\xi_h^{\widehat u}
	\rangle_{\partial\mathcal{T}_h}\\
	&\le Ch^{k+3}(\|\Phi\|_{H^3(\Omega)} +\|\bm \Psi\|_{H^2(\Omega)}).
	\end{align*}
	Hence
	\begin{align}\label{-1leIh}
	\|\xi_h^u\|^2_{-1,h}\le Ch^{k+3}|u|_{k+2}\|\mathcal{I}_h^{\color{black}k+2+d}(\Pi_W\xi_h^u,\Pi_M \xi_h^u)\|_{H^1(\Omega)}.
	\end{align}
	Since $\Pi_W \xi_h^u \in \mathring{W}_h$, and  by the $H^1$ stability of the interpolation operator    \eqref{IC} one has
	\begin{align*}
	\hspace{1em}&\hspace{-1em}\|\mathcal{I}_h^{\color{black}k+2+d}(\Pi_W\xi_h^u,\Pi_M \xi_h^u)\|_{H^1(\Omega)}\\
	&\le C\left(
	\color{black}
	\|\Pi_W\xi_h^u\|_{\mathcal T_h}+
	\color{black}
	\|\nabla \Pi_W\xi_h^u\|_{\mathcal{T}_h}+\|h_K^{-1/2}( \Pi_W\xi_h^u-\Pi_M \xi_h^u    )\|_{\partial\mathcal{T}_h}\right)\\
	&\le C\left(\|\nabla \Pi_W\xi_h^u\|_{\mathcal{T}_h}+\|h_K^{-1/2}(\Pi_k^{\partial} \Pi_W\xi_h^u-\Pi_M\xi_h^u)\|_{\partial\mathcal{T}_h}\right) \\
	&\le C(\|\bm\Pi_{\bm V} \xi_h^u\|_{\mathcal T_h} + \|h_K^{-1/2}(\Pi_k^\partial \Pi_W \xi_h^u - \Pi_M \xi_h^u)\|_{\partial \mathcal T_h})\\
	&\le C\|\xi_h^u\|_{-1,h},
	\end{align*}
	\color{black} 
	which combining \eqref{-1leIh}, then implies
	\color{black}
	\begin{align*}
	\|\xi_h^u\|_{-1,h}\le Ch^{k+3}|u|_{H^{k+2}}.
	\end{align*}
	This completes the proof.
\end{proof}

In a similar fashion as \Cref{uh_infty_bounded} one can establish the stability bound of $u_{Ih}$ in the uniform norm.
\begin{lemma}\label{uih_infty_bounded}
Let $u_{Ih}$ be the solution to the elliptic projection \eqref{projection02}.  Assume $u \in L^\infty(0,T; H^2(\Omega))$. Then one has
	\begin{align*}
	\|u_{Ih}^n\|_{0, \infty}\le C,
	\end{align*}
	where $C$ depends on $\|u\|_{L^\infty(H^2)}$.
\end{lemma}

\subsection{Proof of \Cref{main_res}}
To simplify  notation, we define
\begin{subequations}
	\begin{align}
	e_h^{\bm p^n}:=\bm p_{Ih}^n-\bm p_h^n,\ \ \ e_h^{\phi^n}:=\phi_{Ih}^n-\phi_h^n,\ \ \ e_h^{\widehat \phi^n}:=\widehat \phi_{Ih}^n-\widehat \phi_{h}^n, \label{def_e1}\\
	e_h^{\bm q^n}:=\bm q_{Ih}^n-\bm q_h^n,\ \ \ e_h^{u^n}:=u_{Ih}^n-u_h^n,\ \ \ e_h^{\widehat u^n}:=\widehat u_{Ih}^n-\widehat u_{h}^n. \label{def_e2}
	\end{align}
\end{subequations}

\begin{lemma} 
	For all $(\bm r_1, w_1, \mu_1),(\bm r_2, w_2, \mu_2)\in \bm V_h\times W_h\times M_h$, we have 
	the following error equations
	\begin{subequations}
		\begin{align}
		\hspace{1em}&\hspace{-1em} (\partial_{t}^+ e_h^{u^n},w_1)_{\mathcal{T}_h}+
		\mathcal A(e_h^{\bm p^n},e_h^{\phi^n},e_h^{\widehat \phi^n};\bm r_1, w_1, \mu_1)\nonumber\\
		&=(\partial_{t}^+ u_{Ih}^n-\partial_tu^n,w_1)_{\mathcal{T}_h},\label{error_01}\\
		\hspace{1em}&\hspace{-1em} \epsilon\mathcal{A}(e_h^{\bm q^n},e_h^{u^n},e_h^{\widehat u^n};\bm r_2, w_2, \mu_2)_{\mathcal{T}_h}-(e_h^{\phi^n}, w_2)_{\mathcal T_h}\nonumber\\
		&=(\phi^n-\phi^n_{Ih},w_2)_{\mathcal{T}_h} +\epsilon^{-1}(f^n(u_h^n)-f(u^n),w_2)_{\mathcal{T}_h}.\label{error_02}
		\end{align}
	\end{subequations}
\end{lemma}
\begin{proof}
	We use  the definition of $\mathcal A$ in \eqref{def_A} to get 
	\begin{align*}
	\hspace{1em}&\hspace{-1em}(\partial_{t}^+ e_h^{u^n},w_1)_{\mathcal{T}_h}+
	\mathcal A(e_h^{\bm p^n},e_h^{\phi^n},e_h^{\widehat \phi^n};\bm r_1, w_1, \mu_1)\\
	&=(\partial_{t}^+ u_{Ih}^n,w_1)_{\mathcal{T}_h}
	+\mathcal A(\bm p_{Ih}^n,\phi_{Ih}^n,\widehat \phi_{Ih}^n;\bm r_1, w_1, \mu_1)\\
	&\quad-(\partial_{t}^+ u_{h}^n,w_1)_{\mathcal{T}_h}
	-\mathcal A(\bm p_{h}^n,\phi_{h}^n,\widehat \phi_{h}^n;\bm r_1, w_1, \mu_1) & \textup{by }\eqref{def_e1}\\
	&=(\partial_{t}^+u_{Ih}^n,w_1)_{\mathcal{T}_h}-(\Delta \phi^n,w_1)_{\mathcal{T}_h}
	&\textup{by } \eqref{projection01}, \eqref{hdg01}\\
	&=(\partial_{t}^+u_{Ih}^n-\partial_tu^n,w_1)_{\mathcal{T}_h}& \textup{by } \eqref{o1}.
	\end{align*}
	Next, we have
	\begin{align*}
	\hspace{1em}&\hspace{-1em} \epsilon\mathcal{A}(e_h^{\bm q^n},e_h^{u^n},e_h^{\widehat u^n};\bm r_2, w_2,\mu_2)_{\mathcal{T}_h}-(e_h^{\phi^n}, w_2)_{\mathcal T_h}\\
	&=\epsilon\mathcal{A}(\bm q^n_{Ih},u^n_{Ih},\widehat u^n_{Ih};\bm r_2, w_2, \mu_2)_{\mathcal{T}_h}-(\phi_{Ih}^n, w_2)_{\mathcal T_h}\\
	&\quad-\epsilon\mathcal{A}(\bm q^n_{h},u^n_{h},\widehat u^n_{h};\bm r_2, w_2, \mu_2)_{\mathcal{T}_h}+(\phi_h^n, w_2)_{\mathcal T_h}
	&\textup{by }\eqref{def_e2}\\
	&=-\epsilon(\Delta u^n,w_2)_{\mathcal{T}_h} - (\phi^n_{Ih},w_2)_{\mathcal{T}_h}+\epsilon^{-1}(f^n(u_h^n),w_2)_{\mathcal{T}_h}
	&\textup{by }\eqref{projection01}, \eqref{hdg02}\\
	&=   (\phi^n-\phi^n_{Ih},w_2)_{\mathcal{T}_h} +\epsilon^{-1}(f^n(u_h^n)-f(u^n),w_2)_{\mathcal{T}_h} &\textup{by }\eqref{o2}.
	\end{align*}
	This  completes the proof.
\end{proof}

We start the error analysis in the negative norm. We have
\begin{lemma}[Error estimates in the $-1$ norm]\label{Err-neg}
The following error estimates hold
	\begin{align}
	 &\max_{1\le n\le N}\|e_h^{u^{n}}\|^2_{-1,h}+2\epsilon\Delta t \sum_{n=1}^{\color{black}N}
	\left(\|e_h^{\bm q^n}\|^2_{\mathcal T_h}+2\|h_K^{-1/2}( \Pi_k^{\partial}e_h^{u^n}-e_h^{\widehat u^n})\|^2_{\partial \mathcal{T}_h}\right) \nonumber \\
	& +(\Delta t)^2\sum_{n=1}^{\color{black}N}\|\partial_{t}^+e_h^{u^n}\|^2_{-1,h} \le  \left\{
	\begin{array}{ll}
	 C\big((\Delta t)^2+h^{2(k+3)}\big) &k\geq1, \label{Err-neg1} \\
	C\big((\Delta t)^2+h^{4}\big) & k=0.
	\end{array}
	\right.
	\end{align}
\end{lemma}
\begin{proof}
	Taking $(\bm r_1, w_1, \mu_1) =( -\bm{\Pi_V}e_h^{u^n},\Pi_We_h^{u^n},\Pi_M e_h^{u^n})$ in \eqref{error_01} gives
\begin{align}\label{proof_1_sub}
\begin{split}
& (\partial_{t}^+ e_h^{u^n},\Pi_We_h^{u^n})_{\mathcal{T}_h}+
\mathcal A(e_h^{\bm p^n},e_h^{\phi^n},e_h^{\widehat \phi^n};-\bm{\Pi_V}e_h^{u^n},\Pi_We_h^{u^n},\Pi_M e_h^{u^n})\\
&=(\partial_{t}^+ u_{Ih}^n-\partial_tu^n,\Pi_We_h^{u^n})_{\mathcal{T}_h}.
\end{split}
\end{align}
 \Cref{summertic_A} and \Cref{-1h} imply 
\begin{align}\label{proof_2_sub}
\begin{split}
&\mathcal A(e_h^{\bm p^n},e_h^{\phi^n},e_h^{\widehat \phi^n};-\bm{\Pi_V}e_h^{u^n},\Pi_We_h^{u^n},\Pi_M e_h^{u^n}) \\
&= \mathcal A(\bm{\Pi_V}e_h^{u^n},\Pi_We_h^{u^n},\Pi_M e_h^{u^n}; -e_h^{\bm p^n},e_h^{\phi^n},e_h^{\widehat \phi^n})\\
& = (e_h^{u^n},e_h^{\phi^n})_{\mathcal T_h}.
\end{split}
\end{align}	
Hence
\begin{align}\label{proof_3_sub}
(\partial_{t}^+ e_h^{u^n},\Pi_We_h^{u^n})_{\mathcal{T}_h}+
(e_h^{u^n},e_h^{\phi^n})_{\mathcal{T}_h}=(\partial_{t}^+ u_{Ih}^n-\partial_tu^n,\Pi_We_h^{u^n})_{\mathcal{T}_h}.
\end{align}
Next, one takes $(\bm r_2, w_2, \mu_2)=(e_h^{\bm q^n},e_h^{u^n},e_h^{\widehat u^n})$ in \eqref{error_02} to get 
\begin{align}\label{proof_4_sub}
\begin{split}
\hspace{1em}&\hspace{-1em}\epsilon\|e_h^{\bm q^n}\|^2_{\mathcal T_h}+\epsilon\|h_K^{-1/2}( \Pi_k^{\partial}e_h^{u^n}-e_h^{\widehat u^n})\|^2_{\partial\mathcal{T}_h}-(e_h^{\phi^n},e_h^{u^n})_{\mathcal{T}_h}\\
&+\epsilon^{-1}(f(u^n)-f^n(u_h^n),e_h^{u^n})_{\mathcal{T}_h}
=(\phi^n-\phi^n_{Ih},e_h^{u^n})_{\mathcal{T}_h}.
\end{split}
\end{align}
Combining Eqs.  \eqref{proof_3_sub} and \eqref{proof_4_sub} one obtains
\begin{align}
&(\partial_{t}^+ e_h^{u^n},\Pi_We_h^{u^n})_{\mathcal{T}_h}+\epsilon\|e_h^{\bm q^n}\|^2_{ \mathcal T_h} +\epsilon\|h_K^{-1/2}( \Pi_k^{\partial}e_h^{u^n}-e_h^{\widehat u^n}    )\|^2_{\partial\mathcal{T}_h} =  (\phi^n-\phi^n_{Ih},e_h^{u^n})_{\mathcal{T}_h}\nonumber \\
& +(\partial_{t}^+ u_{Ih}^n-\partial_tu^n,\Pi_We_h^{u^n})_{\mathcal{T}_h}+\epsilon^{-1}\big(f^n(u_h^n)-f(u^n),e_h^{u^n}\big)_{\mathcal{T}_h}. \label{err-neg-eq1}
\end{align}	

We first calculate the first term on the left hand side of Eq. \eqref{err-neg-eq1}.
 Utilizing \Cref{-1h}  and \eqref{def_A}, one has  
\begin{align} \label{c-1}
\hspace{1em}&\hspace{-1em}(\partial_{t}^+ e_h^{u^n},\Pi_We_h^{u^n})_{\mathcal{T}_h}\nonumber\\
&=\mathcal A(\bm{\Pi_V}\partial_{t}^+ e_h^{u^n},\Pi_W\partial_{t}^+e_h^{u^n},\Pi_M \partial_{t}^+ e_h^{u^n};\bm 0,\Pi_We_h^{u^n},\Pi_Me_h^{u^n})\nonumber\\
&=(\nabla\cdot\bm{\Pi_V}\partial_{t}^+ e_h^{u^n},\Pi_We_h^{u^n} )_{\mathcal{T}_h}
-\langle\bm n\cdot \bm{\Pi_V}\partial_{t}^+ e_h^{u^n},\Pi_M e_h^{u^n}
\rangle_{\partial\mathcal{T}_h}\\
&\quad+\langle h_K^{-1}(\Pi_k^{\partial}\Pi_We_h^{u^n}-\Pi_M e_h^{u^n}),\partial_{t}^+ (\Pi_k^{\partial}\Pi_We_h^{u^n}-\Pi_Me_h^{u^n}) \rangle_{\partial\mathcal{T}_h}.\nonumber
\end{align}
On the other hand, 
\begin{align}\label{proof_5_sub}
\mathcal A(\bm{\Pi_V} e_h^{u^n},\Pi_We_h^{u^n},\Pi_M e_h^{u^n};\bm r_h,w_h, \mu_h) = ( e_h^{u^n}, w_h)_{\mathcal T_h}.
\end{align}
Hence  taking $\bm r_h = \bm{\Pi_V}\partial_{t}^+e_h^{u^n}$, $w_h=\mu_h=0$ in \eqref{proof_5_sub} yields 
\begin{align*}
& ({\color{black}\Pi_W} e_h^{u^n}, \nabla\cdot\bm{\Pi_V} \partial_{t}^+ e_h^{u^n})_{\mathcal T_h} - \langle \Pi_M  e_h^{u^n}, \bm n\cdot  \bm{\Pi_V} \partial_{t}^+ e_h^{u^n}\rangle_{\partial \mathcal T_h} =(\bm{\Pi_V} e_h^{u^n},  \partial_{t}^+ \bm{\Pi_V} e_h^{u^n})_{\mathcal T_h}.  
\end{align*}
Eq. \eqref{c-1} now becomes
\begin{align}
&(\partial_{t}^+ e_h^{u^n},\Pi_We_h^{u^n})_{\mathcal{T}_h} \nonumber\\
&=(\partial_{t}^+ \bm{\Pi_V}e_h^{u^n},\bm{\Pi_V}e_h^{u^n})_{\mathcal{T}_h}+\langle h_K^{-1}(\Pi_k^{\partial}\Pi_We_h^{u^n}-\Pi_Me_h^{u^n}),\partial_{t}^+(\Pi_k^{\partial}\Pi_We_h^{u^n}-\Pi_Me_h^{u^n}) \rangle_{\partial\mathcal{T}_h} \nonumber \\
&=\frac{1}{2}\partial_t^+\|e_h^{u^n}\|_{-1,h}^2+\frac{\Delta t}{2}\|\partial_{t}^+e_h^{u^n}\|_{-1,h}^2, \label{temp-neg1}
\end{align}
where one has utilized  Eq. \eqref{-1_equal}.
Substituting \eqref{temp-neg1} into the error equation \eqref{err-neg-eq1}, one has 
\begin{align}\label{err-neg-eq2}
&\frac{1}{2}\partial_t^+\|e_h^{u^n}\|_{-1,h}^2+\frac{\Delta t}{2}\|\partial_{t}^+e_h^{u^n}\|_{-1,h}^2+\epsilon\Big(\|e_h^{\bm q^n}\|^2_{ \mathcal T_h} +\|h_K^{-1/2}( \Pi_k^{\partial}e_h^{u^n}-e_h^{\widehat u^n}    )\|^2_{\partial\mathcal{T}_h}\Big) \nonumber \\
&=(\phi^n-\phi^n_{Ih},e_h^{u^n})_{\mathcal{T}_h}  +(\partial_{t}^+ u_{Ih}^n-\partial_tu^n,\Pi_We_h^{u^n})_{\mathcal{T}_h}+\epsilon^{-1}\big(f^n(u_h^n)-f(u^n),e_h^{u^n}\big)_{\mathcal{T}_h}.  
\end{align}	

The three terms on the right-hand side of Eq. \eqref{err-neg-eq2} are estimated as follows. By \Cref{Cauthy} and \Cref{es_u} one has
\begin{align}\label{err-es-r1}
&|(\phi^n-\phi^n_{Ih},e_h^{u^n})_{\mathcal{T}_h}|\nonumber \\ &=|(\Pi_{k+1}^o\phi^n-\phi^n_{Ih},e_h^{u^n})_{\mathcal{T}_h}| \nonumber \\
&\leq C \|\Pi_{k+1}^o\phi^n-\phi^n_{Ih}\|_{-1,h} (\|e_h^{\bm q^n}\|_{ \mathcal T_h} +\|h_K^{-1/2}( \Pi_k^{\partial}e_h^{u^n}-e_h^{\widehat u^n}    )\|_{\partial\mathcal{T}_h}) \nonumber \\
&\leq C h^{2(k+3)} +\theta (\|e_h^{\bm q^n}\|_{ \mathcal T_h}^2 +\|h_K^{-1/2}( \Pi_k^{\partial}e_h^{u^n}-e_h^{\widehat u^n}    )\|_{\partial\mathcal{T}_h}^2),
\end{align}
with $\theta$ an arbitrary positive constant.
Likewise, \Cref{Cauthy} and \Cref{lem-nega} implies
\begin{align}\label{err-es-r2}
|(\partial_{t}^+ u_{Ih}^n-\partial_tu^n,\Pi_We_h^{u^n})_{\mathcal{T}_h}| & \leq  \|\partial_{t}^+ u_{Ih}^n-\Pi_{k+1}^o \partial_tu^n\|_{-1,h} \|e_h^{u^n}\|_{-1,h} \nonumber \\
& \leq C (h^{2(k+3)} +\Delta t^2)+ \kappa \|e_h^{u^n}\|_{-1,h}^2,
\end{align}
where $\kappa$ is another free parameter.
For the nonlinear term one has
\begin{align*}
f^n(u_h^n)-f(u^n)&=(u_h^n)^3-(u_{Ih}^n)^3+ (u_{Ih}^n)^3-(u^n)^3-\Delta t \partial_t^+u_h^n-(u_h^{n-1}-u^{n-1})\\
&\quad +\Delta t \partial_t^+u^n.
\end{align*}
It follows
\begin{align}\label{err-es-r3}
&\epsilon^{-1}\big(f^n(u_h^n)-f(u^n),e_h^{u^n}\big)_{\mathcal{T}_h}\nonumber \\
&= \epsilon^{-1} \left\{-\big((e_h^{u^n})^2, \xi^n_h\big)_{\mathcal{T}_h}+ \big((u_{Ih}^n)^3-(u^n)^3, e_h^{u^n}\big)_{\mathcal{T}_h}-\Delta t (\partial_t^+u_h^n, e_h^{u^n})_{\mathcal{T}_h}
\right.\nonumber \\
&\left. \quad  -(u_h^{n-1}-u^{n-1}, e_h^{u^n})_{\mathcal{T}_h}+ \Delta t (\partial_t^+u^n, e_h^{u^n})_{\mathcal{T}_h}\right\} \nonumber \\
&\leq -\epsilon^{-1}\big((e_h^{u^n})^2, \xi^n_h\big)_{\mathcal{T}_h}+C\|e_h^{u^{n-1}}\|_{-1,h}(\|e_h^{\bm q^n}\|_{ \mathcal T_h} +\|h_K^{-1/2}( \Pi_k^{\partial}e_h^{u^n}-e_h^{\widehat u^n}    )\|_{\partial\mathcal{T}_h}),\nonumber \\
&+C \Big(\|(u_{Ih}^n)^3-(u^n)^3\|_{(H^1)^\prime}+\|u_{Ih}^{n-1}-\Pi_{k+1}^ou^{n-1}\|_{-1,h}+ \Delta t (\|\partial_t^+u^n\|_{\mathcal{T}_h}+\|\partial_t^+u^n_h\|_{\mathcal{T}_h})\Big)\|\nabla e_h^{{u}^n}\|_{\mathcal{T}_h}, \nonumber \\
&\leq -\epsilon^{-1}\big((e_h^{u^n})^2, \xi^n_h\big)_{\mathcal{T}_h}+C (h^{2(k+3)}+(\Delta t)^2)+ \theta (\|e_h^{\bm q^n}\|_{ \mathcal T_h}^2 +\|h_K^{-1/2}( \Pi_k^{\partial}e_h^{u^n}-e_h^{\widehat u^n}    )\|_{\partial\mathcal{T}_h}^2)  \nonumber \\
&+C\|e_h^{u^{n-1}}\|_{-1,h}^2,
\end{align}
where $\xi_h^n:=(u_h^n)^2+u_h^nu_{Ih}^n+(u_{Ih}^n)^2 \geq 0$, and one has utilized the element-wise duality of $H^1$, the uniform bound in \Cref{uih_infty_bounded}, and the stability bounds in \Cref{phi_infty}.

Taking $\theta=\frac{\epsilon}{4}$ and substituting the inequalities \eqref{err-es-r1}--\eqref{err-es-r3} back into \eqref{err-neg-eq2}, then multiplying  the resulting equation by $\Delta t$ and taking summation from $n=1$ to $n=m$ gives
\begin{align}\label{err-neg-final}
&\frac{1}{2}\|e_h^{u^m}\|_{-1,h}^2+\frac{\Delta t}{2} \sum_{n=1}^m\|\partial_{t}^+e_h^{u^n}\|_{-1,h}^2+\frac{\epsilon}{2}\sum_{n=1}^m\Big(\|e_h^{\bm q^n}\|^2_{ \mathcal T_h} +\|h_K^{-1/2}( \Pi_k^{\partial}e_h^{u^n}-e_h^{\widehat u^n}    )\|^2_{\partial\mathcal{T}_h}\Big) \nonumber \\
&\leq C (h^{2(k+3)}+(\Delta t)^2)+ C\Delta t\sum_{n=1}^m \|e_h^{u^{n-1}}\|_{-1,h}^2+ \kappa \Delta t \|e_h^{u^{m}}\|_{-1,h}^2.
\end{align}	
Since $\kappa$ is an arbitrary positive number, one can choose the maximum of $\kappa$ such that $\kappa \Delta t \leq \frac{1}{4} $. An application of Gronwall's inequality then gives the error estimate \eqref{Err-neg1}. This completes the proof.
\end{proof}

Next we derive the error estimates of the scalar variables in the $L^2$ norm.
\begin{lemma}\label{lem-err-sca} The scalar variables satisfy the following error bounds
	\begin{align}\label{lem-err-sca1}
	\max_{1\le n\le N}\|e_h^{u^n}\|^2_{\mathcal{T}_h}+\Delta t\sum_{n=1}^{N}\|e_h^{\phi^n}\|^2_{\mathcal{T}_h}+
	 \le C((\Delta t)^2+ h^{2(k+2)}).
	\end{align}
	
\end{lemma}
\begin{proof} 
	First, we take $(\bm r_1, w_1, \mu_1)=(-e_h^{\bm q^n},e_h^{u^n},e_h^{\widehat{u}^n})$ in \eqref{error_01} to get
	\begin{align}\label{RR2}
	\begin{split}
	\hspace{1em}&\hspace{-1em}(\partial_{t}^+ e_h^{u^n},e_h^{u^n})_{\mathcal T_h}
	+\mathcal{A}(e_h^{\bm p^n},e_h^{\phi^n},e_h^{\widehat{\phi}^n};-e_h^{\bm q^n},e_h^{u^n},e_h^{\widehat u^n})\\
	&=(\partial_{t}^+ u_{Ih}^n-\partial_tu^n,e_h^{u^n})_{\mathcal{T}_h}.
	\end{split}
	\end{align}

	Next, we take $(\bm r_2,w_2,\mu_2)=(e_h^{\bm p^n},-e_h^{\phi^n},-e_h^{\widehat{\phi}^n})$ in \eqref{error_02} to get
	\begin{align}\label{RR1}
	\begin{split}
	\hspace{1em}&\hspace{-1em}\|e_h^{\phi^n}\|^2_{\mathcal{T}_h}+\epsilon
	\mathcal{A}(e_h^{\bm q^n},e_h^{u^n},e_h^{\widehat u^n};e_h^{\bm p^n},-e_h^{\phi^n},-e_h^{\widehat{\phi}^n})\\
	&=-(\phi^n - \phi_{Ih}^n,e_h^{\phi^n})_{\mathcal{T}_h}
	-\epsilon^{-1}(f^n(u_h^n)-f(u^n),e_h^{\phi^n})_{\mathcal{T}_h}.
	\end{split}
	\end{align}
	
	We multiply \eqref{RR2} by $\epsilon$ and then add \eqref{RR1}  to get
	\begin{align} \label{RR0}
	\begin{split}
	 &\epsilon (\partial_{t}^+e_h^{u^n},e_h^{u^n})
	+\|e_h^{\phi^n}\|^2_{\mathcal{T}_h}
	=	-(\phi^n - \phi_{Ih}^n,e_h^{\phi^n})_{\mathcal{T}_h} \\
	&-\epsilon^{-1}(f^n(u_h^n)-f(u^n),e_h^{\phi^n})_{\mathcal{T}_h}
 +\epsilon (\partial_{t}^+u_{Ih}^n-\partial_tu^n,e_h^{u^n})_{\mathcal{T}_h}.
	\end{split}
	\end{align}
	By the $\L^\infty$ stability bound of $u_h^n$ in \Cref{uh_infty_bounded},  we have 
	\begin{align}\label{uhinftybounded}
	|f^n(u_h^n)-f(u^n)|\le |(u_h^n)^3-(u^n)^3|+|u_h^n-u^n|\le C|u_h^n-u^n|.
	\end{align}
	By \Cref{identity} it follows that 
	\begin{align*}
	(\partial_{t}^+e_h^{u^n},e_h^{u^n})_{\mathcal T_h} =\frac{\|e_h^{u^n}\|^2_{\mathcal{T}_h}
		-\|e_h^{u^{n-1}}\|^2_{\mathcal{T}_h}
		+(\Delta t)^2\|\partial_{t}^+ e_h^{u^n}\|^2_{\mathcal{T}_h}}{2\Delta t}.
	\end{align*}
	Apply Cauchy-Schwartz inequality to \eqref{RR0}, and then
	add the resulting equation from $n=1$ to $n=m$ to get
	\begin{align*}
	& \epsilon \|e_h^{u^{m}}\|^2_{\mathcal{T}_h}+2 \Delta t\sum_{n=1}^{m}\|e_h^{\phi^{n}}\|^2_{\mathcal{T}_h}
	+\epsilon (\Delta t)^2\sum_{n=1}^{m}\|\partial_{t}^+ e_h^{u^{n}}\|^2_{\mathcal{T}_h}\\
	&\le C\Delta t\sum_{n=1}^{m} (\|\phi^n - \phi_{Ih}^n \|_{\mathcal{T}_h}^2+\|\partial_{t}^+ u_{Ih}^n-\partial_t u^n \|_{\mathcal{T}_h}^2+\|e_h^{u^{n}}\|^2_{\mathcal{T}_h}+\epsilon^{-2}\|u^n-u_h^n\|_{\mathcal{T}_h}^2)\\
	&\le C \big(h^{2(k+2)}+(\Delta t)^2\big)+C\Delta t\sum_{n=1}^{m} \epsilon^{-2} \|e_h^{u^{n}}\|^2_{\mathcal{T}_h},
	\end{align*}
	where one has applied the approximation properties of the elliptic projection in \Cref{projection_approximation}. The error estimate in the negative norm \eqref{Err-neg1}
implies that 
\begin{align*}
\Delta t \sum_{n=1}^N \|e_h^{u^n}\|_{\mathcal{T}_h}^2 \leq C\big((\Delta t)^2+h^{2(k+2)}\big).
\end{align*}	
Hence one has
	\begin{align}\label{inyter_mide}
 \|e_h^{u^{m}}\|^2_{\mathcal{T}_h}+2 \Delta t\sum_{n=1}^{m}\|e_h^{\phi^{n}}\|^2_{\mathcal{T}_h}
	+ (\Delta t)^2\sum_{n=1}^{m}\|\partial_{t}^+ e_h^{u^{n}}\|^2_{\mathcal{T}_h}
	&\le C(\epsilon)\big((\Delta t)^2+ h^{2(k+2)}\big).
	\end{align}
	This establishes the optimal error estimates of $u$ and $\phi$ in the $L^2$ norm.

\end{proof}

Finally one performs the error analysis of the flux variables. One has
\begin{lemma}\label{lem-err-flux} The scalar variables satisfy the following error bounds
	\begin{align}\label{lem-err-flux1}
	\max_{1\le n\le N}\|e_h^{\bm q^n}\|^2_{\mathcal{T}_h}+\Delta t\sum_{n=1}^{N}\|e_h^{\bm p^n}\|^2_{\mathcal{T}_h}+
	\le C((\Delta t)^2+ h^{2(k+1)}).
	\end{align}
	
\end{lemma}

\begin{proof}
First we take $(\bm r_1, w_1, \mu_1)=(e_h^{\bm p^n}, e_h^{\phi^n},e_h^{\widehat{\phi}^n})$ in Eq. \eqref{error_01} to get
\begin{align}\label{Inte-1}
&(\partial^+_t e_h^{u^n}, e_h^{\phi^n})_{\mathcal{T}_h}+\|e_h^{\bm p^n} \|^2_{\mathcal T_h}+
\|h_K^{-1/2}(\Pi_k^{\partial}e_h^{\phi^n}-e_h^{\widehat{\phi}^n})\|^2_{\partial\mathcal T_h}=(\partial_t^+ u_{Ih}^n-\partial_t u^n, e_h^{\phi^n})_{\mathcal{T}_h}.
\end{align}
Applying $\partial_t^+$ to Eq. \eqref{error_02} and then setting $(\bm r_2, w_2, \mu_2)=(e_h^{\bm q^n}, 0,0)$ in the resulting equation gives
\begin{align}\label{Inte-2}
\epsilon (\partial_t^+ e_h^{\bm q^n}, e_h^{\bm q^n})_{\mathcal{T}_h}-\epsilon (\partial_t^+ e_h^{u^n}, \nabla \cdot e_h^{\bm q^n})_{\mathcal{T}_h}+ \epsilon \langle\partial_t^+ e_h^{\widehat{u}^n},  e_h^{\bm q^n} \cdot n\rangle_{\partial {\mathcal{T}_h}}=0.
\end{align}
One now takes $(\bm r_2, w_2, \mu_2)=(0, \partial_t^+ e_h^{u^n}, \partial_t^+ e_h^{\widehat{u}^n})$ in Eq. \eqref{error_02}, and take summation of the result with Eqs. \eqref{Inte-1}--\eqref{Inte-2} to obtain
\begin{align}\label{Final-1}
&\epsilon (\partial_t^+ e_h^{\bm q^n}, e_h^{\bm q^n})_{\mathcal{T}_h}+\|e_h^{\bm p^n}\|_{\mathcal{T}_h}^2+\|h_K^{-1/2}(\Pi_k^{\partial}e_h^{\phi^n}-e_h^{\widehat{\phi}^n})\|^2_{\partial\mathcal T_h} \nonumber\\
&+\epsilon\langle h_K^{-1}(\Pi_k^\partial e_h^{u^n}-e_h^{\widehat{u}^n}), \partial_t^+(\Pi_k^\partial e_h^{u^n}-e_h^{\widehat{u}^n})\rangle_{\partial\mathcal T_h}\nonumber \\
&=\big(\partial_t^+ u_{Ih}^n-\partial_t u^n, e_h^{\phi^n}\big)_{\mathcal{T}_h} 
+\big(\phi^n-\phi_{Ih}^n, \partial_t^+ e_h^{u^n}\big)_{\mathcal{T}_h}+\epsilon^{-1}\big(f^n(u_h^n)-f(u^n), \partial_t^+ e_h^{u^n}\big)_{\mathcal{T}_h}\nonumber\\
&=\big(\partial_t^+ u_{Ih}^n-\partial_t u^n, e_h^{\phi^n}\big)_{\mathcal{T}_h} 
+\big(\phi^n-\phi_{Ih}^n, \partial_t^+ e_h^{u^n}\big)_{\mathcal{T}_h}+\epsilon^{-1}\big(f^n(u_{Ih}^n)-f(u^n), \partial_t^+ e_h^{u^n}\big)_{\mathcal{T}_h}\nonumber\\
&+\epsilon^{-1}\big(e_h^{u^n}, \partial_t^+ e_h^{u^n}\big)_{\mathcal{T}_h}+\epsilon^{-1}\big((u_h^n)^3-(u_{Ih}^n)^3, \partial_t^+ e_h^{u^n}\big)_{\mathcal{T}_h} \nonumber\\
&:=I_1+I_2+I_3+I_4+I_5.
\end{align}

For $I_1$, by the Cauchy-Schwartz inequality, the HDG Sobolev inequality \Cref{discrete-soblev-hdg}, \Cref{es_u} and the approximation properties in \Cref{projection_approximation}, one obtains
\begin{align}\label{Inte-I1}
\begin{split}
|I_1| &\leq \|\partial_t^+ u_{Ih}^n-\partial_t u^n\|_{\mathcal{T}_h}\| e_h^{\phi^n}\|_{\mathcal{T}_h} \\
&\leq C \|\partial_t^+ u_{Ih}^n-\partial_t u^n\|_{\mathcal{T}_h} (\|\nabla e_h^{\phi^n}\|_{\mathcal{T}_h}+\|h_K^{-1/2}(\Pi_k^{\partial}e_h^{\phi^n}-e_h^{\widehat{\phi}^n})\|_{\partial\mathcal T_h}) \\
&\leq C \|\partial_t^+ u_{Ih}^n-\partial_t u^n\|_{\mathcal{T}_h} (\| e_h^{\bm p^n}\|_{\mathcal{T}_h}+\|h_K^{-1/2}(\Pi_k^{\partial}e_h^{\phi^n}-e_h^{\widehat{\phi}^n})\|_{\partial\mathcal T_h}) \\
&\leq C (h^{2(k+2)}+\Delta t^2)+ \theta \big(\| e_h^{\bm p^n}\|_{\mathcal{T}_h}^2+\|h_K^{-1/2}(\Pi_k^{\partial}e_h^{\phi^n}-e_h^{\widehat{\phi}^n})\|_{\partial\mathcal T_h}\big),
\end{split}
\end{align}
with $\theta$ a positive number to be chosen later.

Denote by $P_W$ the $L^2$ projection operator onto $W_h$. By Eq. \eqref{error_01} with $(\bm r_1, \mu_1)=(0,0)$, one has
\begin{align}\label{Inte-I2}
|I_2|&=|(P_W \phi^n-\phi_{Ih}^n, \partial_t^+ e_h^{u^n})_{\mathcal{T}_h}| \nonumber \\
&\leq |\mathcal{A}\big(e_h^{\bm p^n}, e_h^{\phi^n}, e_h^{\widehat{\phi}^n}; 0, (P_W \phi^n-\phi_{Ih}^n), 0\big)|+|\big(\partial_t^+ u_{Ih}^n-\partial_t u^n, (P_W \phi^n-\phi_{Ih}^n)\big)_{\mathcal{T}_h}| \nonumber\\
&\leq \|\partial_t^+ u_{Ih}^n-\partial_t u^n\|_{\mathcal{T}_h}\|P_W \phi^n-\phi_{Ih}^n\|_{\mathcal{T}_h}+C(\| e_h^{\bm p^n}\|_{\mathcal{T}_h}+\|h_K^{-1/2}(\Pi_k^{\partial}e_h^{\phi^n}-e_h^{\widehat{\phi}^n})\|_{\partial\mathcal T_h})\nonumber\\
& \quad \times(\|\nabla (P_W \phi^n-\phi_{Ih}^n)\|_{\mathcal{T}_h}+\|h_K^{-1/2} \Pi_k^\partial( P_W \phi^n-\phi_{Ih}^n)\|_{\partial{\mathcal{T}}_h})\nonumber\\
&\leq  C (h^{2(k+1)}+\Delta t^2)+\theta \big(\| e_h^{\bm p^n}\|_{\mathcal{T}_h}^2+\|h_K^{-1/2}(\Pi_k^{\partial}e_h^{\phi^n}-e_h^{\widehat{\phi}^n})\|_{\partial\mathcal T_h}\big),
\end{align}
where one has utilizes the continuity of the operator $\mathcal{A}$ in \Cref{A_bound}, the inverse inequality, the $L^2$ stability of the projections $\Pi_k^\partial$ and $P_W$.

The term $I_3$ is estimated similarly as the term $I_2$ as follows.
\begin{align}\label{Inte-I3}
|I_3|&\leq C(h^{2(k+2)}+\Delta t^2)+ \theta \big(\| e_h^{\bm p^n}\|_{\mathcal{T}_h}^2+\|h_K^{-1/2}(\Pi_k^{\partial}e_h^{\phi^n}-e_h^{\widehat{\phi}^n})\|_{\partial\mathcal T_h}\big)\nonumber\\
&\quad+C(\|\nabla P_W \big(f^n(u_{Ih}^n)-f(u^n)\big) \|_{\mathcal{T}_h}+\|h_K^{-1/2} \Pi_k^\partial P_W \big(f^n(u_{Ih}^n)-f(u^n)\big)\|_{\partial{\mathcal{T}}_h})\nonumber \\
&\leq C(h^{2(k+1)}+\Delta t^2)+ \theta \big(\| e_h^{\bm p^n}\|_{\mathcal{T}_h}^2+\|h_K^{-1/2}(\Pi_k^{\partial}e_h^{\phi^n}-e_h^{\widehat{\phi}^n})\|_{\partial\mathcal T_h}\big),
\end{align}
where the uniform bound of $u_{Ih}^n$ in \Cref{uih_infty_bounded} has been applied here.

For $I_4$ one has
\begin{align}\label{Inte_I4}
I_4=\frac{1}{2}\partial_t^+ \|e_h^{u^n}\|_{\mathcal{T}_h}^2+\frac{\Delta t}{2}\|\partial_t^+e_h^{u^n}\|_{\mathcal{T}_h}^2.
\end{align}

We estimate $I_5$ following the approach in \cite{LCWW2017}.  Choosing $\theta=\frac{1}{6}$, substituting the inequalities \eqref{Inte-I1}--\eqref{Inte-I3} back to the Eq. \eqref{Final-1}, multiplying the result by $\Delta t$ and then taking summation from $n=1$ to $n=m$ implies
\begin{align}\label{Final-2}
&\|e_h^{\bm q^n}\|_{\mathcal{T}_h}^2+\|h_K^{-1/2}(\Pi_k^\partial e_h^{u^n}-e_h^{\widehat{u}^n})\|+\Delta t\sum_{n=1}^{m}\|e_h^{\bm p^n}\|_{\mathcal{T}_h}^2+\|h_K^{-1/2}(\Pi_k^{\partial}e_h^{\phi^n}-e_h^{\widehat{\phi}^n})\|^2_{\partial\mathcal T_h} \nonumber \\
&\leq C(h^{2(k+1)}+\Delta t^2)+ C \|e_h^{u^n}\|_{\mathcal{T}_h}^2+C \Delta t^2 \sum_{n=1}^m\|\partial_t^+e_h^{u^n}\|_{\mathcal{T}_h}^2
-C\Delta t\sum_{n=1}^m \big((u_h^n)^3-(u_{Ih}^n)^3, \partial_t^+ e_h^{u^n}\big)_{\mathcal{T}_h}\nonumber \\
&\leq C(h^{2(k+1)}+\Delta t^2)-C\Delta t\sum_{n=1}^m \big((u_h^n)^3-(u_{Ih}^n)^3, \partial_t^+ e_h^{u^n}\big)_{\mathcal{T}_h},
\end{align}
where the last step follows from the $L^2$ error estimate in \eqref{inyter_mide}.

By the identity \eqref{identity4} the last term in \eqref{Final-2} can be written as
\begin{align}\label{Inte-f1}
&-C\Delta t\sum_{n=1}^m \big((u_{Ih}^n)^3-(u_h^n)^3, \partial_t^+ e_h^{u^n}\big)_{\mathcal{T}_h}=-C\Delta t\sum_{n=1}^m \big(\xi^n e_h^{u^n}, \partial_t^+ e_h^{u^n}\big)_{\mathcal{T}_h} \nonumber\\
&=\frac{C \Delta t}{2} \sum_{n=1}^m  \big(\partial_t^+\xi^n,  (e_h^{u^{n-1}})^2\big)_{\mathcal{T}_h}-\frac{C \Delta t^2}{2} \sum_{n=1}^m\big(\xi^n \partial_t^+ e_h^{u^n},  \partial_t^+ e_h^{u^n}\big)_{\mathcal{T}_h}\nonumber \\
&\quad -\frac{C \Delta t}{2} \big(\xi^m \partial_t^+ e_h^{u^m},  \partial_t^+ e_h^{u^m}\big)_{\mathcal{T}_h}, 
\end{align}
with $\xi^n=(u_{Ih}^n)^2+u_{Ih}^n u_h^n+(u_h^n)^2$. Noting that $\xi^n \geq 0$, the last two terms in \eqref{Inte-f1} are non-positive. One has
\begin{align*}
\partial_t^+\xi^n&=\partial_t^+u_{Ih}^n(u_{Ih}^n+u_{Ih}^{n-1})+\partial_tu_{Ih}^nu_h^{n-1}+\Delta t (\partial_t^+u_h^n)^2\nonumber\\
&\quad+\partial_t^+u_h^n(u_{Ih}^n+2u_{Ih}^{n-1}-2e_h^{u^{n-1}}).
\end{align*}
Hence by the uniform stability of $u_h^n$ and $u_{Ih}^n$ in \Cref{uh_infty_bounded} and \Cref{uih_infty_bounded}, one obtains
\begin{align}\label{Inte-f2}
&\big(\partial_t^+ \xi^n, (e_h^{u^{n-1}})^2\big)_{\mathcal{T}_h} \nonumber\\
&\leq C \|\partial_t^+u_{Ih}^n\|_{\mathcal{T}_h}\|e_h^{u^{n-1}}\|_{L^4(\Omega)}^2+C \|\partial_t^+u_{h}^n\|_{\mathcal{T}_h}\|e_h^{u^{n-1}}\|_{L^4(\Omega)}^2 \\
&\leq C (1+\|\partial_t^+u_{Ih}^n\|_{\mathcal{T}_h}^2+\|\partial_t^+u_{h}^n\|_{\mathcal{T}_h}^2)(\|\nabla e_h^{u^{n-1}}\|_{\mathcal{T}_h}^2+\|h_K^{-1/2}(\Pi_k^\partial e_h^{u^{n-1}}-e_h^{\widehat{u}^{n-1}})\|_{\partial\mathcal{T}_h}^2) \nonumber\\
&\leq C (1+\|\partial_t^+u_{Ih}^n\|_{\mathcal{T}_h}^2+\|\partial_t^+u_{h}^n\|_{\mathcal{T}_h}^2)(\| e_h^{\bm q^{n-1}}\|_{\mathcal{T}_h}^2+\|h_K^{-1/2}(\Pi_k^\partial e_h^{u^{n-1}}-e_h^{\widehat{u}^{n-1}})\|_{\partial\mathcal{T}_h}^2), \nonumber
\end{align}
where one has applied the HDG Sobolev inequality \eqref{sobolev-02} and \Cref{es_u} in deriving the last inequality.

With  \eqref{Inte-f1} and \eqref{Inte-f2}, the inequality \eqref{Final-2} becomes
\begin{align}\label{Final-3}
&\|e_h^{\bm q^n}\|_{\mathcal{T}_h}^2+\|h_K^{-1/2}(\Pi_k^\partial e_h^{u^n}-e_h^{\widehat{u}^n})\|+\Delta t\sum_{n=1}^{m}\|e_h^{\bm p^n}\|_{\mathcal{T}_h}^2+\|h_K^{-1/2}(\Pi_k^{\partial}e_h^{\phi^n}-e_h^{\widehat{\phi}^n})\|^2_{\partial\mathcal T_h} \nonumber \\
&\leq C(h^{2(k+1)}+\Delta t^2) 
+C\Delta t \sum_{n=1}^m\big(1+\|\partial_t^+u_{Ih}^n\|_{\mathcal{T}_h}^2+\|\partial_t^+u_{h}^n\|_{\mathcal{T}_h}^2\big)\nonumber\\
&\quad \times \big(\| e_h^{\bm q^{n-1}}\|_{\mathcal{T}_h}^2+\|h_K^{-1/2}(\Pi_k^\partial e_h^{u^{n-1}}-e_h^{\widehat{u}^{n-1}})\|_{\partial\mathcal{T}_h}^2\big).
\end{align}
By \Cref{phi_infty} and \Cref{projection_approximation}, one has
\begin{align*}
\Delta t \sum_{n=1}^m\big(1+\|\partial_t^+u_{Ih}^n\|_{\mathcal{T}_h}^2+\|\partial_t^+u_{h}^n\|_{\mathcal{T}_h}^2\big) \leq C.
\end{align*}
An application of the Gronwall's inequality gives
\begin{align*}
&\|e_h^{\bm q^n}\|_{\mathcal{T}_h}^2+\|h_K^{-1/2}(\Pi_k^\partial e_h^{u^n}-e_h^{\widehat{u}^n})\|+\Delta t\sum_{n=1}^{m}\|e_h^{\bm p^n}\|_{\mathcal{T}_h}^2+\|h_K^{-1/2}(\Pi_k^{\partial}e_h^{\phi^n}-e_h^{\widehat{\phi}^n})\|^2_{\partial\mathcal T_h} \nonumber \\
& \leq C(h^{2(k+1)}+\Delta t^2).
\end{align*}
This completes the proof.
\end{proof}

The combination of \eqref{projection_approximation_1}, \eqref{projection_approximation_4}, \eqref{projection_approximation_5} and the triangle inequality  finishes the proof of \Cref{main_res}.

\section{Numerical results}
\label{numerics}

We consider two examples on unit square domains in $ \mathbb{R}^2 $.  In the first example we have an explicit solution of the system \eqref{ori}; in the second example an explicit form for the exact solution is not known.

\begin{example}\label{example1}
	The problem data $ u^0$ and the artificial $ f $ are chosen so that the exact solution of the  system \eqref{ori} is given by
	\begin{gather*}
		\varepsilon = 1, \quad u = \phi =  e^{-t}x^2y^2(1-x)^2(1-y)^2.
	\end{gather*}
	We report the errors at the final time $T = 1$ for polynomial degrees $k = 0$ and $k = 1$ in \Cref{table_1,table_2} for  the fully implicit scheme and \Cref{table_3,table_4}
	for the energy-splitting scheme. The observed convergence rates match the theory, where $\Delta t = h^{k+1}$.

	\begin{table}
		\begin{center}
			\begin{tabular}{cccccc}
				\hline
				${h}/{\sqrt 2}$ &1/4& 1/8&1/16 &1/32 & 1/64 \\
				\hline
				$\|\bm q  - \bm q_h\|_{\mathcal T_h}$&8.6745E-04   &4.8767E-04   &2.5058E-04   &1.2614E-04   &6.3177E-05\\
				order&-& 0.83088   &0.96061   &0.99025   &0.99757 \\
				\hline
				$\|\bm p-\bm p_h\|_{\mathcal T_h}$&8.9032E-04   &4.9104E-04   &2.5102E-04   &1.2620E-04   &6.3184E-05   
				\\
				order&-&  0.85847   &0.96803   &0.99214   &0.99804 \\
				\hline
				$\| u -  u_h\|_{\mathcal T_h}$ &2.5400E-04   &6.6287E-05   &1.6746E-05   &4.1975E-06   &1.0501E-06   \\
				order&-& 1.9380   &1.9849  & 1.9962   &1.9990 \\
				\hline
				$\|\phi -  \phi_h\|_{\mathcal T_h}$&2.6147E-04   &6.7626E-05   &1.7040E-05   &4.2683E-06   &1.0676E-06 \\
				order&-& 1.9510   &1.9887   &1.9972   &1.9993\\
				\hline
			\end{tabular}
		\end{center}
		\caption{\Cref{example1}, $k=0$ with fully implicit scheme: Errors, observed convergence orders for  $u$, $\phi$ and their fluxes $\bm q$ and $\bm p$.}\label{table_1}
	\end{table}

	\begin{table}
	\begin{center}
		\begin{tabular}{cccccc}
			\hline
			${h}/{\sqrt 2}$ &1/4& 1/8&1/16 &1/32 & 1/64 \\
			\hline
			$\|\bm q  - \bm q_h\|_{\mathcal T_h}$&1.6623E-04  &4.5233E-05   &1.1599E-05   &2.9202E-06   &7.3141E-07\\
			order&-&  1.8778  &1.9634   &1.9899   &1.9973 \\
			\hline
			$\|\bm p-\bm p_h\|_{\mathcal T_h}$&1.6700E-04   &4.5276E-05   &1.1602E-05   &2.9204E-06   &7.3142E-07   
			\\
			order&-&  1.8830   &1.9644   &1.9901  &1.9974 \\
			\hline
			$\| u -  u_h\|_{\mathcal T_h}$ & 4.8698E-05   &6.1714E-06   &7.7349E-07   &9.6742E-08   &1.2094E-08  \\
			order&-& 2.9802   &2.9962   &2.9992   &2.9998\\
			\hline
			$\|\phi -  \phi_h\|_{\mathcal T_h}$& 4.9152E-05   &6.1862E-06   &7.7391E-07   &9.6753E-08   &1.2095E-08 \\
			order&-& 2.9901  &2.9988   &2.9998   &3.0000\\
			\hline
		\end{tabular}
	\end{center}
	\caption{\Cref{example1}, $k=1$ with fully implicit scheme: Errors, observed convergence orders for  $u$, $\phi$ and their fluxes $\bm q$ and $\bm p$.}\label{table_2}
\end{table}
	\begin{table}
	\begin{center}
		\begin{tabular}{cccccc}
			\hline
			${h}/{\sqrt 2}$ &1/4& 1/8&1/16 &1/32 & 1/64 \\
			\hline
			$\|\bm q  - \bm q_h\|_{\mathcal T_h}$& 8.6761E-04   &4.8768E-04  & 2.5059E-04   &1.2614E-04   &6.3177E-05\\
			order&-& 0.83111   &0.96063   &0.99025  &0.99757\\
			\hline
			$\|\bm p-\bm p_h\|_{\mathcal T_h}$&8.9460E-04   &4.9143E-04   &2.5107E-04   &1.2620E-04   &6.3185E-05   
			\\
			order&-&   0.86427   &0.96891   &0.99232   &0.99809 \\
			\hline
			$\| u -  u_h\|_{\mathcal T_h}$ & 2.5759E-04   &6.7122E-05   &1.6952E-05   &4.2490E-06   &1.0629E-06  \\
			order&-& 1.9402   &1.9853   &1.9963   &1.9991 \\
			\hline
			$\|\phi -  \phi_h\|_{\mathcal T_h}$&  2.6295E-04   &6.7806E-05   &1.7076E-05   &4.2768E-06   &1.0697E-06 \\
			order&-&1.9553   &1.9894   &1.9974   &1.9993\\
			\hline
		\end{tabular}
	\end{center}
		\caption{\Cref{example1}, $k=0$ with energy-splitting scheme: Errors, observed convergence orders for  $u$, $\phi$ and their fluxes $\bm q$ and $\bm p$.}\label{table_3}
\end{table}

	\begin{table}
	\begin{center}
		\begin{tabular}{cccccc}
			\hline
			${h}/{\sqrt 2}$ &1/4& 1/8&1/16 &1/32 & 1/64 \\
			\hline
			$\|\bm q  - \bm q_h\|_{\mathcal T_h}$& 1.5809E-04   &4.3945E-05   &1.1415E-05   &2.8955E-06   &7.2935E-07\\
			order&-&1.8470   &1.9448   &1.9790   &1.9891\\
			\hline
			$\|\bm p-\bm p_h\|_{\mathcal T_h}$& 1.5896E-04   &4.3991E-05   &1.1418E-05   &2.8957E-06   &7.2940E-07 
			\\
			order&-&   1.8534  &1.9459   &1.9793   &1.9891 \\
			\hline
			$\| u -  u_h\|_{\mathcal T_h}$ &  4.9741E-05   &6.3026E-06   &7.9008E-07  & 9.8850E-08   &1.2358E-08 \\
			order&-& 2.9804   &2.9959   &2.9987   &2.9998 \\
			\hline
			$\|\phi -  \phi_h\|_{\mathcal T_h}$&  4.9111E-05   &6.1809E-06   &7.7336E-07   &9.6709E-08   &1.2090E-08 \\
			order&-&2.9902   &2.9986   &2.9994   &2.9998\\
			\hline
		\end{tabular}
	\end{center}
\caption{\Cref{example1}, $k=1$ with energy-splitting scheme: Errors, observed convergence orders for  $u$, $\phi$ and their fluxes $\bm q$ and $\bm p$.}\label{table_4}
\end{table}

\end{example}

\section{Appendix}\label{app}

\begin{definition}
	
	\begin{subequations}\label{def-Ih}
		\color{black}
		For every $K\in\mathcal T_h$, we define $\mathcal I_K^{\color{black}k+1+d}(u_h,\widehat u_h)\in \mathcal {P}^{k+1+d}(K)$ as follows. \color{black}
		
		\textbf{(1)}	For every vertex $A_i$ on mesh $\mathcal{T}_h$, 
		let $N_{A_i}$ be the number of elements adjoint at $A_i$, and $\mathcal{K}_{A_i}$ denote all these elements, then $\mathcal I_K^{\color{black}k+1+d}$ at $A_i$ is defined as
		\begin{align}
		\mathcal{I}_K^{\color{black}k+1+d}(u_h,\widehat u_h)(A_i)=\frac{1}{N_{A_i}}\sum_{K\in\mathcal K_{A_i}}\widehat u_h(A_i)|_{\color{black}\partial K},\label{Id}
		\end{align}
		we note that $N_{A_i}$ is a \color{black} fixed finite number \color{black} since $\mathcal T_h$ is shape-regular.
		
		\textbf{(2)} If, in addition, for  $d=3$,  for every edge $E$ of element $K$, there are $k+d$ interior Lagrange points on edge $E$, for any of these points $B_i$,
		let $N_{B_i}$ be the number of elements adjoint at $B_i$, and $\mathcal{K}_{B_i}$ denote all these elements, then 
		$\mathcal I_K^{\color{black}k+1+d}$ at $B_i$ is defined as
		\begin{align}
		\mathcal{I}_K^{\color{black}k+1+d}(u_h,\widehat u_h)(B_i)=\frac{1}{N_{B_i}}\sum_{K\in\mathcal K_{B_i}}\widehat u_h(B_i)|_{\color{black}\partial K}.\label{Ic}
		\end{align}
		Again, $N_{B_i}$ is finite since $\mathcal T_h$ is shape-regular.
		
		\textbf{(3)} Since $\mathcal I_K^{\color{black}k+1+d}(u_h,\widehat u_h)\in \mathcal {P}^{k+1+d}(K)$, there are $\binom{k+d}{d-1}$ Lagrange points on every face $F$ of $K$, the value of $\mathcal I_K^{\color{black}k+1+d}(u_h,\widehat u_h)$ on these points are determined by	
		\begin{align}
		\langle 
		\mathcal{I}_K^{\color{black}k+1+d}(u_h,\widehat u_h),\widehat v_h
		\rangle_{F}&=\langle 
		\widehat u_h,\widehat v_h
		\rangle_F&\text{for all }\widehat v_h\in \mathcal{P}^{k+1}(F),\label{Ib}
		\end{align}
		holds for all face $F$ of $K$. 
		
		\textbf{(4)} Since $\mathcal I_K^{\color{black}k+1+d}(u_h,\widehat u_h)\in \mathcal {P}^{k+1+d}(K)$, there are $\binom{k+d}{d}$ Lagrange points in every element $K$, the value of $\mathcal I_K^{\color{black}k+1+d}(u_h,\widehat u_h)$ on these points are determined by	
		\begin{align}
		(\mathcal{I}_K^{\color{black}k+1+d}(u_h,\widehat u_h),v_h)_{K}
		&=
		(u_h,v_h)_{K} &\text{for all }v_h\in \mathcal{P}^k(K),\label{Ia}
		\end{align}
	\end{subequations}
\end{definition}

It is easy to check that the degrees of freedom of $\mathcal{P}^{k+1+d}(K)$ is $\binom{k+1+2d}{d}$, the constrains for \eqref{Id}, \eqref{Ic}, \eqref{Ib} and \eqref{Ia} are $d+1$, $(d-1)d\binom{k+d}{1}$,   $(d+1)\binom{k+d}{d-1}$, and $\binom{k+d}{d}$. For $d=3$, there holds
\begin{align*}
\binom{k+1+2d}{d}&=d+1+d(d-1)\binom{k+d}{1}+(d+1)\binom{k+d}{d-1}+\binom{k+d}{d};
\end{align*}
and 
for $d=2$, it holds
\begin{align*}
\binom{k+1+2d}{d}&=d+1+(d+1)\binom{k+d}{d-1}+\binom{k+d}{d}.
\end{align*}
Then the definition of $\mathcal I_K^{\color{black}k+1+d}(u_h,\widehat u_h)$ is a square system, therefore, the uniqueness and the existence of $\mathcal I_K^{\color{black}k+1+d}(u_h,\widehat u_h)$ are equivalence.
In addition, it is obviously that when $u_h=\widehat u_h=0$ we have $\mathcal I_K^{\color{black}k+1+d}(u_h,\widehat u_h)=0$, 
then the operator $\mathcal{I}_K^{\color{black}k+1+d}$ is well-defined. 
We define $\mathcal{I}_h^{\color{black}k+1+d}|_K=\mathcal{I}_K^{\color{black}k+1+d}$,  if for all $(u_h,\widehat u_h)\in W_h\times M_h$, we have $\mathcal{I}_h^{\color{black}k+1+d}(u_h,\widehat u_h)$
is unique defined at every face of $\mathcal{T}_h$ due to \eqref{Id}, \eqref{Ic}, and \eqref{Ib}. Then $\mathcal{I}_h^{\color{black}k+1+d}(u_h,\widehat u_h)\in H^1(\Omega)$.

\begin{lemma}\label{stablity} For all $(u_h,\widehat u_h)\in L^2(\mathcal T_h)\times L^2(\partial\mathcal T_h)$, we have the stability:
	\begin{align*}
	\|\mathcal{I}_h^{\color{black}k+1+d}(u_h,\widehat u_h)\|_{K}\le 
	C
	\left(
	\|u_h\|_{S(K)}
	+\|h_K^{1/2}\widehat u_h\|_{\partial S(K)}
	\right),
	\end{align*}
	where $S(K)$ is the set of all the simplex $K^{\star}\in \mathcal{T}_h$ such that $K^{\star}$ and $K$ has at least one common node, and  $\partial S(K)$ is the set of all the faces of those simplex.
\end{lemma}
\begin{proof}
	To simplify the proof, we only give a proof for $d=3$, the proof of $d=2$ is similar.
	According to \eqref{def-Ih}, we divide the Lagrange points on $K$ of degree $k+1+d$ into 4 parts,
	and the corresponding Lagrange basis denoted as $\{\phi_{1,j}\}_{j=1}^{N_1}$, 
	$\{\phi_{2,j}\}_{j=1}^{N_2}$,  $\{\phi_{3,j}\}_{j=1}^{N_3}$, and  $\{\phi_{4,j}\}_{j=1}^{N_4}$, which are determined by  (1)(2)(3)(4) in \eqref{def-Ih}, respectively.
	It is known that $\phi_{4,j}|_{\partial K}=0$, since the corresponding Lagrange points are inside $K$.
	We also denote the dual basis of $\{\phi_{3,j}\}_{j=1}^{N_3}$ and $\{\phi_{4,j}\}_{j=1}^{N_4}$ as $\{\psi_{3,j}\}_{j=1}^{N_3}$ and $\{\psi_{4,j}\}_{j=1}^{N_4}$, respectively, such that
	\begin{align*}
	\langle \phi_{3,j},\psi_{3,\ell} \rangle_F=\delta_{\color{black}j,\ell},\quad 
	(\phi_{4,j},\psi_{4,\ell} )_K=\delta_{\color{black}j,\ell},
	\end{align*}
	\color{black}
	where $\delta_{j,\ell}$ is the Kronecker delta.
	\color{black}
	A result in \cite[Lemma 3.1]{} show that
	\begin{align}\label{psi-infy}
	\|\psi_{3,j}\|_{0,\infty,F}\le Ch_F^{-(d-1)},\qquad 
	\|\phi_{4,j}\|_{0,\infty,K}\le Ch_K^{-d}.
	\end{align}
	We can write $\mathcal{I}_K^{\color{black}k+1+d}(u_h,\widehat u_h)$ as
	\begin{align}\label{IIHH}
	\mathcal{I}_K^{\color{black}k+1+d}(u_h,\widehat u_h)=
	\sum_{i=1}^{N_1}a_{1,i}\phi_{1,i}
	+\sum_{i=1}^{N_2}a_{2,i}\phi_{2,i}
	+\sum_{i=1}^{N_3}a_{3,i}\phi_{3,i}
	+\sum_{i=1}^{N_4}a_{4,i}\phi_{4,i}, 
	\end{align}
	where 
	\begin{subequations}
		\begin{align}
		a_{1,i}&=\frac{1}{N_{A_i}}\sum_{K\in\mathcal K_{A_i}}\widehat u_h(A_i)|_{\color{black}\partial K},\label{solve-a1}\\
		a_{2,i}&=\frac{1}{N_{B_i}}\sum_{K\in\mathcal K_{B_i}}\widehat u_h(B_i)|_{\color{black}\partial K},\label{solve-a2}\\
		a_{3,j}&=\langle \widehat u_h-\sum_{i=1}^{N_1} a_{1,i}\phi_{1,i}
		-\sum_{i=1}^{N_2} a_{2,i}\phi_{2,i},\psi_{3,j} \rangle_F,\label{solve-a3}\\
		a_{4,j}&=( u_h-\sum_{i=1}^{N_1} a_{1,i}\phi_{1,i}
		-\sum_{i=1}^{N_2} a_{2,i}\phi_{2,i}
		-\sum_{i=1}^{N_3} a_{3,i}\phi_{3,i}
		,\psi_{4,j} )_K,\label{solve-a4}
		\end{align}
		according \color{black} to \color{black} \eqref{def-Ih}.
	\end{subequations}
	By \color{black} a \color{black} scaling argument, one can get
	\begin{subequations}
		\begin{align}\label{scal-phi}
		\|\phi_{i,j}\|_{0,p,K}&\le Ch_K^{d/p}\|\hat{\phi}_{i,p,j}\|_{0,p,\hat{K}}\le Ch_K^{d/p},\\
		\|\phi_{i,j}\|_{0,p,F}&\le Ch_F^{(d-1)/p}\|\hat{\phi}_{i,p,j}\|_{0,p,\hat{F}}\le Ch_F^{(d-1)/p},
		\end{align} 
		for any integer $p\ge 1$.
	\end{subequations}
	Again, by a scaling argument, for the Lagrange point $A_i$ on a face $F\subset\partial K$, and $A_i$ is also the vertex of $\mathcal T_h$, one can get
	\begin{align}\label{a1}
	|a_{1,i}|\le\frac{1}{N_{A_i}}\sum_{A_i\in F}\|\widehat u_h\|_{0,\infty, F}\le C\sum_{ \hat{A_i}\in \hat F}\|\widehat{\widehat u_h}\|_{0,\infty, \hat F}
	\le C\sum_{ A_i\in F}h_F^{-(d-1)/2}\|\widehat u_h\|_{0, F},
	\end{align}
	and the similar for $a_{2,i}$, for the Lagrange point $B_i$ on an edge $E\subset\partial F$, $F\subset \partial K$:
	\begin{align}\label{a2}
	|a_{2,i}|
	\le C\sum_{B_i\in F}h_F^{-(d-1)/2}\|\widehat u_h\|_{0,F}.
	\end{align}
	We use \eqref{solve-a3}, \eqref{a1}, \eqref{a2}, \eqref{psi-infy}, \eqref{scal-phi}, \eqref{a1}, \eqref{a2} and a scaling argument, to get 
	\begin{align}\label{a3}
	|a_{3,j}|&\le \left(\|\widehat u_h\|_{0,1,F}+\sum_{i=1}^{N_1} |a_{1,i}|\cdot\|\phi_{1,i}\|_{0,1,F}
	+\sum_{i=1}^{N_2} |a_{2,i}|\cdot\|\phi_{1,i}\|_{0,1,F}\right)\|\psi_{3,j} \|_{0,\infty,F}\nonumber\\
	&\le C 
	\left( h_F^{-(d-1)/2}\|\widehat u_h\|_{0,F}+\sum_{A_i\in F}h_F^{-(d-1)/2}\|\widehat{u}_h\|_{0,F}
	+\sum_{B_i\in F}h_F^{-(d-1)/2}\|\widehat{u}_h\|_{0,F}.
	\right)
	\end{align}
	We use \eqref{solve-a4}, \eqref{a1}, \eqref{a2}, \eqref{a3}, \eqref{psi-infy}, \eqref{scal-phi}, \eqref{a1}, \eqref{a2} and a scaling argument, to get 
	\begin{align}\label{a4}
	|a_{4,j}|&\le \left(\| u_h\|_{0,1,K}+\sum_{i=1}^{N_1} |a_{1,i}|\cdot\|\phi_{1,i}\|_{0,1,K}
	+\sum_{i=1}^{N_2} |a_{2,i}|\cdot\|\phi_{1,i}\|_{0,1,K}
	+\sum_{i=1}^{N_3} |a_{3,i}|\cdot\|\phi_{1,i}\|_{0,1,K}
	\right)\|\psi_{4,j} \|_{L^\infty,K}\nonumber\\
	&\le C 
	\left( h_F^{-d/2}\|u_h\|_{0,K}+\sum_{A_i\in F}h_F^{-(d-1)/2}\|\widehat{u}_h\|_{0,F}
	+\sum_{B_i\in F}h_F^{-(d-1)/2}\|\widehat{u}_h\|_{0,F}
	\right)
	\end{align}
	
	Then desired result is followed by \eqref{IIHH}, \eqref{a1}, \eqref{a2}, \eqref{a3}, \eqref{a4} and \eqref{scal-phi} with $p=2$.
\end{proof}

Finally, we have the following estimation
\begin{lemma}\label{error_bounds_Ih} 
	For all $(u_h,\widehat u_h)\in W_h\times M_h$, we have 
	\begin{subequations}\label{IC}
		\begin{align}\label{ICK}
		\|\mathcal{I}_h^{\color{black}k+2+d}(u_h,\widehat u_h)-u_h\|_{\mathcal{T}_h}\le C
		\|h_K^{1/2}(u_h-\widehat u_h    )\|_{\partial\mathcal{T}_h},\\
		\|\nabla( \mathcal{I}_h^{\color{black}k+2+d}(u_h,\widehat u_h)-u_h)\|_{\mathcal{T}_h}\le C
		\|h_K^{-1/2}(u_h-\widehat u_h    )\|_{\partial\mathcal{T}_h}.\label{ICK22}
		\end{align}
	\end{subequations}
\end{lemma}

\begin{proof}[Proof of \Cref{error_bounds_Ih}]
	Since $\mathcal I_h^{\color{black}k+2+d}(u_h,u_h)=u_h$ for every $u_h\in V_h$,
	then \eqref{ICK} follows from \Cref{stablity} and the fact that $\mathcal I_h^{\color{black}k+2+d}$ is linear. Then \eqref{ICK22} follows by an application of the inverse inequality. This completes the proof.
\end{proof}

\bibliographystyle{plain}
\bibliography{Cahn_Hilliard,Model_Order_Reduction,Ensemble,HDG,Interpolatory,Mypapers,Added,multiphase-2018}
\end{document}